\documentclass[12pt,a4paper]{amsart}
\usepackage{amsfonts}
\usepackage{amsthm}
\usepackage{amsmath}
\usepackage{xcolor}
\usepackage{amssymb}
\usepackage{mathtools} \usepackage{stmaryrd} \usepackage{amscd}
\usepackage{dsfont} \usepackage[latin2]{inputenc}
\usepackage{t1enc}
\usepackage[mathscr]{eucal}
\usepackage{indentfirst}
\usepackage{graphicx}
\usepackage{graphics}
\usepackage{hyperref}
\usepackage{epic}
\numberwithin{equation}{section}
\usepackage[margin=2.9cm, marginparwidth=2.4cm]{geometry}
\usepackage{epstopdf} 
\usepackage{enumerate}
\usepackage{mathrsfs} \usepackage{upgreek}
\newcommand{\invtemp}{\upbeta}
\newcommand{\tempr}{\mathrm{T}}
\newcommand{\vertiii}[1]{{\left\vert\kern-0.25ex\left\vert\kern-0.25ex\left\vert #1 
    \right\vert\kern-0.25ex\right\vert\kern-0.25ex\right\vert}}

\theoremstyle{plain}
\newtheorem{Th}{Theorem}[section]
\newtheorem{Lemma}[Th]{Lemma}
\newtheorem{Cor}[Th]{Corollary}
\newtheorem{Prop}[Th]{Proposition}

 \theoremstyle{definition}
\newtheorem{Def}[Th]{Definition}

\newtheorem{Rem}[Th]{Remark}
\newtheorem*{Rem*}{Remark}
\newtheorem{?}[Th]{Problem}

\renewcommand{\P}{\mathsf{P}}

\newcommand{\Ind}[1]{\mathds{1}_{\{ #1\}}}

\newcommand{\R}{\mathbb{R}}
\newcommand{\N}{\mathbb{N}}
\newcommand{\Z}{\mathbb{Z}}
\renewcommand{\d}{\mathrm{d}}

\newcommand{\eps}{\epsilon}
\newcommand{\Xn}{X^{(n)}}

\newcommand{\sL}{\mathscr{L}}
\newcommand{\Co}{\mathrm{Co}}
\newcommand{\out}{{\mathrm{out}}}
\newcommand{\Q}{\mathbb{Q}}
\newcommand{\one}{{(1)}}
\newcommand{\two}{{(2)}}
\renewcommand{\aa}{\mathbf{a}}
\newcommand{\xn}{x^{(n)}}
\newcommand{\WW}{\mathbb{W}}
\newcommand{\cW}{\mathcal{W}}
\newcommand{\sW}{\mathsf{W}}
\newcommand{\FF}{\mathbb{F}}
\newcommand{\cF}{\mathcal{F}}
\newcommand{\sF}{\mathsf{F}}
\newcommand{\cG}{\mathcal{G}}
\newcommand{\rr}{\mathsf{r}}
\newcommand{\XX}{\mathbf{X}}

\renewcommand{\b}{\mathsf{b}}
\newcommand{\LL}{\mathbb{L}}
\newcommand{\alphap}{{\alpha,p}}
\newcommand{\alphainfty}{{\alpha,\infty}}
\newcommand{\alphaprime}{{\alpha',p'}}
\newcommand{\cS}{\mathcal{S}}

\newcommand{\m}{{(m)}}
\newcommand{\sM}{\mathscr{M}}

\newcommand{\sS}{\mathsf{S}}
\renewcommand{\Bbb}{\mathbb}

\newcommand{\Wc}{\cW}
\newcommand{\Ws}{\sW}

\newcommand{\Fc}{\cF}
\newcommand{\Fs}{\sF}

\newcommand{\Pc}{\mathcal{P}}

\newcommand{\bpf}[1][Proof]{{\noindent {\sc #1: }}}
\newcommand{\epf}{{{\hfill $\Box$ \smallskip}}}
\newcommand{\ONE}{{\mathbf{1}}}

\newcommand{\Pp}{\mathsf{P}}

\newcommand{\sC}{\mathscr{C}}
\newcommand{\n}{{(n)}}
\newcommand{\cB}{\mathcal{B}}
\newcommand{\Bc}{\cB}
\newcommand{\visc}{\frac{\varkappa}{2}}
\newcommand{\temp}{\varkappa}

\newcommand{\NR}{\N\times\R}

\newcommand{\tshift}{\theta}

\newcommand{\Shear}{\Xi}
\newcommand{\Ac}{{\mathcal A}}
\newcommand{\lsupp}{\mathop{\mathrm{supp}_-}}
\newcommand{\rsupp}{\mathop{\mathrm{supp}_+}}
 \newcommand{\setbetap}{\Pi}

\usepackage{ulem}
\usepackage{cancel}

\begin{document}

\title{Dynamic polymers: invariant measures and ordering by noise}

\author{Yuri Bakhtin}
\author{Hong-Bin Chen}
\address{Courant Institute of Mathematical Sciences\\ New York University \\ 251~Mercer~St, New York, NY 10012, USA }
\email{bakhtin@cims.nyu.edu, hbchen@cims.nyu.edu }

\begin{abstract}
We develop a dynamical approach to infinite volume directed polymer measures in random environments. We define polymer dynamics in $1+1$ dimension as a stochastic gradient flow on polymers pinned at the origin, for energy involving quadratic nearest neighbor interaction and local interaction with random environment. We prove existence and uniqueness of the solution, continuity of the flow, the order-preserving property with respect to the coordinatewise partial order, and the invariance of the asymptotic slope. We establish ordering by noise which means that if two initial conditions have distinct slopes, then the associated solutions eventually get ordered coordinatewise. This, along with the shear-invariance property and existing results on static infinite volume polymer measures, allows to prove that for a fixed asymptotic slope and almost every realization of the environment, the polymer dynamics has a unique invariant distribution given by a unique infinite volume polymer measure, and, moreover, One Force -- One Solution principle holds. We also prove that every  polymer measure is concentrated on paths with well-defined asymptotic slopes and give an estimate on deviations from  straight lines.  
\end{abstract}
\maketitle
\tableofcontents

\section{Introduction}
\subsection{Background.}
The goal of this paper is to initiate the study of infinite volume directed polymer measures in random potential from a dynamical point of view, with the help of infinite-dimensional stochastic gradient flows in random potentials.

The term {\it directed polymers} applies to a variety of mathematical models of  monomer 
chains subject to local self-interactions and interactions with random environment,
see \cite{Giacomin:MR2380992}, \cite{Hollander:MR2504175}, \cite{Comets:MR3444835} and references therein.

For finite chains (often viewed as time-parametrized paths), random polymer measures are usually defined as Gibbs distributions with reference measure
being the distribution of a random walk and with Boltzmann--Gibbs weights given by the potential accumulated by random walk paths from the random environment.  This can be done most naturally when either both endpoints
are fixed (point-to-point polymers) or when only one of them is fixed (point-to-line polymers).

Directed polymers are essential in the study of some basic PDEs with random forcing.
The classical Feynman--Kac  formula for solutions of the Cauchy problem for the linear heat equation  with multiplicative potential can be interpreted in terms of integration with respect to a polymer measure. In turn, one of the basic nonlinear systems, the Burgers equation in dimension $d\in\N$ with viscosity $\visc>0$ and external forcing $f=f(t,x)$, 
\begin{equation*}
\partial_t u + (u \cdot \nabla_x) u=\visc\Delta_x u + f,
\end{equation*}
($\nabla_x$ is the gradient and $\Delta_x$ is the Laplace operator with respect to $x\in\R^d$) considered on gradient vector fields
can be reduced to such a  heat equation by the Hopf--Cole transformation.

In recent papers \cite{Bakhtin-Li:MR3911894}, \cite{Bakhtin-Li:MR3856947}  the polymer approach to the Burgers equation on the line $(d=1)$ with random kick forcing $f=-\partial_x F$ given by a random potential
\begin{equation*}
F(t,x)= F_\omega(t,x)=\sum_{n\in\Z} F_{\omega, n}(x)\delta_{n}(t), 
\end{equation*}
with kicks  $(F_n)$ being i.i.d.\ weakly mixing stationary processes
was instrumental in the study of the ergodic properties of that random dynamical system.

The central part of the program realized in  \cite{Bakhtin-Li:MR3911894} was working with thermodynamic limits of the
polymer measures in the random forcing potential. Denoting the space of polymer chains with asymptotic slope $v\in\R^d$ by $S(v)$,
in the one-dimensional case, it was shown that for each $v$, with probability one there is a unique 
measure on infinite paths concentrated on $S(v)$ and satisfying the Dobrushin--Lanford--Ruelle (DLR) condition.
The latter requires (see~\cite{Sinai:MR691854} 
    or~\cite{Georgii:MR956646}) that the finite-dimensional distributions conditioned on the complementary (infinitely many) coordinates coincide with the point-to-point polymer measures defined as the Boltzmann--Gibbs measures for finitely many degrees of freedom via the usual exponential formula: given 
    the  temperature $\tempr=\temp$, the potential $F:\N\times\R\to\R$, and 
    the endpoints $x_0,x_n\in\R$,
    the Lebesgue density of the point-to-point polymer probability measure on the path $(x_0,x_1,\ldots,x_n)$ is given, up to a normalizing factor, by 
    \begin{align*}
 \exp\Bigg\{-\frac{1}{\tempr} \bigg( \sum_{i=1}^n V(\delta_i x)+\sum_{i=1}^n F_i(x_i)\bigg) \Bigg\},
    \end{align*}
    where $\delta_i x =  x_i - x_{i-1}$
    and $V(r)=|r|^2/2$.
    We give a more detailed technical definition in Section~\ref{section:IVPM}. 

Loosely speaking, these measures can be viewed as Gibbs measures on semi-infinite polymer chains/paths  $x:\{0\}\cup\N\to \R^d$, with 
energy $E(x)$ defined by
\begin{equation}
\label{eq:energy-of-infinite-path}
E(x)=\sum_{i\in\N} V(\delta_i x)+\sum_{i\in\N} F_i(x_i),
\end{equation}
 This series does not converge, but its increments under local perturbations of polymers are well-defined.

Due to the presence of the quadratic term $V(\delta_ix)$ (which in fluid dynamics terms can be interpreted as the kinetic energy of the particle following path $x$, so that
at time $i$ it is located at $x_i$), these infinite volume Gibbs measures can  also be viewed as (tilted)  Discrete Gaussian Free Field in a random potential. They were shown to be limits of a broad class of sequences of finite volume polymer measures. Besides point-to-point, point-to-line polymers, this class includes point-to-distribution ones where we fix one endpoint and the distribution of another point.

Crucially, these thermodynamic limits or infinite volume polymer measures are then directly used in  \cite{Bakhtin-Li:MR3911894} to construct global stationary solutions of the Burgers equation with kick random forcing, prove their uniqueness, and show that they serve as skew-invariant one-point random attractors, thus proving the One Force --- One Solution  principle (1F1S) on each ergodic component composed of functions with fixed average (coinciding with the slope of the polymer). 
The existence-uniqueness and other properties of IVPMs are central in the approach of  \cite{Bakhtin-Li:MR3911894} to the global solutions of the Burgers equation.  It is a powerful approach resulting  in a more detailed and precise description of the basin of attraction and mode of convergence to the global solution given by the 1F1S than in more recent papers \cite{Dunlap-Graham-Ryzhik}, \cite{Dunlap:MR4173562}, where the same circle of problems  was approached with PDE methods at the level of the Markov semigroup.

 In~\cite{Bakhtin-Li:MR3856947},  showing that IVPMs asymptotically concentrate near one-sided infinite minimizers of the random Lagrangian action as the temperature goes to zero, was in the core of the argument that the global stationary solutions of randomly forced Burgers equation 
 with zero viscosity studied in~\cite{kickb:bakhtin2016} and~\cite{BCK:MR3110798}
 are inviscid limits of positive viscosity global solutions, and thus the inviscid limit also holds for the invariant distributions.

Describing thermodynamic limits  --- and thus finding to which extent the local interactions define the macroscopic state of the system --- is a fundamental problem for Gibbs distributions.
There is a growing interest to IVPMs in random environment, especially in connection with the study of Busemann functions playing the role of global solutions of appropriate analogues of the Burgers/KPZ equations.
Besides  \cite{Bakhtin-Li:MR3911894}, \cite{Bakhtin-Li:MR3856947},  see
~\cite{Bakhtin-Khanin:MR2791052}, \cite{Georgiou--Rassoul-Agha--Seppalainen--Yilmaz:MR3395462}, \cite{Rassoul-Agha--Seppalainen--Yilmaz:2013arXiv1311.3016G}, \cite{Alberts--Rassoul-Agha--Simper:MR4091097},
\cite{Janjigian--Rassoul-Agha:MR4089495},
\cite{Janjigian--Rassoul-Agha:MR4099996} studying lattice polymer models.

In~\cite{Bakhtin-Khanin-non:MR3816628} a new notion of generalized polymer based on stochastic control is introduced in the hope that it will be useful in studying a broader class of random systems from the KPZ universality class. 

\medskip

For general kinetic energies or local self-interaction potentials $V$ and in higher dimensions the problem of existence and uniqueness of IVPMs  in $S(v)$ for a fixed $v$ (i.e., with given
asymptotic slope) is open and so is the problem of existence and uniqueness of infinite one-sided geodesics (minimizers or ground states of 
the energy~$E$) in $S(v)$. 
 In this paper, we develop a dynamical point of view that can be useful in addressing these issues.

\bigskip

A fundamental feature of a finite volume Gibbs measure
known at least since the analysis of the Fokker---Planck equations in \cite{Kolmogorov:MR1513121}
is that it is a unique invariant measure of a stochastic gradient flow with potential given by the energy function.
To remind the reader what this means,  suppose that
 $E:\R^{Nd}\to\R$ is a smooth energy function of $Nd$ coordinates of $N$ particles described by $d$ coordinates each and $\tempr>0$ is the temperature  parameter such that for the inverse temperature $\invtemp=\tempr^{-1}$,
\[
Z=\int_{\R^{Nd}} e^{-\invtemp E(x)}\d x <\infty.
\]
Then the Gibbs measure
\begin{equation*}
\mu(\d x)=\frac{e^{-\invtemp E(x)}}{Z}\d x
\end{equation*}
is a unique invariant probability distribution for the diffusion process solving the SDE
\begin{equation}
\label{eq:general-stoch-gradient-flow}
\d X(t)=-\nabla E(X(t))\d t + \sqrt{2 T} \d W(t),
\end{equation}
where $W(t)=(W_1(t),\ldots,W_{Nd}(t))$ is the standard $Nd$-dimensional Wiener process with independent components.
Moreover, the stochastic dynamics \eqref{eq:general-stoch-gradient-flow} is reversible with respect to the Gibbs distribution $\mu$.

\medskip

It is natural to expect that a similar property holds for  infinite volume ($N=\infty$)  Gibbs distributions including IVPMs. 
This suggests the following natural plan for addressing the existence/uniqueness issue for IVPMs corresponding to the case $V(r)=|r|^2/2$
in higher dimensions $d\ge 2$ (and, in principle, more general $V$). The existence arguments from~\cite{Bakhtin-Li:MR3911894}--\cite{Bakhtin-Li:MR3856947} translate to higher dimensions easily since they are mostly based on the shear-invariance property
of the directed  polymer measures due to the quadratic form of $V$. Thus for a given $v$, with probability one, there is an IVPM on $S(v)$.
It is an invariant measure for the associated infinite-dimensional stochastic gradient flow, and to prove its uniqueness,
it is sufficient to prove that there are no other invariant distributions. This unique ergodicity on $S(v)$ seems plausible because some forms
of irreducibility and regularity of transition probabilities hold. For the irreducibility/controllability/accessibility property one can 
show that for any $x\in S(v)$ and any open set $O$ in $S(v)$ in an appropriate metric there are  controls (noise realizations) that bring the stochastic trajectory from $x$ to $O$. Certain regularity of transition probabilities perhaps short of the strong Feller property should follow from the elliptic
nature of the additive noise forcing all the coordinates independently. It is also plausible that a similar approach based on reversible Markov chains  might work for lattice polymers.

\smallskip

We reserve this program for a future publication and note that even the identity between Gibbs and invariant distributions is far from trivial in infinite dimensions.  Besides the fact that the DLR formalism for infinite volume Gibbs measures (see~\cite{Sinai:MR691854} 
or~\cite{Georgii:MR956646}) does not provide their explicit representations,
  one must also deal with infinite-dimensional gradient flows, probability distributions, and diffusions, and the Fokker--Planck equations must be made sense of in infinite dimensions if one wants to use them.
In several cases, these complications have been overcome, and  infinite volume Gibbs measures have been studied as invariant for reversible dynamics in infinite-dimensional configuration spaces, see 
\cite{Fritz:MR656511},
\cite{Liggett-book:MR776231},
\cite{Funaki-Spohn-1997:MR1463032},
\cite{Albeverio-Kondratiev-Rockner-Tsikalenko:MR1813614}
\cite{Bogachev-Rockner-Wang:MR2073138},
\cite{Jahnel-Kulske-2019:MR3919445}.

\medskip

The results and techniques from these papers do not apply to our polymer setting. A distinctive feature of our setup is that the energy function
is generated by an unbounded random stationary potential in $d+1$ dimensions. We note though that in the case of bounded potentials, the conditions of the theorem on equivalence between the Gibbs and invariance properties from  \cite{Fritz:MR656511} are satisfied. That result is based on the analysis of free energy, an approach completely different from ours. The fact that every IVPM is invariant under~\eqref{eq:general-stoch-gradient-flow} seems very robust but the converse is harder to verify directly in our setting. In the bounded potential case, some of our arguments can also be simplified.

\subsection{Main results}
We hope that eventually the dynamic polymers that we introduce  below in~\eqref{eq:main-sde} or their generalizations will be useful, in a variety of settings, in understanding the existence-uniqueness and other properties of static IVPMs providing a direct description of global solutions of the Burgers equation  in various dimensions and under various kinds of forcing (see~\cite{Bakhtin-Li:MR3911894} and a discussion of global solutions for more general
Hamilton--Jacobi equations and the associated generalized polymers in~\cite{Bakhtin-Khanin-non:MR3816628}).
In this paper, rather than  using stochastic dynamics to derive properties of Gibbs measures as the program presented above suggests,
we are going to explore it backwards in the $1+1$-dimensional setting  studied in~\cite{Bakhtin-Li:MR3911894}, \cite{Bakhtin-Li:MR3856947},
where the existence and uniqueness of IVPMs is known,  and prove
several results concerning the dynamics of the associated infinite dimensional stochastic flows.

Thus we assume that $d=1$, $V(r)=r^2/2$, and impose just a few requirements on the  random potential $(F_k(x))_{(k,x)\in \N\times\R}$:
 $(F_k)_{k\in\N}$ form an i.i.d.\ sequence and $(F_1(x))_{x\in\R}$ is a weakly mixing stationary process with finite exponential moments.
We give a complete set of requirements on the random potential in Section~\ref{section:random_pot} including a couple of additional technical ones such as
$C^2(\R)$-smoothness of realizations and the presence of large almost flat regions that is needed to prove an ordering by noise theorem. Here we just note that a broad class of stationary processes is covered including  asymptotically decorrelating Gaussian processes and shot-noise type processes with bounded compactly supported kernels.

In this setting, we study gradient-type stochastic dynamics on infinite polymer chains $x:\{0\}\cup \N\to \R$ pinned at $0$ ($X_0(t)\equiv0$) with Gaussian white noise acting on all coordinates and a random drift that depends on neighboring coordinates and the random potential:
\begin{equation}
\label{eq:main-sde}
\d X_k(t)= - \nabla_k E(X(t)) \d t + \sigma \d W_k(t),\quad i\in\N,
\end{equation}
where $\sigma=\sqrt{2\tempr}$ (the temperature $\tempr=\temp\in(0,\infty)$ is fixed along with the inverse temperature $\invtemp=\tempr^{-1}$  ),
the gradient $\nabla E(x)$ is understood as a formal application of the partial derivative to the formal expression
$E(x)$ given in  \eqref{eq:energy-of-infinite-path}:
\begin{equation*}
-\nabla_k E(x)=- \partial_{x_k}E(x)=  (\Delta x)_k + f_k(x_k),\quad k\in\N.
\end{equation*}
where the discrete Laplacian $\Delta_k$ can be  written as
\begin{equation}
\label{eq:laplace}
 \Delta_k x = \delta_{k+1}x-\delta_kx = x_{k-1}-2x_k+x_{k+1},\quad k\in\N,
\end{equation}
so \eqref{eq:main-sde} can be rewritten as
\begin{equation}
\label{eq:main-sde1}
\d X_k(t)=\Delta_k X(t)\d t +f_k(X_k(t))\d t + \sigma \d W_k(t),\quad k\in\N.
\end{equation}

We begin our study with  solutions of the polymer dynamics  \eqref{eq:main-sde} on finite time intervals. We
introduce appropriate metric spaces of polymers and
prove that the Galerkin approximations to \eqref{eq:main-sde} define the dynamics uniquely  
and show that the resulting solution maps form a continuous and order-preserving random dynamical system (RDS), a self-consistent
family of solution maps $(\Phi^t_{F,W})_{t\ge0}$, associated with the potential realization $F$ and the noise realization $W$.
Here the order-preserving property or monotonicity means that if $x\succeq y$, then $\Phi^t_{F,W}x\succeq\Phi^t_{F,W}y$, where $\succeq$
stands for  coordinatewise comparison: we write $x\succeq y$ if $x_k \ge y_k$ for all $k\in\N$. 

We also show that for each $v\in \R$, the space $S(v)$ is  a.s.-invariant under $(\Phi^t_{F,W})$.

These results are given in Sections~\ref{section:SDE} and~\ref{section:dyn_inf_poly} after introducing the setting in Section~\ref{sec:setting}.

\medskip

Then we study the long-term properties of the  polymer  dynamics. Let us collect our main results and state them in the form of one theorem:
\begin{Th} \label{th:collect} \leavevmode
\begin{enumerate}[I.]
\item
\label{it:Gibbs-is-invariant} For almost all realizations of $F$: every IVPM is invariant for the Markov process generated by \eqref{eq:main-sde}.
\item 
\label{it:uniq-erg-mix}
(Unique ergodicity; identity between Gibbs and invariant measures for a fixed slope; mixing.) For fixed $v\in\R$ and for almost every realization of the
potential~$F$: (i) the unique IVPM $\mu_{F,v}$ on $S(v)$ constructed 
in~\cite{Bakhtin-Li:MR3911894} is also a unique invariant distribution on $S(v)$ for the Markov process generated by \eqref{eq:main-sde}; 
(ii) 
this Markov process is mixing with respect to $\mu_{F,v}$.
\item
\label{it:1F1S}
 (1F1S.)  For fixed $v\in\R$ and for almost every realization of the 
potential $F$,  the RDS generated by~\eqref{eq:main-sde} on $S(v)$ a.s.-admits a unique stationary nonanticipating global solution serving as a one-point global pullback attractor on $S(v)$.
\item 
\label{it:ordering}
(Ordering by noise.) For almost every realization of the
potential $F$ and noise~$W$, if $x^1\in S(v_1)$, $x^2\in S(v_2)$ and $v_1>v_2$, then there is $t_0=t_0(F,W,x^1,x^2)>0$ such that $\Phi^t_{F,W}x^1\succeq\Phi^t_{F,W}x^2$ for~$t>t_0$.
\end{enumerate}
\end{Th}

We prove part~\ref{it:Gibbs-is-invariant} in Section~\ref{section:IVPM}. This is a general statement not involving specific slopes.  Once we fix a slope value $v\in\R$, much more information is provided by parts~\ref{it:uniq-erg-mix} and~\ref{it:1F1S} proven in Section~\ref{section:1F1S}.

Unique ergodicity (uniqueness of an invariant distribution in part~\ref{it:uniq-erg-mix})  implies, in particular, that if we fix an asymptotic slope $v$, then for typical initial conditions with respect to the IVPM $\mu_{F,v}$  and typical noise realizations, 
empirical measures  for the dynamics (time averages or normalized occupation measures) converge to $\mu_{F,v}$. Mixing is a stronger property. It means  that for every initial condition $x\in S(v)$, the distribution of $\Phi^t_{F,W}$ converges to $\mu_{F,v}$ as $t\to\infty$.

The 1F1S principle on each $S(v)$ (part~\ref{it:1F1S}) is a stronger statement which we are able to prove using the additional structure of the system. In fact, we first prove that every IVPM is invariant directly, but to prove the converse implication and the mixing property we have to rely on the 1F1S principle.  

It is natural to expect 1F1S (also known as synchronization) in order-preserving systems (i.e., systems with monotonicity property, which is tightly related to  the maximum principle), one example being the study of the Burgers equation and other HJB-type systems,  see~\cite{Bakhtin-Li:MR3911894} and references therein. Moreover, there are general theorems deriving synchronization from monotonicity under additional assumptions, see~\cite{Flandoli-Gess-Scheutzow:MR3630300}. However, these additional assumptions  (mixing with respect to a unique invariant distribution and concentration of the measure on intervals) are too restrictive to be useful for us before the unique ergodicity is established. So we derive unique ergodicity and mixing from 1F1S 
(part~\ref{it:uniq-erg-mix} from part~\ref{it:1F1S}) and not the other way around.

In fact, we first prove a weaker form of 1F1S in Section~\ref{sec:weak1F1S}, in terms of sample Markov measures and then upgrade it to the strong 1F1S in Section~\ref{sec:strong1F1S}.

The proof of the weaker form of 1F1S relies, in turn, on two properties. One is  the skew-invariance of the dynamics and the IVPMs with respect to shear transformations. This is the result of our assumption on the quadratic nature of  $V(r)=r^2/2$. We address it in 
Section~\ref{sec:shear-inv}.
The other important property (part \ref{it:ordering})  is a strengthening of monotonicity that we call {\it ordering by noise} and prove in Section~\ref{sec:monotonization}.

For a typical random environment, in the absence of noise, there are many local minima of the potential and the solutions of the deterministic counterpart of \eqref{eq:main-sde} with $\sigma=0$  may get  stuck at local energy minimizers and never get ordered. 
Thus ordering or monotonization phenomenon happens due to the presence of noise. 

Other related terms existing in the literature are {\it noise-induced order}  and  {\it order from noise}. The former
describes situations where increasing the noise level makes the leading Lyapunov exponent of the linearization of the stochastic system negative (so synchronization occurs). The latter is a broader term  more relevant for us. It  goes back to 
\cite{vonFoerster2003} and reflects the situation where due to noisy inputs the system has a chance to explore larger parts of the configuration space avoiding local traps and approaching the ground state.  The concept
of order from noise
is relevant in Markov Chain Monte Carlo algorithms~\cite{Winkler-MCMC:MR1316400}, in the stochastic gradient descent in deep learning~\cite{Goodfellow-Bengio-Courville_DeepL:MR3617773}, in the modern view of biological evolution~\cite{koonin_logic_2011}. 
In our setting, we can actually speak of an emerging total order since it is immediate to generalize  part~\ref{it:ordering}  to any number of initial conditions with distinct slopes.

We prove the ordering by noise property under the assumption of presence of large regions where the potential is almost flat. The mechanism of ordering that we exploit is based on a recurrence property for the polymer dynamics with respect to those flat regions and comparison of the polymer dynamics with solutions of homogeneous linear equation on those regions. However, we believe that the ordering holds (perhaps taking extremely long times) for a broader class of potentials like periodic ones where our assumption fails.

\subsection{Extreme decompositions and transversal fluctuations of polymers}
\leavevmode

Part~\ref{it:uniq-erg-mix} of Theorem~\ref{th:collect} falls short of a complete description of the ergodic decomposition for the Markov process generated by~\eqref{eq:main-sde}. It only says that once we fix~$v$, the measure $\mu_{F,v}$ is ergodic with probability 1. 
One could also be interested in joint behavior for all $v\in\R$ simultaneously. In principle, it can be more complicated due to the presence of uncountably many exceptional sets, one per each $v\in\R$.
However, a natural conjectural picture  for almost every  realization of $F$  is the following:  

For each $v\in\R$, $S(v)$ supports  a unique ergodic measure, $\mu_{F,v}$; there are no other ergodic measures, and thus every invariant distribution is a mixture of $\mu_{F,v}$, $v\in\R$.  Every invariant measure is an IVPM and vice versa. Each ergodic measure is an extreme IVPM and vice versa. For each $v\in\R$, $S(v)$ supports  a unique extreme IVPM, $\mu_{F,v}$; there are no other extreme IVPMs, so that
every IVPM is a mixture of $\mu_{F,v}$, $v\in\R$.

An equivalent way to  state this is to say that
two convex sets --- (i) all ergodic measures for polymer dynamics and (ii) all IVPMs ---
coincide, and the set of their extreme points (measures that are not mixtures of other measures from the same set) coincides with $\{\mu_{F,v}\}_{v\in\R}$.  See \cite[Chapter~7]{Georgii:MR956646} for generalities on extreme decompositions of infinite volume Gibbs measures. 

Moreover, we believe that the conjecture stated above for our model actually extends 
to a much broader class of self-interaction energies $V(\cdot)$ and, more generally, for a broad class of directed polymer models.   
In Section~\ref{section:IVPM}, we prove the following result towards this conjecture:

\begin{Th}\label{eq:extreme-measures-have-direction}
For almost every realization of $F$,  the following holds: if  (i) $\mu$ is an extreme IVPM or (ii) $\mu$ is an ergodic measure for the Markov polymer dynamics generated by~\eqref{eq:main-sde} and an IVPM,
then there is $v\in\R$ such that 
$\mu$ is concentrated on~$S(v)$. Since every IVPM is a mixture of extreme ones, each IVPM is concentrated on paths with well-defined slopes.
\end{Th}

The results of~\cite{Bakhtin-Li:MR3911894} state a.s.-existence of IVPMs concentrated on $S(v)$ simultaneously for all $v\in\R$ and a.s.-uniqueness for individual values of~$v$.
The proofs actually allow to conclude that there are at most countably many non-uniqueness values of $v$ at the same time which is still short of 
the conjectural picture above.

For the zero temperature (or viscosity in the Burgers/KPZ terms) case, an analogous conjecture would be that for almost every potential $F$ and simultaneously for all $v\in\R$, there is a unique one-sided ground state (semi-infinite energy minimizer, or geodesic) in $S(v)$. Most likely, this is false, in analogy with the a.s.-presence of shocks in global solutions of Burgers  equation, see also~\cite{janjigian2020geometry}, where related questions were studied for a last passage percolation lattice model.

Theorem~\ref{eq:extreme-measures-have-direction} means that transversal fluctuations of polymer paths under IVPMs are sublinear. 
Extending its proof and
using the tools developed in~\cite{Bakhtin-Li:MR3911894}, we are able to obtain an additional upper bound on transversal fluctuations in Section~\ref{section:ext_IVPM_asymp_slope}:
\begin{Th}\label{thm:fluc_intro}
For every $\xi'>3/4$, there is $Q>0$ such that 
for every $v\in\R$, for almost every realization of $F$, for $\mu_{F,v}$-a.e.\ polymer path $x$, 
$|x_k-vk|<Qk^{\xi'}$ for sufficiently large $k$.
\end{Th}
In other words, the transversal fluctuation exponent $\xi$ for polymers under any IVPM is bounded above by $3/4$. The KPZ universality suggests that the correct value for $\xi$ is $2/3$, see, e.g., \cite{Bakhtin-Khanin-non:MR3816628}, but this remains an open problem for our model. Of course, our estimate matches 
the upper bound  established in~\cite{Bakhtin-Li:MR3911894} on $\xi$ defined for
a sequence of polymer measures on paths of growing length.

One expects (see \cite{Bakhtin-Khanin-non:MR3816628}) that in dimension $1+1$, for each $v\in\R$ the IVPM (and stationary distribution for the dynamic polymer) $\mu_{F,v}$ is localized.  This localization statement means that under $\mu_{F,v}$  the deviations of polymers from a certain path
(fluctuating with the KPZ exponent $2/3$)
behave like a stationary process. Such localization statement has been obtained in the more traditional setting of polymers of growing finite length in \cite{bakhtin-seo2020localization} following the pioneering ideas of \cite{Bates-Chatterje:MR4089496} and using the 
topology on the space of measures introduced in~\cite{Mukherjee-Varadhan:MR3572328}.

\bigskip

{\bf Acknowledgments.} Yuri Bakhtin is grateful to the National Science Foundation for partial support via grant DMS-1811444. Both authors thank the anonymous referee for the comments that helped to improve the paper significantly.

\section{Setting and notation}\label{sec:setting}

\subsection{Random potentials}\label{section:random_pot}
Let us describe the probability space $(\FF,  \cF, \sF)$ of random potentials defined on  $\NR$. We denote the space of real-valued twice continuously differentiable functions on $ \R$  by $\sC^2(\R)$. We define a metric on $\sC^2(\R)$ by
\begin{align*}
    d_{\sC^2(\R)}(\gamma,\gamma') = \sum_{n=1}^\infty 2^{-n}\bigg(1\wedge \|\gamma -\gamma'\|_{\sC^2([-n,n])}\bigg)
\end{align*}
where $\|\gamma -\gamma'\|_{\sC^2([-n,n])} = \sum_{j=0}^2\sup_{s\in[-n,n]}\big|\frac{\d^j}{\d s^j}(\gamma(s) - \gamma'(s))\big|$.

\smallskip

We assume that $\FF$ is the space of functions $F:\N\times \R\to\R$ such that $F_k:\R\to\R$ is in $\sC^2(\R)$ for each $k\in\N$. 
Let $\mathcal{F}$ be the completion with respect to $\sF$ of the Borel $\sigma$-algebra generated by $\big(\sC^2(\R),
d_{\sC^2(\R)} \big)^\N$. 
Throughout, for each $F\in \FF$, we set
\begin{align}\label{eq:def_f}
    f(x) = \big(f_k(x_k)\big)_{k\in \N} = \big(-F'_k(x_k)\big)_{k\in\N}, \quad x=(x_1,x_2,\ldots)\in\R^\N,
\end{align}
where $F'_k$ is the first-order derivative of $F_k$. 
We assume that the probability measure $\sF$ on $(\FF,\cF)$ satisfies the following conditions:

\begin{itemize}
    \item  $\sF$ is invariant under  translations defined by
    \begin{align}\label{eq:Theta^n,a}
        (\Theta^{n,a} F)_m(r)=F_{m+n}(r+a),\quad (n,a)\in(\N\cup\{0\})\times\R,\ (m,r)\in\NR.
    \end{align}
Moreover, $(F_k(\cdot))_{k\in\N}$ forms an i.i.d.\ sequence.
   \item The action of the group of spatial shifts $(\Theta^{0,a})_{a\in\R}$ is ergodic on $\FF$. In particular, for each $k\in\N$, $F_k(\cdot)$ is an ergodic process.
    \item For the inverse temperature $\invtemp=\tempr^{-1}=\varkappa^{-1}$,
\begin{align*}\sF \exp\big(-\invtemp F_1(0)\big)<\infty.
\end{align*}
Here and throughout the paper, we use $\Pp \xi$ or $\Pp[\xi]$ to denote expectation of a r.v.\ $\xi$ with respect to a probability measure $\Pp$.

    \item For some $\eta>0$,
\begin{align}\label{eq:exp_moment_F}
    \sF \exp\big(\eta\| F_1\|_{\sC^2([-1,1])}\big)<\infty.
\end{align}
  \item The following holds $\sF$-a.s.: for every $n\in\N$, every $l\geq 0$ and every $\delta>0$, there is $a\in\R$ such that
\begin{align}\label{eq:flat}
    |f_k(r)|\leq \delta,\quad \text{\rm for all}\ (k,r)\in \{1,2,\dots, n\}\times [a-l,a+l].
\end{align}
\end{itemize}

\begin{Rem}
The results of~\cite{Bakhtin-Li:MR3911894} on IVPMs were obtained under all these restrictions with two exceptions.
One is the condition on the $\sC^2$-norm in~\eqref{eq:exp_moment_F} replacing a similar condition on the sup-norm in~\cite{Bakhtin-Li:MR3911894} .
Another is the ``presence of flatness'' condition~\eqref{eq:flat} which  was not needed in~\cite{Bakhtin-Li:MR3911894} at all, and in the present paper is  
only used in Section~\ref{sec:monotonization} and Section~\ref{section:1F1S}.
Section~\ref{section:SDE} and Section~\ref{section:dyn_inf_poly}  are independent of the results of \cite{Bakhtin-Li:MR3911894} and require only \eqref{eq:exp_moment_F} and the invariance of $\sF$ under \eqref{eq:Theta^n,a}.
\end{Rem}

It is easy to see that the conditions above define a broad class of random potentials including stationary Gaussian processes with decaying correlations and stationary  processes of  shot noise type based on contributions from configuration points of a Poisson point process.

The i.i.d.\ assumption implies that besides the invariance under the action of shift semigroup $\Theta=(\Theta^{n,a})$, the measure $\sF$ is also invariant under the action of the 
shear group $\Shear=(\Shear^v)$, where the shear transformation $\Shear^v$, $v\in\R$, is defined by 
\begin{equation}
\label{eq:shear-on-F}
(\Shear^v F)(x)=\big(F_k(x_k-kv)\big)_{k\in\N}\, , \quad x\in  \R^\N.
\end{equation}

\subsection{Function spaces}
\label{sec:funcion}
For $\alpha\in(0,\infty)$ and $p\in[1,\infty]$, we define $\|\cdot\|_{\alpha,p} :\R^\N\to[0,+\infty]$ by, for $x\in\R^\N$,
\begin{align}\label{eq:def_norms}
    \begin{split}
    &\|x\|_\alphap=\Bigg(\sum_{k\in\N}\bigg|\frac{x_k}{k^\alpha}\bigg|^p\Bigg)^\frac{1}{p},\quad p<\infty, \\
    &\|x\|_{\alpha,\infty}=\sup_{k\in\N}\frac{|x_k|}{k^\alpha}.
    \end{split}
\end{align}
Accordingly, we set
\begin{align*}
    \XX^\alphap= \big\{x\in\R^\N:\ \|x\|_\alphap<\infty\big\}.
\end{align*}
It can be readily checked that $\XX^\alphap$ equipped with the norm $\|\cdot\|_\alphap$ is a Banach space. 
We mainly work with $(\alphap)$ in
\begin{align}\label{eq:betap_set}
   \setbetap= \Big\{(\alpha,p)\in (0,+\infty)\times[1,+\infty]:\ \alpha p >1\Big\}.
\end{align}
We will often refer to the following obvious statement:
\begin{Lemma}\label{lemma:ptw_cvg}
If $(x^n)_{n\in\N}$ converges to $x$ in $\XX^\alphap$ for some $(\alphap)\in (0,\infty)\times [1,\infty]$, then $\lim_{n\to\infty}x^n_k=x_k$ in $\R$ for every $k\in\N$.
\end{Lemma}
\smallskip

Next, we introduce spaces for continuous paths. For two real numbers $t_1<t_2$ and a Banach space $\XX$ with norm $\vertiii{\cdot}$, we denote by $\sC(t_1,t_2;\XX)$ the space of continuous $\XX$-valued functions on $[t_1,t_2]$ and $\sC(\R;\XX)$ the space of $\XX$-valued continuous functions on $\R$. We equip $\sC(t_1,t_2;\XX)$ with the norm
\begin{align*}
    \| \gamma \|_{\sC(t_1,t_2;\XX)} = \sup_{s\in[t_1,t_2]}\vertiii{\gamma (s)},
\end{align*}
and equip $\sC(\R;\XX)$ with the LU (local uniform) metric
\begin{align}\label{eq:metric_cont}
    d_{\sC(\R;\XX)}(\gamma,\gamma')=\sum_{n=1}^\infty 2^{-n}\bigg(1\wedge\big\|\gamma-\gamma'\big\|_{\sC(-n,n;\XX)}\bigg).
\end{align}
Here, we used the notation $a\wedge b = \min\{a,b\}$ for real numbers $a,b$. Later, we will also write $a\vee b = \max\{a,b\}$. 
\medskip

Lastly,  for a linear space $\XX$ we denote the set of linear endomorphisms of $\XX$ by $\sL(\XX)$; for a Banach space $\XX$ with norm 
$\vertiii{\cdot}$,  we denote the set of bounded linear endomorphisms of $\XX$ by $\sL_\b(\XX)$ (omitting the dependence on the norm for brevity).

We often abbreviate a normed space like $(\XX,\vertiii{\cdot})$ to $\XX$ and speak
about the topology, the Borel $\sigma$-algebra of  $\XX$, or continuity in $\XX$, omitting an explicit reference to the norm. Our main spaces are
$\XX^\alphap$ equipped with $\|\cdot\|_\alphap$, the real line with Euclidean norm, and $\R^\N$ with the product topology which is characterized by the coordinatewise convergence and is metrizable. We note that $\XX^\alphap$ is Polish for $p<\infty$ and $\R^\N$ is also Polish.

For a topological space $\XX$, we denote by $\cB(\XX)$ the associated Borel $\sigma$-algebra. It is easy to see that
\begin{align*}\XX^\alphap\in\cB(\R^\N),\quad(\alphap)\in\Pi.
\end{align*}
Due to the lack of separability in $\XX^\alphainfty$, we need a $\sigma$-algebra generated by a weaker topology. For $(\alphap)\in\Pi$, we introduce
\begin{align}\label{eq:def_B^alphap}
    \cB^\alphap = \{E\cap \XX^\alphap: E\in\cB(\R^\N)\}.
\end{align}
We mostly work with the measurable spaces $(\XX^\alphap, \cB^\alphap)$ and $(\R^\N,\cB(\R^\N))$. The lemma below shows that, for $p<\infty$, $\cB^\alphap$ coincides with the $\sigma$-algebra generated by the strong topology.
\begin{Lemma}\label{lem:B^alphap}
If $p<\infty$, then $\cB^\alphap = \cB(\XX^\alphap)$.
\end{Lemma}
\begin{proof}
The result follows from the separability of $\XX^\alphap$ and observations that every closed ball in $\cB(\XX^\alphap)$ belongs to $\cB(\R^\N)$ and that cylindrical sets in $\cB(\R^\N)$ intersected with $\XX^\alphap$ belong to $\cB(\XX^\alphap)$.
\end{proof}

\subsection{Noise}Let us describe the probability space $(\WW,\cW,\sW)$ supporting the two-sided $\R^\N$-valued Wiener process. We identify $\WW$ with $(\sC_0(\R;\R))^\N$, the countable product of real-valued continuous functions $W_k, k\in\N$, defined on $\R$ and satisfying $W_k(0)=0$. Under $\sW$,  the components of $W=(W_k)_{k\in\N}\in\WW$ are independent standard two-sided Wiener processes. We assume that $\cW$ is complete with respect to $\sW$. 
For any interval $I\subset \R$, we define $\cW_{I}$ to be the completion of
\begin{align*}
     \sigma\big\{W(r_2)-W(r_1):\ r_1,r_2\in I;\ r_1\leq r_2 \big\}
\end{align*}
under $\sW$.
In addition, we define $\Wc_t=\Wc_{(-\infty,t]}$, $t\in\R$. We call a stochastic process $(X_t)_{t\in\R}$ nonanticipating if it is adapted to the filtration $(\Wc_t)_{t\in\R}$.

The group of time shifts $\tshift =(\tshift^s)_{s\in \R}$ acting on $\WW$   is defined by
\begin{align}
\label{eq:tshift}
    \tshift^sW(t) = W(t+s)-W(s),\quad s,t\in\R.
\end{align}
These shifts preserve $\sW$. 
We will also show in Lemma~\ref{Lemma:W_local_Holder} that $W$ is a continuous process in $\XX^{\alphap}$ under $\sW$ for all $(\alphap)\in\Pi$.

\section{Well-posedness of the SDE}\label{section:SDE}
We will consider a slightly more general system than \eqref{eq:main-sde1}. 
Recalling the definition of $f$ in \eqref{eq:def_f}, fixing an arbitrary $\sigma=\sqrt{2\tempr}>0$ and
 $A\in \sL(\R^\N)$, we consider the SDE
\begin{align}\label{eq:gen_SDE}
\begin{split}
\d X (t) &=  \big(A X(t) + f(X(t))\big)\d t + \sigma \d W(t),\quad  t\in\R ,\\
X (0) &= x.
\end{split}
\end{align}

\begin{Def}\label{def:solution}
For $x\in \R^\N$, $A\in\sL(\R^\N)$, $F\in\FF$ and $W\in \WW$, we say that $X$ is a solution of \eqref{eq:gen_SDE} and write $X\in\cS(x,A,F,W)$ if $X,\, AX \in \big(\sC(\R;\R)\big)^\N$ and
\begin{align}\label{eq:sys_int_eq}
        X_k(t) = x_k +\int_0^t \big(A_k X(s) + f_k(X_k(s))\big)\d t +\sigma W_k(t),\quad k\in\N,\ t\in \R.
\end{align}
Here, $A_ky=(Ay)_k$ for $y\in\R^\N$.
\end{Def}
Note that this is a pathwise definition, it does not involve any stochastic integrals or probability at all. We are mostly interested in the forward solutions 
($t\ge0$)  but solving the equation
backwards ($t\le 0$) also makes sense for bounded operators~$A$, see below.

For a topological space $\XX$, we denote
\[
\cS(x,A,F,W;\XX)=\cS(x,A,F,W)\cap \sC(\R;\XX).
\] 
The following well-posedness result will be proved in the next section. We recall the definitions of $\Pi$ in~\eqref{eq:betap_set} and $\cB^\alphap$ in \eqref{eq:def_B^alphap}.

\begin{Th}\label{thm:RDS} 
There are full measure sets $\FF_0\in\mathcal{F}$ and $\WW_0\in\cW $ such that the following holds. 
Let $(\alphap),(\alphaprime)$ satisfy
\begin{equation}
\label{eq:(beta'-beta)p'>1}
(\alphap),(\alphaprime)\in\setbetap;\quad 
p'<\infty;\quad  (\alpha'-\alpha)p'>1.
\end{equation}
Let
\begin{equation}
\label{eq:A-bounded}
A\in \sL_\b(\XX^\alphap)\cap \sL_\b(\XX^\alphaprime).
\end{equation}
Then, there is a map
\begin{align*}
    \Phi :\ \R\times \XX^\alphap \times \FF_0\times \WW_0 \to \XX^\alphap
\end{align*}
with the following properties:

\begin{enumerate}[1.]
    \item \label{item:RDS_meas} Measurability: $\Phi$ is $(\cB(\R)\times \cB^\alphap\times \cF\times \cW,\ \cB^\alphap)$-measurable, and progressively measurable,
    i.e., for every $t\in \R$, the restricted map
    \begin{align*}
        \Phi|_{I_t}:\  I_t\times \XX^\alphap \times \FF_0\times \WW_0\to \XX^\alphap
    \end{align*}
    is $(\cB(I_t)\times \cB^\alphap\times \cF\times \cW_{I_t},\ \cB^\alphap)$-measurable, where $I_t=[0\wedge t,\,0\vee t]$. 
    \item \label{item:RDS_SDE} Solving the SDE: for each $(x,F,W)\in \XX^\alphap \times \FF_0\times \WW_0$, the map $\Phi(\cdot,x,F,W)$ is a unique solution to \eqref{eq:gen_SDE} in the sense that
    \begin{align*}
        \big\{\Phi(\cdot,x,F,W)\big\}= \cS(x,A,F,W;\XX^\alphap).
    \end{align*}
    \item \label{item:RDS_cont_dep} Continuous dependence: for each $F\in\FF_0$, the map 
    \begin{align*}
        \R \times \XX^\alphap\times \WW_0\ni\quad  (t,x,W)\mapsto \Phi(t,x,F,W)\quad \in \XX^\alphap
    \end{align*}
    is continuous.
    Here $\WW_0$ is endowed with the topology induced from $\sC(\R;\XX^\alphap)$.
    \item \label{item:RDS_cocycle} Cocycle property: for each $F\in\FF_0$ and $W\in\WW_0$, the maps
    \begin{align}
    \label{eq:solution-map}
        \Phi^t_{F,W}= \Phi(t,\cdot,F,W):\ \XX^\alphap \to \XX^\alphap
    \end{align}
    satisfy 
    \begin{enumerate}[i.]
        \item $\Phi^0_{F,W}$ is the identity map on $\XX^\alphap$,
        \item $\Phi^{t+s}_{F,W}= \Phi^{t}_{F,\theta^s W}\circ \Phi^{s}_{F,W} $ for all $s,t\in\R$.
    \end{enumerate}
\end{enumerate}

\end{Th}

\begin{Rem}\label{remark:F=0_W=0}
The sets $\FF_0$ and $\WW_0$ are described in Lemma~\ref{Lemma:f_property} and Lemma~\ref{Lemma:W_local_Holder}, respectively. In particular, we have $0\in\FF_0$ and $0\in\WW_0$.
\end{Rem}

\begin{Rem}
Throughout the paper, in a slight abuse of notation, we will not distinguish between $\Fc$, $\Wc$ and their restrictions to $\FF_0$ and $\WW_0$.
The same goes for the restrictions of measures  $\sF$ and $\sW$ to those  $\sigma$-algebras.
\end{Rem}

\subsection{Preliminary results}
We denote by $C$ a positive constant, which may vary from instance to instance.
\subsubsection{Conditions for the Cauchy property and convergence}

\begin{Lemma}\label{Lemma:Cauchy}
Suppose $X^n \in \cS(x^n,A^n,F^n,W^n;\XX^\alphap)$ for all $n\in\N$ and for some $(\alphap)\in(0,\infty)\times[1,\infty]$. Further assume
\begin{itemize}
    \item $(x^n)_{n\in\N}$ is Cauchy in $\XX^\alphap$;
    \item $(W^n)_{n\in\N} $ is Cauchy in $\sC(\R;\XX^\alphap)$;
    \item for each $T>0$, there are $C>0$ and real $(\delta_{k,j})_{k,j\in\N}$ with $\lim_{k,j\to\infty}\delta_{k,j}=0$ such that the following holds for all $m,n\in\N$ and $s\in[-T,T]$,
    \begin{align}\label{eq:A_diff}
        \|A^mX^m(s)-A^nX^n(s)\|_\alphap\leq \delta_{m, n}+C\|X^m(s)-X^n(s)\|_\alphap.
    \end{align}
    \item for each $T>0$, there are $K,\upsilon>0$ and real $(\delta'_{k,j})_{k,j\in\N}$ with $\lim_{k,j\to\infty}\delta'_{k,j}=0$ such that the following holds for all $m,n\in\N$ and $s\in[-T,T]$,
\begin{align}\label{eq:f_difference_condition}
    \|f^m(X^m(s))-f^n(X^n(s))\|_\alphap\leq \delta'_{m, n}+K\inf_{N\in1+\N}\big( N^{-\upsilon}+\|X^m(s)-X^n(s)\|_\alphap\log N\big).
\end{align}
\end{itemize}
Then, the sequence $(X^n)_{n\in\N}$ is Cauchy in $\sC(\R;\XX^\alphap)$.
\end{Lemma}

\begin{proof}[Proof of Lemma~\ref{Lemma:Cauchy}]
For convenience, we write $\|\cdot\|=\|\cdot\|_\alphap$ and $\XX=\XX^\alphap$ in this proof. By the definition of the metric in \eqref{eq:metric_cont}, we only need to show that $(X^n)_{n\in\N}$ is Cauchy in $\sC(-T,T;\XX)$ for any fixed $T>0$. We only prove the Cauchy property in $\sC(0,T;\XX)$; the other part is similar.

\smallskip
Recalling $K$ from \eqref{eq:f_difference_condition}, we choose $\kappa\in\N$ large enough to ensure 
\begin{align}\label{eq:choice_of_m}
    -\upsilon +K\frac{T}{\kappa}<0.
\end{align}
Set $T_j = \frac{j}{\kappa}T$, for $j=0,1,2,\dots,\kappa$.
The plan is to show the Cauchy property of $(X^n)_{n\in\N}$ in $\sC(T_j,T_{j+1};\XX)$ iteratively. By Definition~\ref{def:solution}, the solution $X^n$ satisfies 
\begin{align*}
    X^n(t) &= x^{n}+\int_0^t \Big(A^n X^n(s)+f^n\big(X^n(s)\big)\Big)\d s + \sigma W^n(t),
\end{align*}
understood in the sense of \eqref{eq:sys_int_eq}.
For  $t\in[0,T_1]$ we write
\begin{align*}
    &\|X^m(t)-X^n(t)\|\\
    &\qquad\leq \|x^{m}-x^{n}\| + \int_0^t\Big(\|A^mX^m(s)-A^nX^n(s)\| + \|f^m(X^m(s)-f^n(X^n(s))\| \Big)\d s\\
    &\qquad\qquad\qquad+ \sigma \|W^m(t)-W^n(t)\|,
\end{align*}
and apply Gronwall's inequality, \eqref{eq:A_diff}, and  \eqref{eq:f_difference_condition} to obtain
\begin{align*}
    &\|X^m-X^n\|_{\sC(0,T_1;\XX)}\\
    &\leq \Big(\|x^{m}-x^{n}\|+\sigma\|W^m-W^n\|_{\sC(0,T_1;\XX)}+\delta_{m,n}+\delta'_{m, n} +K T_1N^{-\upsilon}\Big)e^{(C+K\log N)T_1}.
\end{align*}
For any $\eps>0$, by \eqref{eq:choice_of_m}, we can set $N$ large enough to ensure
\begin{align*}
    K T_1N^{-\upsilon}e^{(C+K\log N)T_1}=\big(K T_1e^{CT_1}\big)N^{-\upsilon+K T_1}<\eps.
\end{align*}
With this choice of $N$, for sufficiently large $m,n$, we also have
\begin{align*}
    \Big(\|x^{m}-x^{n}\|+\sigma\|W^m-W^n\|_{\sC(0,T_1;\XX)}+\delta_{m, n}+\delta'_{m, n} \Big)e^{(C+K\log N)T_1}<\eps.
\end{align*}
The last three displays imply now that  $(X^n)_{n\in\N}$ is Cauchy in $\sC(0,T_1;\XX)$.

\smallskip

To proceed, let us assume that $(X^n)_{n\in\N}$ is Cauchy in $\sC(T_{j-1},T_{j};\XX)$ for some $j\geq 1$. For each $n$ and $t\geq T_j$, we have
\begin{align*}
    X^n(t)= X^n(T_j)+ \int_{T_j}^t\Big(A X^n(s)+f^n\big(X^n(s)\big)\Big)\d s + \sigma W^n(t)- \sigma W^n(T_j).
\end{align*}
Analogously, Gronwall's inequality together with \eqref{eq:A_diff} and \eqref{eq:f_difference_condition} implies 
\begin{align*}
    \|X^m-X^n\|_{\sC(T_j,T_{j+1};\XX)}\leq \Big(\|X^m(T_j)-X^n(T_j)\|+\sigma\|W^m-W^n\|_{\sC(T_j,T_{j+1};\XX)}\\+\delta_{m,n}+\delta'_{m, n} +KT_1N^{-\upsilon}\Big)e^{(C+K\log N)T_1}.
\end{align*}
Similar to the analysis on  $[0,T_1]$, we first take $N$ sufficiently large and then $m,n$ sufficiently large to obtain the Cauchy property in $\sC(T_j,T_{j+1};\XX)$. This completes the inductive step thus proving that the sequence $(X^n)_{n\in\N}$ is Cauchy in $\sC(\R;\XX^\alphap)$.
\end{proof}

We also need a convergence version of Lemma~\ref{Lemma:Cauchy} which can be proved analogously:

\begin{Lemma}\label{Lemma:cvg_sol}
Suppose $X^n \in \cS(x^n,A^n,F^n,W^n;\XX^\alphap)$ for all $n\in\N\cup\{\infty\}$ and for some $(\alphap)\in(0,\infty)\times[1,\infty]$. Further assume
\begin{itemize}
    \item $\lim_{n\to\infty}x^n = x^\infty$ in $\XX^\alphap$;
    \item $\lim_{n\to\infty}W^n=W^\infty$ in $\sC(\R;\XX^\alphap)$;
    \item for each $T>0$, there are $C>0$ and real $(\delta_{k})_{k\in\N}$ with $\lim_{k\to\infty}\delta_{k}=0$ such that the following holds for all $n\in\N$ and $s\in[-T,T]$,
    \begin{align*}
        \|A^n X^n(s)-A^\infty X^\infty(s)\|_\alphap\leq \delta_{n}+\|X^n(s)-X^\infty(s)\|_\alphap.
    \end{align*}
    \item for each $T>0$, there are $K,\upsilon>0$ and real $(\delta'_{k})_{k\in\N}$ with $\lim_{k\to\infty}\delta'_{k}=0$ such that the following holds for all  $n\in\N$ and $s\in[-T,T]$,
\begin{align*}\|f^n(X^n(s))-f^\infty(X^\infty(s))\|_\alphap\leq \delta'_{n}+K\inf_{N\in1+\N}\big( N^{-\upsilon}+\|X^n(s)-X^\infty(s)\|_\alphap\log N\big).
\end{align*}
\end{itemize}
Then, $\lim_{n\to\infty}X^n=X^\infty$  in $\sC(\R;\XX^\alphap)$.
\end{Lemma}

\subsubsection{Properties of the random potential}

\begin{Lemma}\label{Lemma:f_property}
Let $\FF_0\subset \FF$ be the set of potentials $F$ with the following properties: 
\begin{enumerate}[(1)]
    \item \label{item:f_k_upper_bound}  
    
    there is a positive constant $C>0$ such that
    \begin{align*}
    |f_k(r)|\leq C(1+\log k+\log_+|r|), \quad k\in \N,\ r\in\R,
\end{align*}
where $\log_+ = \log \vee 0$;
    \item \label{item:f-f} for every $(\alpha,p)\in\setbetap$ (see \eqref{eq:betap_set}), every $M>0$ and every $u,c\in\R$, there are $C,\upsilon>0$ such that
    \begin{align*}
    \|f(x)-f(y)\|_\alphap\leq C\inf_{N\in1+\N}\big(N^{-\upsilon}+\|x-y\|_\alphap\log N\big),
\end{align*}
    holds for all $x,y\in \R^\N$ satisfying $\|x-z\|_\alphap,\|y-z\|_\alphap\leq M$ where $z=(ku+c)_{k\in\N}$. 
\end{enumerate}

\medskip

Then, $\FF_0\in\cF$ and $\sF(\FF_0)=1$. In addition,  $\FF_0$ is invariant under the action of the shift semigroup $\Theta$ defined in 
\eqref{eq:Theta^n,a} and the shear group $\Shear$ defined in \eqref{eq:shear-on-F}.
\end{Lemma}

In the remaining part of this section, we often apply Lemma~\ref{Lemma:f_property}\eqref{item:f-f} with $u,c=0$ to $x,y\in\XX^\alphap$ satisfying $\|x\|_\alphap,\|y\|_\alphap\leq M$.

\begin{proof}[Proof of Lemma \ref{Lemma:f_property}]
We use the Borel--Cantelli lemma and the completeness of $\cF$ to verify the claim.
First, we construct a full measure set of potentials having property~\eqref{item:f_k_upper_bound}. For $k\in\N$, $n\in\Z$ and $c\in \R$, set
\begin{align*}
    \xi_{k,n} =\exp\bigg(\eta \sup_{r\in[n,n+1]}|F'_k(r)|\bigg).
\end{align*}
Using \eqref{eq:exp_moment_F}, the invariance of $\sF$ under \eqref{eq:Theta^n,a} (so $\sF \xi_{k,n}=\sF \xi_{1,0}$), and the Markov inequality, we have
\begin{align*}
    \sum_{k\in\N,n\in\Z}\sF\big\{\xi_{k,n}> k^2 (|n|+1)^2\big\}\leq \sum_{k\in\N,n\in\Z}\frac{\sF \xi_{k,n}}{k^2 (|n|+1)^2}\\
    = \sF \xi_{1,0}\sum_{k\in\N,n\in\Z}\frac{1}{k^2(|n|+1)^2} <\infty.
\end{align*}
The Borel--Cantelli lemma implies the existence of a full measure set $\widetilde\FF\in\mathcal{F}$ such that for each $F\in\widetilde\FF$ there is $\kappa\in\N$ such that $\xi_{k,n}\leq k^2 (|n|+1)^2$ for all $k, n$ satisfying $k+|n|> \kappa$.
Hence, for all $k, n$ satisfying $k+|n|> \kappa$, we have
\begin{align*}
    F'_k(r)\leq 2\eta^{-1}\big(\log k+\log(|n|+1)\big)\leq C(\log k + \log_+|r|),\quad   r\in[n,n+1].
\end{align*}
 Since each $F_k$ is continuously differentiable, we can adjust the constant $C>0$ in this inequality to ensure that
\begin{align*}
    F'_k(r)\leq C,\quad   r\in[n,n+1]
\end{align*}
holds also for $k$ and $n$ satisfying $k+|n|\leq\kappa$. Hence, property~\eqref{item:f_k_upper_bound} holds for all potentials in $\widetilde\FF$.
\bigskip

Now we turn to property \eqref{item:f-f}. We first construct a full measure set $\widehat{\FF}$ and then show that the desired property holds on $\widehat{\FF}$.

Let $M', \alpha' \in \Q\cap(0,\infty)$, $u',c' \in \Q$, and 
\begin{align}\label{eq:z'}
    z' = (ku'+c')_{k\in\N}.
\end{align}
For $\gamma'>0$ to be specified later, we set
\begin{align}\label{eq:a_k,b_k}
    a_k = M'k^\alpha,\quad b_k=\gamma'\log k,\qquad  k\in\N.
\end{align}
Consider the event
\begin{align*}
    E_k =\bigg\{\sup_{(r,r')\in[-a_k,a_k]^2}\frac{|f_k(r+z'_k)-f_k(r'+z'_k)|}{|r-r'|}\geq b_k\bigg\}\in \mathcal{F}.
\end{align*}
For $m,n\in\Z$ with even $m-n$, we set $B_{m,n} = (m,n)+\{(s_1,s_2):|s_1|+|s_2|\leq 1\}$. These sets cover $\R^2$. The number of $B_{m,n}$ with $m=n$ and intersecting $[-a_k,a_k]^2$ nontrivially is bounded by $Ca_k$. For $(r,r')\in B_{m,n}$ with $m=n$, we have $(r,r')\in (m,n)+[-1,1]^2$. Hence, for such $(m,n)$, by \eqref{eq:exp_moment_F} and the invariance of $\sF$ under \eqref{eq:Theta^n,a}, we have
\begin{align*}
    &\sF\bigg\{\sup_{(r,r')\in B_{m,n}}\frac{|f_k(r+z'_k)-f_k(r'+z'_k)|}{|r-r'|}\geq b_k\bigg\}
    \\
    &\leq \sF\bigg\{\sup_{(r,r')\in[-1,1]^2}\frac{|f_1(r+z'_k)-f_1(r'+z'_k)|}{|r-r'|}\geq b_k\bigg\}\leq Ce^{-\eta b_k}.
\end{align*}
The rest of $[-a_k,a_k]^2$ can be covered by at most $Ca^2_k$ many $B_{m,n}$ with $m\neq n$. Note that for $(r,r')\in B_{m,n}$ with $m\neq n$, we must have $|r-r'|\geq 1$. Hence, if $m\neq n$, we get, by \eqref{eq:Theta^n,a} and \eqref{eq:exp_moment_F}, that
\begin{align*}
    &\sF\bigg\{\sup_{(r,r')\in B_{m,n}}\frac{|f_k(r+z'_k)-f_k(r'+z'_k)|}{|r-r'|}\geq b_k\bigg\}\\
    &\leq \sF\bigg\{\sup_{r\in m+[-1,1]}|f_k(r+z'_k)|\geq \frac{1}{2}b_k\bigg\}+\sF \bigg\{\sup_{r'\in n+[-1,1]}|f_k(r'+z'_k)|\geq \frac{1}{2}b_k\bigg\}\leq Ce^{-\frac{\eta}{2}b_k}.
\end{align*}
Applying the union bound, we obtain
\begin{align*}
    \sF (E_k) &\leq \sum_{m,n:\ B(m,n)\cap[-a_k,a_k]^2\neq \emptyset}\sF\bigg\{\sup_{(r,r')\in B_{m,n}}\frac{|f_k(r+z'_k)-f_k(r'+z'_k)|}{|r-r'|}\geq b_k\bigg\}\\
    &\leq Ca_k e^{-\eta b_k} + Ca^2_k e^{-\frac{\eta}{2}b_k}.
\end{align*}
To have $\sum_{k\in\N} \sF(E_k)<\infty$, 
we only need to ensure
\begin{align*}\sum_{k>1} a^2_k e^{-\frac{\eta}{2}b_k} = M^2\sum_k k^{2\alpha' -\frac{1}{2}\eta' \gamma' }<\infty.
\end{align*}
It suffices to make sure that
\begin{align*}\frac{1}{2}\eta' \gamma' - 2\alpha' >1,
\end{align*}
which is possible by choosing $\gamma'$ sufficiently large. Then, the Borel--Cantelli lemma implies that there is a full measure set $\FF_{M',\alpha',u',c'}$ such that for each $F \in \FF_{M',\alpha',u',c'}$ there is a constant $\kappa>0$ such that
\begin{align*}
    \sup_{r,r'\in[-a_k,a_k]}\frac{|f_k(r+z'_k)-f_k(r'+z'_k)|}{|r-r'|}<  b_k, \quad k\geq \kappa.
\end{align*}
For $k<\kappa$, we can find a large constant $C>1$ such that the left-hand side of the above display is bounded by $C$. Hence, 
\begin{align}\label{eq:lipschitz_constant_at_a_k}
    \sup_{r,r'\in[-a_k,a_k]}\frac{|f_k(r+z'_k)-f_k(r'+z'_k)|}{|r-r'|}< C (1+b_k), \quad k\in\N.
\end{align}
Now, we set
\begin{align}\label{eq:hat_F}
    \widehat{\FF}= \widetilde \FF \cap \bigcap_{\substack{M',\,\alpha'\in\Q\cap(0,\infty) \\ u',\,c'\in\Q}} \FF_{M',\alpha',u',c'}
\end{align}
where $\widetilde \FF$ is the full measure set  on which  property \eqref{item:f_k_upper_bound} holds.

Since in the following we will need an approximation argument, let us summarize the above for convenience: on $\widehat\FF$, for every $M',\alpha'\in\Q\cap (0,\infty)$, $u',c'\in\Q$ and $z'= z'(u',c')$ defined in \eqref{eq:z'},
there is $\gamma'>0$ such that \eqref{eq:lipschitz_constant_at_a_k} holds for $a_k,b_k$ given in~\eqref{eq:a_k,b_k}.

Now let us now show that  property \eqref{item:f-f} holds on $\widehat{\FF}$. Let $z,(\alphap),M,u,c$ be given in the statement of \eqref{item:f-f}.
For simplicity, we write $\|\cdot\|=\|\cdot\|_\alphap$. For all $x,y\in\R^\N$ satisfying $\|x-z\|,\|y-z\|\leq M$, the definition of the norm in \eqref{eq:def_norms} gives that
\begin{align*}
    |x_k-ku-c|,\ |y_k-ku-c|\leq  Mk^{\alpha}, \quad k\in\N.
\end{align*} 
Then, we fix some $u',c'\in\Q$ sufficiently close to $u,c$, and some $M',\alpha'\in\Q\cap(0,\infty)$ sufficiently large to ensure that, for all such $x,y$,
\begin{align}\label{eq:upper_bound_by_a_k}
    |x_k-z'_k|,\ |y_k-z'_k|\leq  a_k, \quad k\in\N
\end{align}
for $z'_k$ and $a_k$ given in \eqref{eq:z'} and \eqref{eq:a_k,b_k}, respectively.

To estimate $\|f(x)-f(y)\|$, using \eqref{eq:upper_bound_by_a_k} and \eqref{eq:lipschitz_constant_at_a_k}, we get
\begin{align}\label{eq:f_diference_k<N}
    |f_k(x_k)-f_k(y_k)|\leq C (1+b_k) |x_k-y_k|.
\end{align}
By property \eqref{item:f_k_upper_bound},
we also have 
\begin{align}\label{eq:f_diference_k>N}
\begin{split}
    |f_k(x_k)-f_k(y_k)|\leq C(1+\log k &+ \log_+|x_k-z'_k|+\log_+|z'_k|
    \\
    &+ \log_+|y_k-z'_k|+\log_+|z'_k| )\leq C(1+\log k),
\end{split}
\end{align}
where \eqref{eq:upper_bound_by_a_k} is used in the last inequality. 
We distinguish two cases: $p<\infty$ or $p=\infty$. In the first case, taking any natural number $N>1$ , applying \eqref{eq:f_diference_k<N} to $k\leq N$ and \eqref{eq:f_diference_k>N} tor $k>N$, we obtain
\begin{align*}
    \|f(x)-f(y)\|^{p}& = \sum_{k\in\N}|f_k(x_k)-f_k(y_k)|^{p}k^{-\alpha p} \\
    &\leq \sum_{k\leq N} C (1+b_k)^{p}|x_k-y_k|^{p}k^{-\alpha p} + \sum_{k>N}C(1+\log k)^{p}k^{-\alpha p}\\
    &\leq C(1+b_N)^{p}\|x-y\|^{p}+ C \sum_{k>N}k^{-\alpha p+\delta}\\
    &\leq C (\log N)^{p}\|x-y\|^{p}+CN^{-\alpha p+1+\delta},
\end{align*}
where we fix some $\delta\in(0,\alpha p-1)$.
\medskip
If $p=\infty$, then applying \eqref{eq:f_diference_k<N} to $k\leq N$ and~\eqref{eq:f_diference_k>N} to $k>N$, we obtain
\begin{align*}
    \|f(x)-f(y)\|& = \sup_{k\in\N}|f_k(x_k)-f_k(y_k)|k^{-\alpha} \\
    &\leq \sup_{k\leq N} C (1+b_k)|x_k-y_k|k^{-\alpha} + \sup_{k>N}C(1+\log k)k^{-\alpha}\\
    &\leq C (1+b_N)\|x-y\|+ C N^{-\alpha}\log N\\
    &\leq C \log N\|x-y\|+CN^{-\alpha+\delta},
\end{align*}
for some $\delta\in(0,\alpha)$. Hence, property \eqref{item:f-f} holds on $\widehat \FF$.

In conclusion, the set $\widehat\FF$ defined in \eqref{eq:hat_F} is measurable and has full measure as a countable intersection of full measure sets, and  the properties \eqref{item:f_k_upper_bound} and \eqref{item:f-f} hold for each $F\in \widehat \FF$. Hence $\widehat\FF\subset \FF_0$.
By the completeness of $\cF$, we conclude that $\FF_0\in\cF$ and $\sF(\FF_0)=1$. 
 Checking the  invariance under the action of $\Theta$ and $\Shear$ is straightforward.
\end{proof}

\subsubsection{Properties of the noise}
\begin{Lemma}\label{Lemma:W_local_Holder}
Let $\WW_0= \widetilde\WW\cap\big(\bigcap_{\alphap\in\Pi}\sC(\R;\XX^\alphap)\big)$
where $\widetilde \WW$ is the set on which $W_k$ satisfies the law of the iterated logarithm for every $k\in\N$, 
i.e, for every $W\in \widetilde\WW$,
\[
\limsup_{t\to\pm\infty} \frac{|W_k(t)|}{\sqrt{2t\log\log t}} =1,\qquad k\in\N.
\]
Then $\WW_0\in\cW$ and $\sW(\WW_0)= 1$. In addition, $\WW_0$ is invariant under the action of the time shifts $\tshift=(\tshift^t)$ defined
in~\eqref{eq:tshift}.
\end{Lemma}

\begin{proof} 
Let $(\alphap)\in\setbetap$. If $p<\infty$, then applying the Minkowski integral inequality and estimates of Gaussian moments, we can obtain, for $q>2$,
\begin{align*}
    &\sW\big\|W(t_1)-W(t_0)\big\|^{q}_\alphap = \sW\bigg(\sum_{k\in\N}\big|W_k(t_1)-W_k(t_0)\big|^p k^{-\alpha p }\bigg)^\frac{q}{p} \\
    &\leq \bigg(\sum_{k\in\N}\Big(\sW\big|W_k(t_1)-W_k(t_0)\big|^{q}\Big)^\frac{p}{q} k^{-\alpha p}\bigg)^\frac{q}{p}\leq C |t_1-t_0|^\frac{q}{2}.
\end{align*}
Due to $q/2>1$ and the completeness of $\XX^\alphap$, the Kolmogorov--Chentsov continuity theorem implies that $t\mapsto W(t)$ is H\"older-continuous in $\XX^\alphap$, $\sW$-a.s. If $p=\infty$, we choose $q>0$ so large that $q>2$ and $\alpha q>1$. Then, using bounds for Gaussian moments, we have
\begin{align*}
    &\sW\big\|W(t_1)-W(t_0)\big\|^{q}_{\alpha,\infty} = \sW\bigg(\sup_{k\in\N}\big|W_k(t_1)-W_k(t_0)\big|^q k^{-\alpha q }\bigg) \\
    &\leq \sum_{k\in\N}\Big(\sW\big|W_k(t_1)-W_k(t_0)\big|^{q}\Big) k^{-\alpha q}\leq C |t_1-t_0|^\frac{q}{2}.
\end{align*}
We take the intersection of all these full measure sets indexed by rational $\alpha$, rational $p$ and $p=\infty$. This set is thus measure and contained in $\bigcap_{\alphap}\sC(\R;\XX^\alphap)$. By the completeness of $\cW$, we conclude that $\bigcap_{\alphap}\sC(\R;\XX^\alphap)\in\cW$ and $\sW\big(\bigcap_{\alphap}\sC(\R;\XX^\alphap)\big)=1$. On the other hand, the law of the iterated logarithm is a classical result implying $\widetilde\WW\in\cW$ and $\sW(\widetilde\WW)=1$. Checking the shift-invariance of $\widetilde\WW$ is straightforward, and  the proof is completed.
\end{proof}

\subsection{Existence and uniqueness}

Henceforth, we fix $\FF_0$ from Lemma~\ref{Lemma:f_property} and $\WW_0$ from Lemma~\ref{Lemma:W_local_Holder}.
We approach the infinite-dimensional polymer dynamics using Galerkin approximations.

\begin{Def}\label{def:galerkin}
For $(x,A,F,W)\in \R^\N\times \sL(\R^\N)\times \FF\times \WW$, a sequence $(X^\n)_{n\in\N}$ is called a sequence of Galerkin approximations of $\cS(x,A,F,W)$ if, for all $n\in\N$, $X^\n\in\cS(x,A^\n,F^\n,W^\n)$ where 
\begin{gather}\label{eq:A^nF^nW^n}
    A^\n = \big(\Ind{k\leq n}A_k\big)_{k\in\N}\ ,\quad
    F^\n = \big(\Ind{k\leq n}F_k\big)_{k\in\N}\ ,\quad
    W^\n = \big(\Ind{k\leq n}W_k\big)_{k\in\N}\ .
\end{gather}
\end{Def}

Note that if $X^\n\in\cS(x,A^\n,F^\n,W^\n)$, then $X^\n$ satisfies.
\begin{align}
    \d \Xn_k(t) & = \big(A_k\Xn(t) + f_k(\Xn_k(t))\big)\d t +\sigma \d W_k(t),\quad k=1,\dots,n, \label{eq:Galerkin_approx}\\
    \Xn_k(0) & =x_k,\quad k=1,\dots,n,\nonumber \\
    \Xn_k(t) & = x_n, \quad t\in \R,\quad k> n\nonumber
\end{align}
The last display indicates that $\Xn$ is essentially an $n$-dimensional SDE. For every $F\in\FF_0$, each coordinate of $f$ has slow growth given in Lemma~\ref{Lemma:f_property}~\eqref{item:f_k_upper_bound} and is locally Lipschitz. Using this and the growth of sample paths given by the law of iterated logarithm stated in Lemma~\ref{Lemma:W_local_Holder}, we can apply the Picard--Lindel\"of theorem (a slight modification of \cite[Theorem 2.5]{teschl2012ordinary} is sufficient for our purpose) to obtain the existence and uniqueness of solutions to \eqref{eq:Galerkin_approx} for every $(F,W)\in\FF_0\times \WW_0$. Therefore, the following holds.

\begin{Lemma}\label{lemma:galerkin_exists}
For each $(x,A,F,W)\in \R^\N\times \sL(\R^\N)\times \FF_0\times \WW_0$, there is a unique sequence of Galerkin approximations of $\cS(x,A,F,W)$.
\end{Lemma}

\begin{Prop}[Convergence of Galerkin approximations]\label{prop:existence}
Suppose that $A$, $(\alphap)$, and $(\alphaprime)$ satisfy
\eqref{eq:(beta'-beta)p'>1} and \eqref{eq:A-bounded}.  Then, for each $(x,F,W)\in \XX^\alphap\times\FF_0\times\WW_0$, the sequence of Galerkin approximations of $\cS(x,A,F,W)$ converges in $\sC(\R;\XX^\alphaprime)$ to some $X\in \cS(x,A,F,W;\XX^\alphaprime)$.

\end{Prop}

\begin{proof}
For convenience, we write $\|\cdot\|=\|\cdot\|_\alphap$, $\XX= \XX^\alphap$, $\|\cdot\|'=\|\cdot\|_\alphaprime$ and $\XX'=\XX^\alphaprime$. We want to prove that the sequence of Galerkin approximations $(X^\n)_{n\in\N}$ of $\cS(x,A,F,W)$ is Cauchy in  $\sC(\R;\XX')$ using Lemma~\ref{Lemma:Cauchy} applied to $(\alphaprime)$. 

\smallskip

Step 1. The definition of norms in \eqref{eq:def_norms} implies
\begin{align}\label{eq:y_k<k^beta}
    |y_k|\leq \|y\|k^\alpha,\quad  y\in\XX.
\end{align}
Then, we get
\begin{align*}
    \|y\|'= \Bigg(\sum_{k\in\N}\bigg|\frac{y_k}{k^{\alpha'}}\bigg|^{p'}\Bigg)^\frac{1}{p'}\leq \|y\|\Bigg(\sum_{k\in\N}k^{(\alpha-\alpha')p'}\Bigg)^\frac{1}{p'}\leq C\|y\|,\quad  y\in\XX
\end{align*}
where in the last inequality we used~\eqref{eq:(beta'-beta)p'>1}.
Therefore, $\XX\subset\XX'$ and $\|x'\|\leq C$. Since the Galerkin approximations are essentially finite dimensional, we obtain
\begin{align*}
    X^\n\in\cS(x,A^\n,F^\n,W^\n)\cap\sC(\R;\XX)\cap\sC(\R;\XX'),
\end{align*}
for $A^\n,F^\n,W^\n$ given in \eqref{eq:A^nF^nW^n}. 

\smallskip

Step 2. Let us show that $(W^\n)_{n\in\N}$ is Cauchy in $\sC(\R;\XX')$. By Lemma~\ref{Lemma:W_local_Holder}, we have $W\in\sC(\R;\XX)$. 
By \eqref{eq:y_k<k^beta}, for each $T>0$, there is $C>0$ such that
\begin{align*}
    \sup_{n,\,k\in\N,\ s\in[-T,T]}|W_k(s)k^{-\alpha}|\leq C.
\end{align*}
Then, for $s\in[-T,T]$, we have
\begin{align*}
    \Big(\|W^\m(s)-W^\n(s)\|'\Big)^{p'}=\sum_{k=(m\wedge n)+1}^{m\vee n} \bigg|\frac{W_k(s)}{k^{\alpha'}}\bigg|^{p'}\leq C\sum_{k=(m\wedge n)+1}^{m\vee n} k^{(\alpha-\alpha')p'}.
\end{align*}
Hence, due to \eqref{eq:(beta'-beta)p'>1}, the sequence $(W^\n)_{n\in\N}$ is Cauchy in $\sC(-T,T;\XX')$ for each $T>0$ and thus Cauchy in $\sC(\R;\XX')$.

\smallskip

Step 3. Let $T>0$. We show that $\|X^\n\|_{\sC(-T,T;\XX)}$ is bounded uniformly in $n$. 
Lemma~\ref{Lemma:f_property}~\eqref{item:f_k_upper_bound} implies
\begin{align}\label{eq:||f(y)||}
    \|f^\n(y)\|\leq \|f(y)\| \leq C(1+\|y\|),\quad   y\in \XX.
\end{align}
From the definition of the norm in \eqref{eq:def_norms}, we can get
\begin{align}\label{eq:An<A}
    \sup_{n\in\N}\|A^\n\|_{\sL_\b(\XX)} \leq \|A\|_{\sL_\b(\XX)}<\infty.
\end{align}
The definition of $W^\n$ gives
\begin{align*}
    \sup_{n\in\N}\|W^\n\|_{\sC(\R;\XX)} \leq \|W\|_{\sC(\R;\XX)}<\infty.
\end{align*}
These along with \eqref{eq:Galerkin_approx} yield
\begin{align}\label{eq:X^n<}
    \|\Xn(t)\| \leq \|x\| + \int_0^{|t|} C\big(1+\|\Xn(s)\|\big)\d s + \sigma \|W(t)\|.
\end{align}
Gronwall's inequality gives
\begin{align*}
    \|\Xn\|_{\sC(-T,T;\XX)}\leq \Big(CT+\|x\| + \sigma \|W(t)\|_{\sC(-T,T;\XX)}\Big)e^{CT}\leq C, \quad   n \in \N.
\end{align*}
For convenience, we rewrite the above as
\begin{align}\label{eq:sup_Xn}
    \sup_{n\in\N}\|\Xn\|_{\sC(-T,T;\XX)}\leq C.
\end{align}
\smallskip

Step 4. Let us verify condition \eqref{eq:A_diff}. The triangle inequality yields
\begin{align*}
    &\big\|A^{(m)}X^{(m)}(s)-A^\n X^\n(s)\big\|'
    \\
&
 \leq \big\|A^\m\big(X^{(m)}(s)-X^\n(s)\big)\big\|'+\big\|\big(A^{(m)}-A^\n\big) X^\n(s)\big\|'.
\end{align*}
The first term on the right is bounded by $C\|X^\m(s)-X^\n(s)\|'$ due to~\eqref{eq:An<A} with $\XX$ replaced by $\XX'$. Assuming $n\geq m$, the second term can be expanded as
\begin{align*}
    \Big(\big\|\big(A^{(m)}-A^\n\big) X^\n(s)\big\|'\Big)^{p'}=\sum_{k=m+1}^{n} \bigg|\frac{A_k X^\n(s) }{k^{\alpha'}}\bigg|^{p'}.
\end{align*}
Using \eqref{eq:sup_Xn}, $A\in \sL_\b(\XX)$ and \eqref{eq:y_k<k^beta}, we have
\begin{align*}
    \sup_{n, k\in\N, \ s\in[-T,T]}\bigg|\frac{A_kX^\n(s)}{k^\alpha}\bigg|\leq C.
\end{align*}
Setting $\delta_{m,n}= \big(\sum_{k=(m\wedge n)+1}^{m\vee n}(Ck^{\alpha-\alpha'})^{p'}\big)^\frac{1}{p'}$, we have $\lim_{m,n\to\infty}\delta_{m,n}=0$ due to \eqref{eq:(beta'-beta)p'>1}. Then, the above two displays imply
\begin{align*}
    \big\|\big(A^{(m)}-A^\n\big) X^\n(s)\big\|'\leq \delta_{m,n},\quad   m,n, \    s\in [-T,T].
\end{align*}
Hence, we conclude that
\begin{align*}
    \big\|A^{(m)}X^{(m)}(s)-A^\n X^\n(s)\big\|'\leq C\|X^\m(s)-X^\n(s)\|' + \delta_{m,n},
\end{align*}
for all $m,n\in\N$ and $s\in[-T,T]$,
verifying \eqref{eq:A_diff}.

\medskip

Step 5. We verify \eqref{eq:f_difference_condition}.
By \eqref{eq:sup_Xn}, there is $M>0$ such that $\|X^\n(t)\|\leq M$ for all $n\in\N$ and $t\in[-T,T]$. Let $y,y'\in\XX$ satisfy $\|y\|,\|y'\|\leq M$. This along this \eqref{eq:y_k<k^beta} yields
\begin{align}\label{eq:y,y'}
    |y_k|,\, |y'_k|\leq M k^\alpha,\quad  k\in\N.
\end{align}
The triangle inequality yields
\begin{align}\label{eq:fm(y)-fn(y')}
    \|f^\m(y)-f^\n(y')\|'\leq \|f^\m(y)-f^\n(y)\|'+\|f^\n(y)-f^\n(y')\|'.
\end{align}
Due to $p'<\infty$, by the definition of $f^\n$, \eqref{eq:y,y'} and Lemma~\ref{Lemma:f_property}~\eqref{item:f_k_upper_bound}, the first term on the right of \eqref{eq:fm(y)-fn(y')} can be bounded by
\begin{align*}
    \Bigg(\sum_{k=(m\wedge n)+1}^{m\vee n}|f_k(y_k)|^{p'}k^{-\alpha' p'}\Bigg)^\frac{1}{p'}\leq C\Bigg(\sum_{k=(m\wedge n)+1}^{m\vee n}\big(1+\log k +\log_+|y_k|\big)^{p'}k^{-\alpha' p'}\Bigg)^\frac{1}{p'}\\\leq C\Bigg(\sum_{k=m\wedge n}^{m\vee n}k^{-\alpha' p' +\eps}\Bigg)^\frac{1}{p'}
\end{align*}
for some $\eps$ satisfying $\eps \in (0,\alpha' p' -1)$. Set $\delta'_{m,n}$ to be the last term in the above display, and we have $\lim_{m,n\to\infty}\delta'_{m,n}=0$.

\smallskip

For the second term on the right of \eqref{eq:fm(y)-fn(y')}, we can apply Lemma~\ref{Lemma:f_property}~\eqref{item:f-f} to see that it is bounded by
\begin{align*}
    \|f(y)-f(y')\|'\leq C(N^{-\upsilon}+\|y-y'\|'\log N).
\end{align*}
Replacing $y$ by $X^\m(s)$ and $y'$ by $X^\n(s)$ in the above three displays yields \eqref{eq:f_difference_condition}.

\medskip

Step 6. Combining the results from previous steps, we can apply Lemma~\ref{Lemma:Cauchy} to see that $(X^\n)_{n\in\N}$ is Cauchy in $\sC(\R;\XX')$. By the completeness, the limit denoted by $X$ exists in $\sC(\R;\XX')$. In particular, Lemma~\ref{lemma:ptw_cvg} implies
\begin{align}\label{eq:ptw_X^n}
    \lim_{n\to\infty}X^\n_k(s)=X_k(s),\quad   k\in\N,\ s\in[-T,T].
\end{align}
Due to $A\in\sL_\b(\XX')$ and the same lemma, we also have $\lim_{n\to\infty}A_kX^\n(s)=A_kX(s)$ for each $k$ and $s$. Using the continuity of each $f_k$, we can pass to the limit for each $k$ in \eqref{eq:Galerkin_approx} to see that $X$ satisfies \eqref{eq:sys_int_eq}, namely, $X\in\cS(x,A,F,W)$. This completes our proof.
\end{proof}

\begin{Lemma}[Existence]Suppose that $(\alphap), (\alphaprime)$, and $A$ satisfy \eqref{eq:(beta'-beta)p'>1} and \eqref{eq:A-bounded}.
Then, for each $(x,F,W)\in \XX^\alphap\times\FF_0\times\WW_0$, 
\begin{align*}\cS(x,A,F,W;\XX^\alphap)  \neq \emptyset.
\end{align*}
In particular, the limit of the Galerkin approximations of $\cS(x,A,F,W)$ in $\sC(\R;\XX^\alphaprime)$ belongs to $\cS(x,A,F,W;\XX^\alphap)$.
\end{Lemma}

\begin{proof}
As before, we write $\|\cdot\|=\|\cdot\|_\alphap$, $\XX=\XX^\alphap$ and $\XX'=\XX^\alphaprime$. By Proposition~\ref{prop:existence}, there is a solution
$ X\in \cS(x,A,F,W;\XX')$
and there is a sequence Galerkin approximations $(X^\n)_{n\in\N}$ satisfying
\begin{align*}
    \lim_{n\to\infty}X^\n = X\quad \text{ in }\XX'.
\end{align*}
It remains to show that $X\in\sC(\R;\XX)$.

\smallskip
Step 1. Let us show that $X(t)\in\XX$ for all $t$.
Recall the estimate \eqref{eq:sup_Xn} and the pointwise convergence \eqref{eq:ptw_X^n}. If $p<\infty$, invoking Fatou's lemma, we can see that, for each $T>0$, there is $C>0$ such that
\begin{align}\label{eq:supX_betap}
    \sup_{t\in[-T,T]}\|X(t)\| \leq C.
\end{align}
If $p=\infty$, then \eqref{eq:sup_Xn} implies that, for each $T>0$, there is $C>0$ such that
\begin{align*}
    |X^\n_k(t)|\leq Ck^\alpha,\quad   n\in\N, \ k\in\N, \  t\in[-T,T].
\end{align*}
By the pointwise limit \eqref{eq:ptw_X^n}, we have
\begin{align*}
    |X_k(t)|\leq Ck^\alpha,\quad    k\in\N,\ t\in[-T,T],
\end{align*}
which also gives \eqref{eq:supX_betap}

\smallskip

Step 2. We show $t\mapsto X(t)$ is continuous in $\XX$. Due to $X\in\cS(x,A,F,W)$ and Definition~\ref{def:solution}, we know that $X$ satisfies \eqref{eq:sys_int_eq}. Then, by $A\in\sL_\b(\XX)$, we have, for $t_0,t_1\in[-T,T]$ satisfying $t_0\leq t_1$,
\begin{align*}
    \|X(t_1)-X(t_0)\|  \leq \int_{t_0}^{t_1}\Big(C\|X(s)\|  + \|f(X(s))\| \Big)\d s\\ + \sigma\|W(t_1)-W(t_0)\| .
\end{align*}
Along with \eqref{eq:||f(y)||} and \eqref{eq:supX_betap}, this display implies
\begin{align*}
    \|X(t_1)-X(t_0)\|  \leq C|t_1-t_0|+\sigma\|W(t_1)-W(t_0)\|.
\end{align*}
From this and Lemma~\ref{Lemma:W_local_Holder}, we can deduce $X\in \sC(-T,T;\XX)$ for any $T>0$. Therefore, $X\in \sC(\R;\XX)$.
\end{proof}

\begin{Lemma}[Uniqueness]For each $(\alphap)\in\setbetap$, and each $(x,A,F,W)\in \XX^\alphap\times \sL_\b(\XX^\alphap)\times\FF_0\times\WW_0$, the set
$\cS(x,A,F,W;\XX^\alphap)$
contains at most one element.
\end{Lemma}

\begin{proof}

Let $X,Y\in \cS(x,A,F,W;\XX^\alphap)$. Then for any $T>0$, there is $M>0$ such that
\begin{align*}
    \|X\|_{\sC(-T,T;\XX^\alphap)},\quad  \|Y\|_{\sC(-T,T;\XX^\alphap)}\leq M.
\end{align*}
Then, by Lemma~\ref{Lemma:f_property}~\eqref{item:f-f}, there are constants $K,\upsilon>0$  such that
\begin{align*}
    \|f(X(s))-f(Y(s))\|_\alphap\leq K\big(N^{-\upsilon}+\|x-y\|_\alphap \log N\big),\quad   s\in[-T,T].
\end{align*}
This allows us to apply Lemma~\ref{Lemma:cvg_sol} by setting
\begin{gather*}
    x^n=x,\quad W^n =W,\quad A^n=A,\quad F^n = F,\quad   n\in \N\cup\{\infty\};\\
    \delta_n=\delta'_n=0,\quad X^n = X, \quad  n\in\N; \qquad X^\infty = Y,
\end{gather*}
and conclude that $X=Y$.
\end{proof}

\begin{Cor}\label{cor:galerkin}
Let   $(\alphap)$, $(\alphaprime)$, and  $A$ satisfy~\eqref{eq:(beta'-beta)p'>1} and~\eqref{eq:A-bounded}. 
 For each $(x,F,W)\in \XX^\alphap\times\FF_0\times\WW_0$, the set $\cS(x,A,F,W;\XX^\alphap)$  is a singleton and its unique element is the limit of the Galerkin approximations of $\cS(x,A,F,W)$ in $\sC(\R;\XX^\alphaprime)$.
\end{Cor}

\subsection{Proof of Theorem~\ref{thm:RDS}}

The existence and uniqueness results in the previous subsection guarantees that the map $\Phi(\cdot,x,F,W)$ sending $(x,F,W)$ to the unique solution in
$ \cS(x,A,F,W;\XX^\alphap)$
is well-defined. Hence, \eqref{item:RDS_SDE} is verified. The cocycle property \eqref{item:RDS_cocycle} is valid due to the uniqueness of solutions and the invariance of $\WW_0$ under $\tshift$, see Lemma~\ref{Lemma:W_local_Holder}. Corollary~\ref{cor:galerkin} and Lemma~\ref{lemma:ptw_cvg} imply that the sequence of Galerkin approximations converges to the solution coordinatewise. The Galerkin approximations are obviously  progressively measurable. Therefore, passing to the limit, we can verify the measurability property \eqref{item:RDS_meas}.
Lastly, the lemma below implies the continuous dependence property \eqref{item:RDS_cont_dep}. 

\begin{Lemma}\label{Lemma:continuous_dependence}
Let $(\alpha,p)\in\setbetap$ and suppose $A\in \sL_\b(\XX^\alphap)$, $F\in\FF_0$,
\begin{gather*}
    X^n\in \cS(x^n, A, F, W^n;\XX^\alphap), \quad  n\in\N\cup\{\infty\},\\
    \lim_{n\to\infty}x^n = x^\infty \text{ in }\XX^\alphap,\qquad \lim_{n\to\infty}W^n=W^\infty \text{ in } \sC(\R;\XX^\alphap).
\end{gather*}
Then, it holds that $\lim_{n\to\infty}X^n=X^\infty$ in $\sC(\R;\XX^\alphap)$.

\end{Lemma}

\begin{proof}[Proof of Lemma \ref{Lemma:continuous_dependence}]

In view of Lemma~\ref{Lemma:cvg_sol}, it suffices to verify that for each $T>0$, there are $K,\upsilon>0$ such that, for all $N\in\N$ and $s\in[-T,T]$, we have
\begin{align*}
    \|f(X^n(s))-f(X^\infty(s))\|_\alphap\leq K\big(N^{-\upsilon}+\|X^n(s)-X^\infty(s)\|_\alphap \log N\big).
\end{align*}
Using a  Gronwall inequality  argument similar to the one that we used to obtain~\eqref{eq:sup_Xn}, we can find  $M>0$ such that
\begin{align*}
    \|X^n\|_{\sC(-T,T;\XX^\alphap)}\leq M, \quad   n\in\N\cup\{\infty\}.
\end{align*}
The desired result immediately follows from this and Lemma~\ref{Lemma:f_property}~\eqref{item:f-f}. 
\end{proof}

\section{Dynamics of infinite polymers}\label{section:dyn_inf_poly}
We recall the discrete Laplacian $\Delta = (\Delta_k \cdot)_{k\in\N}$ given in \eqref{eq:laplace} and the Dirichlet boundary condition $X_0(t)\equiv 0$. It is immediate that $\Delta \in \sL(\R^\N)$. The next basic result can be checked directly using the definition of norms given in \eqref{eq:def_norms}.
\begin{Lemma}\label{lemma:lap_bdd}
It holds that $\Delta\in \bigcap_{(\alphap)\in[0,\infty)\times[1,\infty]}\sL_\b(\XX^\alphap)$.
\end{Lemma}

\begin{Rem}\label{Rem:SDE_Delta}
Due to Lemma~\ref{lemma:lap_bdd},  all results of Section~\ref{section:SDE} hold  for solutions of
 the following SDE which is a special case of \eqref{eq:gen_SDE} with $A=\Delta$:
\begin{align}\label{eq:SDE}
\begin{split}
\d X (t) &=  \big(\Delta X(t) + f(X(t))\big)\d t + \sigma \d W(t),\quad t\in\R,\\
X (0) &= x.
\end{split}
\end{align}
\end{Rem}

The random dynamical system $(\Phi^t_{F,W})$ generated by \eqref{eq:SDE} is called the polymer dynamics (we interpret elements of $\R^\N$ as polymers).
The goal of the section is to prove that  this system  is monotone and preserves slopes of polymers.

Recall that full measure sets $\FF_0\in \cF$ and $\WW_0\in\cW$ are described in Lemma~\ref{Lemma:f_property} and  Lemma~\ref{Lemma:W_local_Holder}, respectively. We consider solutions to \eqref{eq:SDE} in the space
\begin{align}\label{eq:def_LL}
    \LL=\XX^{1,\infty}
\end{align}
with norm $\|\cdot\|_\LL = \|\cdot\|_{1,\infty}$. The explicit formula for this norm is given in \eqref{eq:def_norms} for $\alpha =1$ and $p=\infty$. 

\smallskip

For $v_\pm\in \R$, we introduce the set
\begin{align*}S(v_-,v_+) = \Big\{x\in \R^\N: v_-\leq  \liminf_{k\to\infty}\frac{x_k}{k}\leq \limsup_{k\to\infty}\frac{x_k}{k}\leq v_+\Big\}.
\end{align*}
The set of polymers with the asymptotic slope $v\in\R$ is given by
\begin{align*}
    S(v) = S(v,v).
\end{align*}
The space $\LL$ is complete, and, for any choice of parameters $v,v_\pm$, sets $S(v_-,v_+)$ and $S(v)$ are closed in $\|\cdot\|_{\LL}$, and hence also complete. In addition, $S(v)$ (for any $v$) is separable under $\|\cdot\|_{\LL}$.

\begin{Prop}\label{prop:invariance_slope}
Let real numbers $v_-$ and $v_+$ satisfy $v_-\leq v_+$. For each $(x,F,W)\in S(v_-,v_+)\times \FF_0\times \WW_0$, the set
$ \cS(x,\Delta,F,W;\LL)$
contains exactly one element $\Phi^\cdot_{F,W}x= \Phi(\cdot,x,F,W)$, which satisfies
\begin{align*}
    \Phi^t_{F,W}x \in S(v_-,v_+),\quad  t\geq 0.
\end{align*}
In particular, we have $\Phi^\cdot_{F,W}x|_{[0,\infty)}\in \sC\big(0,\infty; S(v_-,v_+)\big)$.
\end{Prop}

Here, the LU metric on $\sC\big(0,\infty; S(v_-,v_+)\big)$ is defined similarly to \eqref{eq:metric_cont}, and $S(v_-,v_+)$ is endowed with the metric induced by $\|\cdot\|_\LL$.

A special case of  Proposition~\ref{prop:invariance_slope} for $v_-=v_+$ is the forward invariance of slope:
\begin{Cor} 
If $x\in S(v)$ for some $v\in\R$ and  $(F,W)\in \FF_0\times \WW_0$, then the unique solution $(\Phi^t_{F,W}x)_{t\in\R}$ satisfies
$\Phi^t_{F,W}x\in S(v)$ for all $t\ge0$.
\end{Cor}

\subsection{Monotonicity}\label{sec:monotonicity}

We introduce the following partial order. For $x,y \in \R^\N$, we write
\begin{align*}x \preceq y,\quad\text{if and only if}\quad x_k\leq y_k,\   k\in\N,
\end{align*}
and $x \succeq y$ provided $y\preceq x$. We also write
\begin{align*}x \prec y,\quad\text{if and only if}\quad x_k< y_k,\   k\in\N,
\end{align*}
and $x \succ y$ provided $y\prec x$.

\begin{Lemma}[Monotonicity]\label{Lemma:monotonicity}
Suppose
\begin{gather*}
    X\in \cS(x,\Delta, F,W;\XX^\alphap),\qquad
    Y\in \cS(y,\Delta, F,W;\XX^\alphap),
\end{gather*}
with $(\alphap)\in\setbetap$, $x,y \in \XX^\alphap$, $F\in\FF_0$ and $W\in\WW_0$. If $x\preceq y$, then $X(t)\preceq Y(t)$ for all $t\geq 0$.

\end{Lemma}

\begin{proof}[Proof of Lemma \ref{Lemma:monotonicity}]

It suffices to show that $X(t)\preceq Y(t)$ for all $t\geq 0$ when $x\prec y$. To extend this to the general case $x\preceq y$, we can find a sequence $(X^n)_{n\in\N}$ in $\XX^\alphap$ such that $\lim_{n\to\infty} x^n =x \text{ in }\XX^\alphap$ and $x^n\prec y$ for all $n\in\N$ and then use the continuity  established in Theorem~\ref{thm:RDS} to take limits in  $X^n(t)\preceq Y(t),$  $t\geq 0.$

\smallskip

Henceforth, we assume
\begin{align}\label{eq:x<y}
    x\prec y.
\end{align}
Consider Galerkin approximations $X^{(n)}$ and $Y^{(n)}$ of $\cS(x,\Delta,F,W)$ and $\cS(y,\Delta,F,W)$, respectively, given in Definition~\ref{def:galerkin}. Corollary~\ref{cor:galerkin} implies that
\begin{align*}
    \lim_{n\to\infty}X^\n = X, \quad \lim_{n\to\infty}Y^\n = Y,\quad \text{ in }\sC(\R;\XX^\alphaprime),
\end{align*}
for some $(\alphaprime)\in\setbetap$ with $p'<\infty$. 
Using Lemma~\ref{lemma:ptw_cvg}, we obtain
\begin{align*}
    \lim_{n\to\infty}X^\n_k(t)=X_k(t),\quad \lim_{n\to\infty}Y^\n_k(t)=Y_k(t),\quad   t\in\R,\ k\in\N.
\end{align*}
Hence, it suffices to prove
\begin{align}\label{eq:X^(n)<Y^(n)}
    Y^{(n)}_k(t)>X^{(n)}_k(t), \quad   k\in\N,\ t\geq 0.
\end{align}
To this end, we set $Z(t) = Y^{(n)}(t)-X^{(n)}(t)$. Let
\begin{align*}
    \tau = \inf \Big\{t\geq 0: \min_{k=1,\dots, n}Z_k(t)\leq 0\Big\}.
\end{align*}
We want to show that $\tau =\infty$. Suppose $\tau<\infty$. Then  the continuity of individual coordinates of $Z$ implies that  there is $m$ such that
\begin{align}\label{eq:Z_m(tau)=0}
    Z_m(\tau)=0.
\end{align}
We claim 
\begin{align}\label{eq:Delta_mZ(tau)=0}
    \Delta_m Z(\tau)=0.
\end{align}
Indeed, suppose $\Delta_m Z(\tau)\neq 0$. Since $Z_m(\tau)=0$ and $Z_k(s)\ge 0$ for all $k$ and $s\le \tau$, by the definition of $\Delta_m$, we must have $\Delta_mZ(\tau)>0$. Note that, for all $t_1\geq t_0\geq 0$,
\begin{equation}
\label{eq:increment-of-Z}
Z_m(t_1)-Z_m(t_0)=\int_{t_0}^{t_1}\Big(\Delta_mZ(s)+f_m(Y_m(s))-f_m(X_m(s))\Big)\d s. 
\end{equation}
Using the continuity of $Z_m$ and the continuity of 
$f_m(Y_m(\cdot))-f_m(X_m(\cdot))$ which is equal to zero at $s=\tau$, we obtain from~\eqref{eq:increment-of-Z}
that $Z_m(\tau)-Z_m(\tau-r)>0$ for sufficiently small $r>0$.
However, the definition of $\tau$ implies that $Z_m(\tau-r)>0$ and thus $Z_m(\tau)>0$, which contradicts \eqref{eq:Z_m(tau)=0}. Therefore, we must have \eqref{eq:Delta_mZ(tau)=0}. 

\smallskip

 The definition of $\Delta_m$ and  \eqref{eq:Z_m(tau)=0}, \eqref{eq:Delta_mZ(tau)=0} imply $Z_{m-1}(\tau) = Z_{m+1}(\tau)=0$. Repeating the argument in the above paragraphs, we get $\Delta_{m-1} Z(\tau)=\Delta_{m+1} Z(\tau)$. Iterating in this fashion, we obtain $Z_k(\tau)=0$ for all $k=1,2,\dots,n,n+1$, which is impossible because \eqref{eq:Galerkin_approx} and \eqref{eq:x<y} imply
\begin{align*}
    Z_{n+1}(\tau) = Y^\n_{n+1}(\tau) - X^\n_{n+1}(\tau) = y_{n+1}-x_{n+1}>0.
\end{align*}
By contradiction we conclude that $\tau=\infty$. The definition of $\tau$ ensures that $Y^\n_k(t)>X^\n_k(t)$ for $k=1,2,\dots,n$ and $t\geq 0$. Now, \eqref{eq:X^(n)<Y^(n)} follows from this, \eqref{eq:x<y} and that $Y^\n_k(t)=y_k$, $X^\n_k(t)=x_k$ for all $k>n$ and all $t$.
\end{proof}

\subsection{Invariance of slopes}

To prove Proposition~\ref{prop:invariance_slope}, we will need solutions to the (homogeneous discrete) heat equation
\begin{align}\label{eq:discrete_heat}
    \dot x = \Delta x.
\end{align}
This equation can be viewed as a special case of \eqref{eq:SDE} with zero potential $F\equiv0$ and zero noise $W\equiv 0$.

\begin{Def}\label{def:heat_eqn}
If $x\in \XX^\alphap$ for some $(\alphap)\in\setbetap$, we denote by $\sS^\cdot x : \R\to \XX^\alphap$ the unique element in $\cS(x,\Delta,0,0;\XX^\alphap)$, i.e., the solution of the homogeneous discrete heat equation \eqref{eq:discrete_heat}.
\end{Def}

\begin{Rem}\label{rem:heat_eq}
By Remark~\ref{remark:F=0_W=0} and Remark~\ref{Rem:SDE_Delta}, $\sS$ is well-defined and, due to \eqref{eq:def_LL}, belongs to  $\sC(\R, \LL)$.
\end{Rem}

Recalling the definition of norms in \eqref{eq:def_norms}, we introduce
\begin{align}\label{eq:*norm}
    \|\cdot\|_* = \|\cdot\|_{\frac{3}{4},2}.
\end{align}
We need the following result, which will be used again later.

\begin{Lemma}\label{lemma:comp_in_*}
For each $F\in\FF_0$, $T>0$, $M>0$ there is a finite constant $C>0$ such that 
\begin{align*}
    \sup_{t\in[0,T]}\|\Phi^t_{F,W}x-\sS^tz\|_*<C
\end{align*}
holds for all $x, z\in\LL$ satisfying $\|x\|_\LL\leq M$ and $\|x-z\|_*\leq M$, and all $W\in\WW_0$ satisfying $\sup_{t\in[0,T]}\|W(t)\|_*\leq M$.
\end{Lemma}

\begin{proof}[Proof of Lemma~\ref{lemma:comp_in_*}]
We denote by $C$ a finite positive constant depending only on $T$ and $M$; it may change from line to line. Let $x, z, W$ satisfy the conditions of the lemma.
For brevity, we write $X(t) = \Phi^t_{F,W}x$ and $Z(t) = \sS^tz$.
By Corollary~\ref{cor:galerkin} applied to $(\alphap )= (\frac{3}{4}, 2)$ (see \eqref{eq:*norm}), there are $(\alphaprime)\in\Pi$ and Galerkin approximations $(X^\n)$ and $(Z^\n)$ such that
\begin{align}\label{eq:X^ntoX_z^ntoz}
    \lim_{n\to\infty} X^\n = X,\quad \lim_{n\to\infty} Z^\n = Z,\quad \text{ in }\sC(\R;\XX^\alphaprime).
\end{align}

Due to~\eqref{eq:Galerkin_approx}, $\|x-z\|_*\leq M$, and the continuity of each coordinate of $X^\n-Z^\n$, the function
\begin{align*}
    t\mapsto \|X^\n(t)-Z^\n(t)\|_*
\end{align*}
is finite and continuous. Subtracting the equations for~$X^{n}$ and~$Z^{(n)}$ from one another and using the boundedness of $\Delta$ (see Lemma~\ref{lemma:lap_bdd}) we obtain, for $t\in \R$,
\begin{align*}
    \|X^\n(t)-Z^\n(t)\|_*\leq \|x-z\|_* + \int_0^{|t|} \Big(C\|X^\n(s)-Z^\n(t)\|_* + \|f^\n(X^\n(s))\|_* \Big)\d s \\
    + \|W^\n(t)\|_*.
\end{align*}
The norms $\|X^{(n)}(\cdot)\|_\LL$ satisfy an inequality similar to 
 \eqref{eq:X^n<}.
Moreover, using the definition of the norm in \eqref{eq:def_norms}, we can see that $\|W(t)\|_\LL\leq \|W(t)\|_*$ for all $t$. 
Gronwall's inequality implies now that $|X^\n_k(s)|\leq C k$ uniformly in $n$, $k$ and $s\in [0,T]$. This along with the definition of $f^\n$ in \eqref{eq:A^nF^nW^n} and Lemma~\ref{Lemma:f_property}~\eqref{item:f_k_upper_bound} yields
\begin{align*}
    \sup_{s\in[0,T],\ n\in\N}\|f^\n(X^\n(s))\|_* \leq C.
\end{align*}
Lastly, the definition of $W^\n$ in \eqref{eq:A^nF^nW^n} implies
\begin{align*}
    \sup_{t\in[0,T],\ n\in\N}\|W^\n(t)\|_*\leq M.
\end{align*}
Combining these estimates and applying Gronwall's inequality we obtain
\begin{align*}
    \sup_{t\in[0,T],\ n\in\N}\|X^\n(t)-Z^\n(t)\|_* \leq C.
\end{align*}
Lemma~\ref{lemma:ptw_cvg} and \eqref{eq:X^ntoX_z^ntoz} imply the coordinatewise convergence of $X^\n(t)- Z^\n(t)$ at each $t\in[-T,T]$. Invoking Fatou's lemma, we deduce the desired result from the above display.
\end{proof}

It remains to prove Proposition~\ref{prop:invariance_slope}.

\begin{proof}[Proof of Proposition~\ref{prop:invariance_slope}] Theorem~\ref{thm:RDS} gives the existence and uniqueness of 
$X=\Phi^\cdot_{F,W}x\in\cS(x,\Delta,F,W;\LL)$. We want to show that $X(t)\in S(v_-,v_+)$ for all $t\geq 0$. 
The idea is to compare $X$ with solutions of the homogeneous heat equation \eqref{eq:discrete_heat} 
and use Lemma~\ref{Lemma:monotonicity}. 

\smallskip

Due to $x\in S(v_-,v_+)$, for each $\eps>0$, we can find $m_\eps\in\N$ satisfying
\begin{align}\label{eq:slope_preserve_choice_m}
    x_k\in \big(k(v_--\eps), k(v_++\eps)\big), \quad k\geq m_\eps.
\end{align}
We set
\begin{align}\label{eq:def_z_k}
    z_k =
    \begin{cases}
    x_k , & \quad k \geq m_\eps,\\
    k(\tfrac{1}{2}v_-+\tfrac{1}{2}v_+) , & \quad k <m_\eps .
    \end{cases}
\end{align}
This implies the existence of $C_\eps>0$ depending on $\eps$ such that $\|x-z\|_*<C_\eps$. Applying Lemma~\ref{lemma:comp_in_*}, we obtain
\begin{align}\label{eq:slope_preserve_Z-X}
    \|X(t)-\sS^tz\|_*\leq C_{\eps,t} ,\quad   n, k \in \N,
\end{align}
where $C_{\eps,t}$ depends on $\eps$ and $t$.

\medskip

We compare $\sS^\cdot z$ with straight rays. Let $y^+, y^- \in \LL$ be given by
\begin{align*}
    y^\pm_k =k (v_\pm\pm \eps),\quad   k \in \N.
\end{align*}
By \eqref{eq:slope_preserve_choice_m} and \eqref{eq:def_z_k}, we have $y^- \preceq z\preceq y^+$. Lemma~\ref{Lemma:monotonicity} implies
\begin{align*}
    \sS^ty^-\preceq \sS^tz\preceq \sS^ty^+,\quad  t\geq 0.
\end{align*}
Using Remark~\ref{rem:heat_eq}, we can verify $\sS^ty^\pm = y^\pm$ for all $t\in\R$ (in other words, $y^\pm$ is a stationary solution of the heat equation). Therefore, the above gives
\begin{align}\label{eq:Z_k(t)_in_eps_strip}
    (\sS^tz)_k\in(k(v_--\eps),k(v_++\eps)), \quad   k\in \N,\ t\geq 0.
\end{align}

\medskip

To finish the proof, we combine \eqref{eq:slope_preserve_Z-X} with \eqref{eq:Z_k(t)_in_eps_strip} to see that, for every $t\geq 0$ and $\eps>0$, 
\begin{align*}
    \frac{X_k(t)}{k}\in \big(v_--\eps-C_{\eps,t}k^{-\frac{1}{4}},\ v_++\eps+C_{\eps,t}k^{-\frac{1}{4}}\big), \quad   k\in\N.
\end{align*}
Hence, we have $X(t)\in  S(v_--\eps,v_++\eps)$ for all $t\geq 0$ and all $\eps>0$. Sending $\eps\to0$, we conclude $X(t)\in S(v_-,v_+)$ for all $t\geq 0$.
\end{proof}

\section{Infinite volume polymer measures}\label{section:IVPM}

In this section, we introduce infinite volume polymer measures (IVPMs) following~\cite{Bakhtin-Li:MR3911894} and obtain first results on their
connection to the polymer dynamics studied in the previous sections. We also obtain several results concerning the extreme decompositions of these measures complementing the results of ~\cite{Bakhtin-Li:MR3911894}.

 Here we are mostly interested in IVPMs with a fixed endpoint which we place, without loss of generality, at $0\in\R$. 
 
 For $x\in \R^\N$, we introduce projections $x_{\leq n}=(x_1,x_2,\dots, x_n)$ and $x_{\geq n}=(x_{n},x_{n+1},\dots)$.

The following definition is the standard DLR (Dobrushin--Lanford--Ruelle) condition stating that conditional distributions for an infinite volume Gibbs measure conditioned on the configuration outside a finite volume are given by the classical Boltzmann--Gibbs definition.

\begin{Def}
For each $F\in\FF$, $\invtemp\in(0,\infty)$ we denote by $\Pc_{F,\upbeta}$ the collection of probability measures $\mu$ on $(\R^\N,\cB(\R^\N))$ satisfying the following property: for every $n\in\N$, there is a probability measure $\nu$ on $\R^{\{n+1,n+2,n+3,\,\dots\}}$ such that
\begin{align}\label{eq:disintegration}
    \mu(\d x)=\rho(x_{n+1}; \d x_{\leq n})\nu(\d x_{\geq n+1}),
\end{align}
where
\begin{gather}
    \rho(x_{n+1};\d x_{\leq n}) = \frac{e^{-\invtemp E_n(x_{n+1}; x_{\leq n})}}{Z_n(x_{n+1})}\,\d x_{\leq n}\ ,\nonumber\\
    Z_n(x_{n+1}) = \int_{\R^n} e^{-\invtemp E_n(x_{n+1};x_{\leq n})} \d x_{\leq n}\ ,\nonumber\\
    E_n(x_{n+1};x_{\leq n}) = \frac{1}{2}\sum_{k=0}^{n}(x_{k+1}-x_k)^2+\sum_{k=1}^{n}F_k(x_k)\ .\label{eq:def_E_n}
\end{gather}

\end{Def}

The elements in $\Pc_{F,\invtemp}$ are IVPMs in the random environment $F$ at inverse temperature~$\invtemp$. For $v\in\R$, we denote by $\Pc_{F,\invtemp}(v)$ the collection of measures in $\Pc_{F,\invtemp}$ which are concentrated on $S(v)$. We also set $\Pc_{F,\invtemp}(*)=\cup_{v\in\R}\Pc_{F,\invtemp}(v)$. In other words, under each IVPM in $\Pc_{F,\invtemp}(*)$, the asymptotic
slope is well-defined and constant.

The existence and uniqueness result from \cite[Theorem 4.2]{Bakhtin-Li:MR3911894} can be stated as follows:

\begin{Th}\label{thm:thermodynamic-limit}
 Let $\invtemp\in(0,\infty)$ and $v\in\R$. Then there is a full measure set $\FF_{\invtemp}(v)\in\Fc$ such that $\Pc_{\invtemp,F}(v)$ is a singleton 
$\{\mu_{\invtemp,v,F}\}$
 for every $F\in\FF_{\invtemp,v}$.
\end{Th}

From now on we fix an arbitrary  $\invtemp\in(0,\infty)$ and suppress the dependence of various objects like $\Pc_F=\Pc_{F,\invtemp}$,
 $\FF(v)= \FF_{\invtemp}(v)$, or $\mu_{v,F}=\mu_{\invtemp,v,F}$  on $\invtemp$. Sometimes it is also convenient to omit the dependence on $F$.

\subsection{Invariance of IVPMs}For a measurable space $(\XX,\cB)$, we denote by $\sM_\b(\XX,\cB)$ the space of bounded $\cB$-measurable real-valued functions on $\XX$. Recall the definition of $\cB^\alphap$ in \eqref{eq:def_B^alphap}. By Remark~\ref{Rem:SDE_Delta} and Theorem~\ref{thm:RDS}, for each $F\in \FF_0$ and $(\alphap)\in\setbetap$, the polymer dynamics $(\Phi^t_{F,W})$ defines a Markov semigroup $(P^t_F)_{t\geq 0}=(P^t)_{t\geq 0}$ on $\phi\in \sM_\b(\XX^\alphap,\cB^\alphap)$ given by
\begin{align*}
    P^t\phi(x) = \sW \phi\big( \Phi^t_{F,W}x\big),\quad   \phi\in \sM_\b(\XX^\alphap,\cB^\alphap),\  x \in\XX^\alphap.
\end{align*}
The goal of this subsection is to show the invariance of IVPMs with respect to $(P^t)_{t\geq 0}$. This is stated in the proposition below, which corresponds to part~\ref{it:Gibbs-is-invariant} in  Theorem~\ref{th:collect}.

\begin{Prop}\label{prop:invariance}
Let $F\in\FF_0$. If for some $(\alphap)\in\Pi$, $\mu\in \Pc_F$ satisfies 
\begin{equation}
\label{eq:mu-concentr-on-LL}
\mu(\XX^\alphap)=1, 
\end{equation}
then the restriction of $\mu$ to $(\XX^\alphap,\cB^\alphap)$ is invariant under the Markov process defined by the polymer dynamics. Equivalently,

\begin{align*}
    \int P^t \phi(x) \mu (\d x) = \int \phi(x) \mu(\d x),\quad t\geq 0,\ \phi\in \sM_\b(\XX^\alphap,\cB^\alphap).
\end{align*}

\end{Prop}

\begin{Rem}\label{rem:S(v)}
For any $v\in\R$, we have $S(v)\subset \LL=\XX^{1,\infty}$. Using the definition of $\cB^{1,\infty}$ in \eqref{eq:def_B^alphap} and the above proposition, we can see that if $\mu\in\Pc_F$ is supported on $S(v)$, then $\mu$ viewed as a measure on  $(S(v),\cB_v)$ is invariant under the polymer dynamics. Here $\cB_v$ is the $\sigma$-algebra induced by $\cB^{1,\infty}$ on $S(v)$, given by $\cB_v=\{E\cap S(v):E\in\cB(\R^\N)\}$. Now, by the separability of $S(v)$ and arguments similar to those in the proof of Lemma~\ref{lem:B^alphap}, we have that $\cB_v = \cB(S(v))$, the latter being, in accordance with our notation in Section~\ref{sec:funcion}, 
 the Borel $\sigma$-algebra generated by the metric on $S(v)$ induced by norm $\|\cdot\|_{\LL}$ and making $S(v)$ a Polish space.
\end{Rem}

\begin{proof}[Proof of Proposition~\ref{prop:invariance}]

We denote by $\Phi^t_{n,F,W}x$, $n\in\N$, the Galerkin approximation (see Definition~\ref{def:galerkin}) of $\Phi^t_{F,W}x$ for $x\in\XX^\alphap$. Invoking Lemma~\ref{lemma:galerkin_exists}, we can define a sequence of approximation semigroups $\big((P^t_n)_{t\geq 0}\big)_{n\in\N}$ by
\begin{align*}
    P_n^t\phi (x) =  \sW \phi(\Phi^t_{n,F,W}x) ,\quad   t\geq 0,\   \phi \in \sM_\b(\XX^\alphap,\cB^\alphap),\ x\in\XX^\alphap.
\end{align*}
By \eqref{eq:Galerkin_approx} with $A$ replaced by $\Delta$, the projection $(\Phi^t_{n,F,W}x)_{\leq n}$ satisfies
\begin{align}\label{eq:SDE_X_leq_n}
    \d  (\Phi^t_{n,F,W}x)_{\leq n} = - \nabla_{x_{\leq n}}E_n\big(x_{n+1};(\Phi^t_{n,F,W}x)_{\leq n}\big)\d t + \sigma\d W_{\leq n}(t)
\end{align}
where $\nabla_{x_{\leq n}}$ is the gradient taken with respect to $x_{\leq n}$. The following lemma describes the dissipative behavior of $E_n$.

\begin{Lemma}
\label{lem:Lyapunov-function}
For each $F\in\FF_0$, each $n\in\N$ and each $x_{n+1}\in\R$, there are constants $C_1,C_2>0$ such that
\begin{align*}
   - y\cdot \nabla_y E_n(x_{n+1};y) \leq C_1 -C_2 |y|^2,\quad   y\in\R^n.
\end{align*}
\end{Lemma}
\begin{proof}
Recalling the formula \eqref{eq:def_E_n}, we can rewrite $\nabla E_n(x_{n+1};\cdot)$ as 
\begin{align}\label{eq:grad_E_n}
    -\nabla_y E_n(x_{n+1};y)= - A y + b(x_{n+1}) + \big(f_k(y_k)\big)_{k=1}^n\ ,\quad   y \in \R^n,
\end{align}
where $b(x_{n+1})= (0,0,\dots,0,x_{n+1})$ and $A$ is an $n\times n$ matrix whose diagonal entries are $2$, super-diagonal and sub-diagonal entries are $-1$, and the rest are $0$. For example, when $n=4$, we have
\begin{align*}
    A = \begin{pmatrix}
2 & -1 & 0 & 0\\
-1 & 2 & -1 & 0 \\
0 & -1 & 2 & -1  \\
0  & 0 & -1 & 2
\end{pmatrix}.
\end{align*}
We claim that $A$ is positive definite. Setting $\theta = \frac{\pi}{n+1}$ and using classical trigonometric identities, we can verify that $\lambda_m=2-2\cos(m\theta)$ is an eigenvalue of $A$ with the eigenvector $v_m= (\sin(m\theta),\sin(2m\theta),\dots, \sin(nm\theta))$ for each $m=1,2,\dots,n$. This implies $y\cdot Ay \geq 2(1-\cos(\frac{1}{n+1}\pi))|y|^2$,
for all $y\in\R^n$. Plugging this into \eqref{eq:grad_E_n} and using Lemma~\ref{Lemma:f_property}~\eqref{item:f_k_upper_bound} to control $f_k$, we obtain the desired result.
\end{proof}

Continuing the proof of Proposition~\ref{prop:invariance}, let us fix $F\in\FF_0$. It is classical (c.f. \cite[Theorem 2.1]{roberts1996exponential}) that the condition stated 
in Lemma~\ref{lem:Lyapunov-function} is sufficient to guarantee that $\rho(x_{n+1},\cdot)$ is an invariant measure for dynamics generated by \eqref{eq:SDE_X_leq_n}. Since $(\Phi^t_{n,F,W}x)_{\geq n+1}= x_{\geq n+1}$ for $t\geq 0$ (see \eqref{eq:Galerkin_approx}), we obtain
\begin{align*}
    \int \sW\phi \big(\Phi^t_{n,F,W}x\big) \rho(x_{n+1},\d x_{\leq n}) = \int \phi(x)\rho(x_{n+1},\d x_{\leq n}),\quad  \phi \in \sM_\b(\XX^\alphap,\cB^\alphap).
\end{align*}
and due to \eqref{eq:disintegration}
\begin{align}\label{eq:invar_X^n}
    \int \sW\phi \big(\Phi^t_{n,F,W}x\big)\,\mu(\d x) = \int \phi\, \d \mu,\quad  \phi \in \sM_\b(\XX^\alphap,\cB^\alphap).
\end{align}

Then, we want to take the limit on the left-hand side of \eqref{eq:invar_X^n}. Let us pick any $(\alphaprime)$ satisfying
\eqref{eq:(beta'-beta)p'>1}.  Then,  Corollary~\ref{cor:galerkin} implies
 that
$\Phi^\cdot_{n,F,W}x$
converges to $\Phi^\cdot_{F,W}x$ in $\sC(\R;\XX^\alphaprime)$ for every $x\in \XX^\alphap$  and every $W\in\WW_0$. 

Fixing an arbitrary $t>0$ and recalling~\eqref{eq:mu-concentr-on-LL}, we obtain by Lemma~\ref{lemma:ptw_cvg}  that $\Phi^t_{n,F,W}x$
converges to $\Phi^t_{F,W}x$ coordinatewise (i.e., in the topology on $\XX^{\alphap}$ induced by the product topology on $\R^\N$) for $\mu$-a.e.~$x$ and every $W\in\WW_0$. Hence, we can send $n\to\infty$ in \eqref{eq:invar_X^n} for any bounded function $\phi:\XX^\alphap \to\R$ continuous in that topology to see that
\begin{align*}
    \int \sW\phi \big(\Phi^t_{F,W}x \big)\,\mu(\d x) = \int \phi\, \d \mu.
\end{align*}
Recalling the definition of $\cB^\alphap$ in \eqref{eq:def_B^alphap}, we can extend this identity to any $\phi \in \sM_\b(\XX^\alphap,\cB^\alphap)$,
which completes the proof of Proposition~\ref{prop:invariance}.
\end{proof}

\subsection{Extreme IVPMs have asymptotic slopes}\label{section:ext_IVPM_asymp_slope}

For $\mu\in \Pc_F$, we say that $\mu$ is extreme or that $\mu$ is an extreme point of $\Pc_F$, if $\mu$ is not equal to $ s \nu_1 + (1-s)\nu_2$ for any $s \in (0,1)$ and any distinct $\nu_1,\nu_2\in\Pc_F$.

The following result is a part  of Theorem~\ref{eq:extreme-measures-have-direction}.

\begin{Prop}\label{prop:extreme_IVPM_slope}
There is a full measure set $\FF^*\in \cF$ such that for every $F\in\FF^*$, all extreme points of $\Pc_F$ belong to $\Pc_F(*)$.
\end{Prop}
Before we prove this proposition, let us give a corollary.  Proposition~\ref{prop:extreme_IVPM_slope} and the extreme decomposition (\cite[Theorem (7.26)]{Georgii:MR956646}) applied to the Polish space $\R^\N$
imply that every IVPM is a mixture of IVPMs concentrated on paths with well-defined and constant asymptotic slope, for almost every realization of $F$:

\begin{Cor}
Let $F\in \FF^*$. If $\mu\in \Pc_F$, then there is a probability measure $\Pi_\mu$ on $\Pc_F(*)$ such that
\begin{align*}
    \mu = \int_{\Pc_F(*)} \nu \ \Pi_\mu(\d \nu).
\end{align*}
\end{Cor}

To prove the proposition, we need a lemma on the straightness of the polymer measure $\mu$. For $x\in \R^\N$ and $k\in \N$, we define 
\begin{align*}
    x^\out(k)= \{(m, x_m): m>k\}.
\end{align*}
For $(k,r)\in\N\times \R$ and $\eta>0$, we define
\begin{align*}
    \Co(k,r,\eta)=\{(m,r')\in\N\times\R:|r'/m - r/k|\leq \eta\}.
\end{align*}

\begin{Lemma}\label{lemma:straightness}
Let $\delta \in (0,\frac{1}{4})$. There is a full measure set $\FF^*\subset \FF_0$ and a constant $Q>0$ such that for every $F \in \FF^*$ and every $\mu\in\Pc_F$, we have $\mu(B)=1$, where
\begin{align*}
    B= \big\{x\in\R^\N: \exists n\in \N,\ \forall    k \geq n,\quad  x^{\out}(k)\subset \Co(k, x_k, Qk^{-\delta})\big\}.
\end{align*}
\end{Lemma}

\begin{proof}[Proof of Proposition~\ref{prop:extreme_IVPM_slope}]

Fix $\FF^*$ in the above lemma and let $\mu$ be an extreme point of $\Pc_F$ for some $F\in\FF^*$. By \cite[Theorem (7.7) (a)]{Georgii:MR956646}, we know that $\mu(A)\in\{0,1\}$ for every set $A$ in the natural tail $\sigma$-algebra on $\R^\N$. Since $x\mapsto\liminf_{k\to\infty}\frac{x_k}{k}$ and $x\mapsto\limsup_{k\to\infty}\frac{x_k}{k}$ are measurable with respect to the tail $\sigma$-algebra, they must be $\mu$-a.s.\ constant. Denote them by $v_-$ and $v_+$, respectively.  We want to show $v_-=v_+\in\R$. Let $B$ be the set given in Lemma \ref{lemma:straightness} and we have $\mu(B) =1$. For every $x\in B$, the definition of $B$ yields the existence of $n=n(x)$ such that
\begin{align}\label{eq:x/k}
    \Big|\frac{x_{k'}}{k'}-\frac{x_k}{k}\Big| \leq Qk^{-\delta},\quad  k'\geq k\geq n,
\end{align}
Hence, $(\frac{x_k}{k})_{k\in\N}$ is Cauchy in $\R$, and thus $ v_+=  v_-\in\R$.
\end{proof}

It remains to prove Lemma~\ref{lemma:straightness}. 
The next result is a consequence of \eqref{eq:disintegration} applied to \cite[Theorem 9.1]{Bakhtin-Li:MR3911894} which concerns finite volume polymer measures with arbitrary length and arbitrary terminal measures. 

\begin{Lemma}\label{Lemma:straightness_general}
For each $M\in\N$ and each $\upsilon \in (0,1-2\delta)$, there are $\FF^{(n)}_1 = \FF^{(n)}_1(\delta, M,\upsilon) \in\cF$ for $n\in \N$, and $Q=Q(\delta)>0$ such that
\begin{enumerate}
    \item for all $\gamma\in(0,1-4\delta)$, there is $n_1(M,\gamma)\in\N $ such that
    \begin{align*}
        \P(\FF^{(n)}_1)\geq 1 - e^{-n^\gamma}\quad\quad n\geq n_1(M,\gamma);
    \end{align*}
    \item on $\FF^{(n)}_1$, there is $n_2(M,\upsilon)\in\N$ such that 
    \begin{align*}
        \mu\big(B(n,M,\upsilon)\big)
      \geq 1-e^{-n^\upsilon}, \quad n\geq n_2(M,\upsilon),
    \end{align*}
   where
     \[
    B(n,M,\upsilon) =\Big\{x\in \R^\N:    x^\out(k)\subset \Co(k,x_k,Qk^{-\delta}),\, \forall k\geq n \text{ with }(k,x_k)\in \Co(M)\Big\},
     \]
     with $\Co(M)=\{(k,r)\in \N\times \R:|r|\leq M k\}$.
\end{enumerate}
\end{Lemma}

\begin{proof}[Proof of Lemma~\ref{lemma:straightness}]

To apply Lemma~\ref{Lemma:straightness_general}, let us 
set $n(M) = n_1(M,\gamma)\vee n_2(M,\upsilon)$. 
Taking
\begin{align*}
    \FF^*= \bigcap_{M\in\N}\bigcup_{n\geq n(M)}\FF^{(n)}_1(\delta, M,\upsilon)
\end{align*}
and noticing that
\begin{align*}
    B\supset \bigcap_{M\in \N}\bigcup_{n\geq n(M)}B(n,M,\upsilon)
\end{align*}
we obtain the desired result.
\end{proof}

Lastly, using Lemma~\ref{lemma:straightness}, we can prove the following result on the transversal fluctuations of polymer paths under IVPMs, which is a restatement of Theorem~\ref{thm:fluc_intro}.

\begin{Th}For every $\xi'>3/4$, there is a full measure set $\FF^*\in \cF$ and a constant $Q>0$ such that 
for every $F\in \FF^*$, $v\in\R$, $\mu\in \Pc_F(v)$, and $\mu$-a.e.\ polymer path $x$, we have
$|x_k-vk|<Qk^{\xi'}$ for sufficiently large $k$ (depending on $x$).
\end{Th}

\begin{proof} 
Let $\FF^*$ and $B$ be given by Lemma~\ref{lemma:straightness} for  $\delta = 1 -\xi'$. We have $\mu(B\cap S(v))=1$. For each $x \in B\cap S(v)$, we have \eqref{eq:x/k} for sufficiently large $k$ and $k'$ satisfying $k'\geq k$. Sending $k'\to \infty$, we complete the proof.
\end{proof}

\subsection{Ergodic IVPMs have asymptotic slopes} 
We recall the definition of $\LL$ in~\eqref{eq:def_LL}.
For  $A\in\cB^{1,\infty}$ and a probability measure $\mu$ on  $(\LL,\cB^{1,\infty})$, we say that $A$ is $\mu$-invariant under a Markov semigroup $(P^t)$ if $P^t\mathds{1}_A = \mathds{1}_A$, $\mu$-a.s. A $(P^t)$-invariant probability measure $\mu$ on  $(\LL,\cB^{1,\infty})$ is said to be ergodic if $\mu(A)\in\{0,1\}$ for every $\mu$-invariant $A$. The following result complements  Proposition~\ref{prop:extreme_IVPM_slope} in  Theorem~\ref{eq:extreme-measures-have-direction}.

\begin{Prop}
There is a full measure set $\FF^*\in \cF$ such that the following holds: if $F\in\FF^*$ and $\mu\in \Pc_F$ satisfying $\mu(\LL)=1$ is ergodic for the Markov semigroup generated by the polymer dynamics on $(\LL,\cB^{1,\infty})$, then there is $v\in\R$ such that $\mu(S(v))=1$.
\end{Prop}

\begin{proof}
Let~$\FF^*$ be given by Lemma~\ref{lemma:straightness} for some $\delta$. Let $F\in\FF^*$ and let $\mu\in\Pc_F$ be ergodic. For real $v_-\leq v_+$, Proposition~\ref{prop:invariance_slope} implies that $S(v_-,v_+)$ is $\mu$-invariant and thus $\mu(S(v_-,v_+))\in \{0,1\}$.
Therefore, $\cG=\sigma\{S(v_-,v_+): v_-\le v_+\}$, a $\sigma$-algebra on $\LL$, contains only sets of measure $0$ and $1$. The random variables
$
 \bar v_-=\liminf_{k\to \infty} \frac{x_k}{k},\quad  \bar v_+ =\limsup_{k\to \infty} \frac{x_k}{k}
$
are well defined on $\LL$, finite, and $\cG$-measurable. Therefore, there is a set $A$ with $\mu(A)=1$ such that $ \bar v_+$ and  $\bar v_-$ are  constants on $A$, and it remains to show that $\bar v_-=\bar v_+\in\R$.  We have $\mu(A\cap B) =1$, where $B$ is given in Lemma \ref{lemma:straightness}. For every $x\in A\cap B$, the  definition of $B$ ensures that there is $n=n(x)$ such that \eqref{eq:x/k} holds.
This implies that $(\frac{x_k}{k})_{k\in\N}$ is Cauchy in~$\R$. Since $x$ is also in $A$, we must have $\bar v_+= \bar v_-\in\R$.
\end{proof}

\section{Ordering by noise}
\label{sec:monotonization}

We recall the solution operator $\Phi^t_{F,W}$ defined 
in \eqref{eq:solution-map} and the relations $\preceq$, $\succeq$, $\prec$, $\succ$ introduced in Section~\ref{sec:monotonicity}.

The following is our monotonization result  which is the precise statement of part~\ref{it:ordering} in Theorem~\ref{th:collect}:

\begin{Th}[Ordering by noise]\label{th:monotonization}
For almost every realization of $F$, there is a full measure set $\WW(F)\in \cW$ 
with the following property: for every $u,v\in\R$ satisfying $u>v$, $x\in S(u)$ and $y\in S(v)$, there is a stopping time $\tau = \tau(x,y,F)\geq 0$ with respect to $(\cW_{[0,t]})_{t\geq 0}$ such that, for every $W\in\WW(F)$, we have $\tau <\infty$ and
\begin{align}\label{eq:monotonized}
    \Phi^t_{F,W}x \succeq \Phi^t_{F,W}y,\quad  t\geq\tau.
\end{align}

\end{Th}

The proof consists of three parts. We first prove the monotonization for solutions of the homogeneous discrete heat equation (see Definition~\ref{def:heat_eqn}). 

For $u,a \in \R$, we denote by $\rr^u+a$ the polymer chain given by $(\rr^u+a)_k = uk+a$ for all $k\in\N$.
It can be viewed as a ray with slope $u$ shifted by $a$.  In particular, we write $\rr^u=\rr^u+0$.
Recalling that we impose the Dirichlet boundary condition $X_0(t)\equiv 0$ on the polymer dynamics, we note that these rays are actually broken near the origin, so (except when $a=0$) they are not preserved by the homogeneous heat flow $\sS$ introduced in Definition~\ref{def:heat_eqn}.

\begin{Lemma}\label{lemma:monot_det}
For $a,b>0$, and $x=\rr^a-b$, there is $T\in(0,\infty)$ such that $\sS^T x\succ 0$.
\end{Lemma}

Then we use our assumption \eqref{eq:flat} on the presence of large regions where the potential is almost flat  together with  this lemma to prove that for almost every $F$, the monotonization happens with positive probability.

For $x,y,z\in\R^\N$, we write
\begin{align*}
y \in \llbracket x, z \rrbracket\quad  \text{if}\quad x\preceq y \preceq z.
\end{align*}

\begin{Lemma}\label{lemma:posi_prob_mt}
For every $u,v\in\R$ satisfying $u>v$, every $c>0$ and every $h\in(0,\frac{1}{2}(u-v))$, there is $T>0$ such that
\begin{align*}
    \sW\Big\{\Phi^T_{F,W}x \succ \Phi^T_{F,W}y,\quad   \forall x\in \llbracket \rr^{u-h}-c,\rr^{u+h}+c\rrbracket,\ \forall  y\in \llbracket \rr^{v-h}-c,\rr^{v+h}+c\rrbracket \Big\}>0
\end{align*}
holds for each $F\in\FF_1$, where $\FF_1$ is the intersection of $\FF_0$ given in Theorem~\ref{thm:RDS} with the full measure set on which~\eqref{eq:flat} holds. 
\end{Lemma}

\begin{Rem} This lemma is trivially true with $T=0$ for the case $c=0$. For $c>0$, the initial conditions $x$ and $y$ may fail to satisfy
$x\succeq y$.
\end{Rem}

The last part is the following recurrence result.
\begin{Lemma}\label{lemma:recur}
For every $u,v\in\R$ satisfying $u>v$, there is a full measure set $\FF(u,v)\in\cF$ such that the following holds. For every $x\in S(u)$, $y\in S(v)$, $h>0$ and $F\in\FF(u,v)$ there is a constant $c>0$ and a full measure set $\WW(x,y,h,F)\in\cW$ such that for every $W\in \WW(x,y,h,F)$ there is a sequence of finite times $t_j\to\infty$ satisfying
 \begin{align*}
 \Phi^{t_j}_{F,W}x\in \llbracket \rr^{u-h}-c,\rr^{u+h}+c\rrbracket, \quad   j\in\N,\\
\Phi^{t_j}_{F,W}y \in \llbracket \rr^{v-h}-c,\rr^{v+h}+c\rrbracket,\quad   j\in\N.
 \end{align*}
\end{Lemma}

\begin{proof}[Proof of Theorem~\ref{th:monotonization}]

Let $\FF_2=\FF_1\cap \big(\cap_{u'<v',\  u',v'\in \Q}\FF(u',v')\big),$ where the sets $\FF_1$ and $\FF(u',v')$ are taken from the above lemmas. This is a full measure set. Let us fix $F\in\FF_2$ and construct the set $\WW(F)$.

By the continuity of $t\mapsto \Phi^t_{F,W}x$ in \eqref{item:RDS_cont_dep} of Theorem~\ref{thm:RDS} and the standard time discretization argument (see the proof of \cite[Theorem 8.3]{rogers1994diffusions}), we can deduce that the polymer dynamics defines a strong Markov process. 
The continuity also ensures that $\tau(x,y,F)$, the first nonnegative time for $\Phi^t_{F,W}x\succeq \Phi^t_{F,W}y$ to hold, is a stopping time.
Along with Lemma~\ref{lemma:posi_prob_mt} and Lemma~\ref{lemma:recur} by setting $h=1$, this ensures that for every
$u'>v'$, every  $x'\in S(u')$ and $y'\in S(v')$
there is a full measure set $\widetilde\WW(x',y',F)\in \cW$ on which $\tau(x',y',F)<\infty$.

Let
\begin{align*}
    \WW(F)=\WW_0\cap \bigcap_{\substack{u'>v'\\ u',v'\in\Q}}\bigcap_{a',b'\in\Q} \widetilde\WW\left(\rr^{u'}+a',\rr^{v'}+b',F\right) .
\end{align*}
Now, for arbitrary real $u>v$, $x\in S(u)$ and $y\in S(v)$, we set $\tau = \tau(x,y,F)$. We fix some $u',v',a',b'\in\Q$, and set $x'=\rr^{u'}+a'$ and $y'=\rr^{v'}+b'$ to ensure $u>u'>v'>v$, $x\succeq x'$ and $y'\succeq y$. By Lemma~\ref{Lemma:monotonicity}, we have $\Phi^t_{F,W}x \succeq \Phi^t_{F,W}x'$ and $\Phi^t_{F,W}y' \succeq \Phi^t_{F,W}y$ for all $t\geq 0$ and $W\in\WW_0$. On $\WW(F)$, we have $\tau(x,y,F) \leq \tau(x',y',F)<\infty$ by the above construction. 
Lastly, \eqref{eq:monotonized} is a consequence of the definition of $\tau$ and Lemma~\ref{Lemma:monotonicity}.
\end{proof}

\subsection{Proof of Lemma~\ref{lemma:monot_det}}

We need the next lemma, which states that if the initial condition is convex, then the solution stays convex.

\begin{Lemma}\label{lemma:convexity}
Suppose $x\in\XX^\alphap$ for $(\alphap)\in\setbetap$ (see \eqref{eq:betap_set}). If $\Delta x \succeq 0$, then $\Delta \sS^tx\succeq 0$ for all $t\geq 0$.
\end{Lemma}

\begin{proof}
Fix any $n\in\N$, we consider the $n$-th Galerkin approximation given by
\begin{align}
    \dot y_k(t) &= \Delta_k y(t),\quad   t \in \R, \   k\in \{1,\dots, n\},\label{eq:y_heat_eq}\\
    y_k(0) &= x_k, \quad   k\in \{1,\dots, n\},\nonumber\\
     y_{k}(t) & = x_{k},\quad   t\in \R,\   k \geq n+1.\label{eq:bdry_x_gal_approx}
\end{align}
By this definition, $\Delta_k y(0)=\Delta_kx\geq 0$ for all $k\in\{1,\dots,n\}$. Set
\begin{align*}
    \tau = \sup\Big\{t\geq 0: \Delta_k y(s)\geq 0, \ \forall s\leq t,\,  \forall  k\in\{1,\dots,n\}\Big\}.
\end{align*}
We want to show $\tau =\infty$. Now suppose $\tau<\infty$. Then, there is $m\in\{1,\dots,k\}$ and a sequence of positive numbers $(t_j)_{j\in\N}$, which satisfies $\lim_{j\to\infty}t_j =0$, such that
\begin{align*}
    \Delta_m y(\tau +t_j) <0,\quad   j\in\N.
\end{align*}
We comment that the existence of $m$ and $(t_j)_{j\in\N}$ relies on the finite dimensionality of Galerkin approximations. By continuity of $t\mapsto \Delta_my(t)$ and setting $j\to \infty$, we have
\begin{align}\label{eq:Delta_my(tau)=0}
    \Delta_m y(\tau)=0.
\end{align}
Recall the definition of the discrete Laplacian in \eqref{eq:laplace}. The above two displays imply
\begin{align*}
    y_{m+1}(\tau+t_j)-y_{m+1}(\tau)+y_{m-1}(\tau+t_j)-y_{m-1}(\tau)<2\big(y_{m}(\tau+t_j)-y_{m}(\tau)\big).
\end{align*}
Sending $j\to \infty$, we obtain
\begin{align*}
    \dot y_{m+1}(\tau)+ \dot y_{m-1}(\tau)\leq 2\dot y_{m}(\tau).
\end{align*}
If $1<m<n$, then, using \eqref{eq:y_heat_eq} and \eqref{eq:Delta_my(tau)=0}, we have
\begin{align*}
    \Delta_{m+1} y(\tau)+ \Delta_{m-1} y(\tau)\leq 2\Delta_{m} y(\tau)=0.
\end{align*}
Since the two terms on the left are nonnegative due to the definition of $\tau$, we must have
\begin{align*}
    \Delta_{m+1} y(\tau) = \Delta_{m-1} y(\tau)= 0 .
\end{align*}
Similarly, we can deduce $\Delta_2 y(\tau)= 0 $ if $m=1$, and $\Delta_{n-1} y(\tau)= 0 $ if $m=n$. Iterating this, we obtain $\Delta_k y(\tau)=0$ for all $k\in\{1,\dots,n\}$. Using this and \eqref{eq:bdry_x_gal_approx}, we can compute $y_k(\tau)= k\frac{x_{n+1}}{n+1}$. It is easy to see that $y(\tau+\cdot)$ is a stationary solution of \eqref{eq:y_heat_eq} when only the first $n$ coordinates are considered. Hence, $\Delta_k y(\tau+t)=0$ for all $t\geq 0$ and all $k\in\{1,\dots,n\}$, which contradicts the assumption $\tau<\infty$. We thus deduce that $\Delta_k y(t)\geq0$ for all $t\geq 0$ and all $k\in\{1,\dots,n\}$.

\smallskip

Corollary~\ref{cor:galerkin} implies that the sequence of Galerkin approximations $\xn(\cdot)$ converges to $\sS^\cdot x$ in $\sC(\R;\XX^\alphaprime)$ for some $(\alphaprime)\in\setbetap$. Then, Lemma~\ref{lemma:ptw_cvg} yields that $\xn_k(s)$ converges to $x_k(s)$ in $\R$ for each $s\in[0,t]$ and $k\in\N$. In the previous paragraph we know that $\Delta_k\xn(s)\geq 0$ for all $n$, all $k\leq n$ and all $s\geq 0$. Passing to the limit, we obtain the desired result.
\end{proof}

Back to the proof of Lemma~\ref{lemma:monot_det}, consider $x$ satisfying its conditions. It is easy to see that $\Delta x\succeq 0$. By Lemma~\ref{lemma:convexity}, this implies $\Delta \sS^t x\succeq 0$ and thus $\frac{\d}{\d t}\sS^t x\succeq 0$ for all $t\geq 0$. In other words, $\sS^tx$ is nondecreasing in $t$. Let $z=\rr^a$. Then $z = \sS^t z$ is a stationary solution of the heat equation, and we also have $x\preceq z$. Invoking the monotonicity result of Lemma~\ref{Lemma:monotonicity} applied to zero potential and zero noise, we have $\sS^t x\preceq z$ for all $t\geq 0$. The monotone convergence theorem gives the existence of $x^*$ such that $(\sS^t x)_k$ converges to $x_k^*\in\R$ as $t\to \infty$ for all $k\in\N$. Sending $t\to \infty$ in $(\sS^{t+1}x)_k=(\sS^tx)_k+\int_0^1\Delta_k \sS^{t+s}x\,\d s$ for each $k$, one can see that  $\Delta x^*\equiv0$. Hence it can be computed iteratively that $x^*_k = kx^*_1$ for all $k\in\N$. 

\smallskip

Due to our assumptions on~$x$, there is $m$ such that
\begin{align}\label{eq:m_pf_lm_mt_det}
    x_k> 0,\quad   k\geq m.
\end{align}
Since $\sS^tx$ is nondecreasing, the coordinatewise convergence of $\sS^tx$ to $x^*$ implies $x^*_k> 0$ for $k\geq m$. This gives $x^*_1> 0$ and thus $x^*_k>0$ for all $k$. Lastly, due to this convergence, there is $T>0$ such that $(\sS^Tx)_k> 0$ for all $k< m$. Since \eqref{eq:m_pf_lm_mt_det} holds and the function $t\mapsto\sS^tx$ is nondecreasing, we can conclude $\sS^tx\succ 0$ for all $t\geq T$. This completes the proof. \epf

\subsection{Proof of Lemma~\ref{lemma:posi_prob_mt}}

The idea is  to use the property \eqref{eq:flat} first, finding a sufficiently large region where the potential $F$ is almost flat. Then, we use the noise~$W$ as a control to bring $\Phi^t_{F,W}x$ and $\Phi^t_{F,W}y$ into that region and keep them there over a period of time. 
Finally, we compare them with certain solutions of the homogeneous discrete heat equation which is possible due to the smallness of the influence of the potential in that region, and conclude by using Lemma~\ref{lemma:monot_det}.

The next lemma ensures that the controls we need as described above form a set of positive probability.

\begin{Lemma}\label{lemma:tame_noise}
For every $a \in \R$ and $ t_1,t_2$ satisfying $t_2>t_1>0$, there is $M>0$ such that, for all $\eps>0$, 
\begin{align}
\label{eq:positive-prob-of-intersect}
    \sW\left(\bigcap^\infty_{k=1}A_k\right)>0,
\end{align}
where 
\begin{gather*}
    A_k=  \Big\{\sup_{t\in[0,t_1]}|W_k(t)|\leq Mk^\frac{1}{8}\Big\}\cap\Big\{\sup_{t\in[t_1,t_2]}|W_k(t)-a| \leq\eps^2 k^\frac{1}{8}\Big\},\quad k\in\N\setminus\{1\},\\
    A_1 = \Big\{\sup_{t\in[0,t_1]}|W_1(t)|\leq M\Big\}\cap\Big\{\sup_{t\in[t_1,t_2]}|W_1(t)-(t-t_1+1)a| \leq\eps^2 \Big\}.
\end{gather*}

\end{Lemma}
\begin{proof}
It is sufficient to prove \eqref{eq:positive-prob-of-intersect} for small $\eps$.
Our choice of $M$ must at least ensure $\sW(A_k)>0$ for all $k$. For this, 
taking into account the continuity of~$W_k(t)$ at $t=t_1$, it suffices to have
\begin{align*}
    (a-\eps^2 k^\frac{1}{8}, a +\eps^2 k^\frac{1}{8})\cap (-Mk^\frac{1}{8}, Mk^\frac{1}{8})\neq \emptyset,\quad   k\in\N,
\end{align*}
which is guaranteed by $M >|a|$. Let us fix such $M$.

Using the independence of $W_k$, we have $\sW(\cap^\infty_{k=1}A_k)=\prod_{k=1}^\infty \sW(A_k)$. To show that this product is positive, it suffices to prove
\begin{align}
    \sum_{k=1}^\infty -\log \sW(A_k)<\infty.
    \label{eq:conv-series}
\end{align}
Note that for $k$ large, there are $C_1, C_2>0$ such that
\begin{align*}
    \sW(A_k^c) \leq \sW\Big\{\sup_{t\in[0,t_2]}|W_k(t)|\geq C_1 k^\frac{1}{8}\Big\}\leq 2e^{-C_2k^\frac{1}{4}}.
\end{align*}
Hence, for $k$ large, we have
\begin{align*}
    -\log\sW(A_k) = \log\big(1-\sW(A_k^c)\big)^{-1}\leq \log(1 - 2e^{-C_2k^\frac{1}{4}})^{-1}\leq C_3 e^{-C_2k^\frac{1}{4}},
\end{align*}
and~\eqref{eq:conv-series} follows.
\end{proof}

\begin{proof}[Proof of Lemma~\ref{lemma:posi_prob_mt}]

For brevity, we write $X(t)=\Phi^t_{F,W}x$ and $Y(t)= \Phi^t_{F,W} y$. By Proposition~\ref{prop:invariance_slope}, we have $X,Y\in \cS(x,\Delta,F,W;\LL)$ and by Remark~\ref{rem:heat_eq}, we know $\sS^\cdot z\in \cS(z,\Delta,0,0;\LL)$ whenever $z\in\LL$.

\smallskip
Step 1. We introduce basic deterministic objects and choose the size of the region where the potential is almost flat.
By Lemma~\ref{Lemma:monotonicity} applied to zero random potential and noise allowed by Remark~\ref{remark:F=0_W=0}, 
if $x\in \llbracket \rr^{u-h}-c,\rr^{u+h}+c\rrbracket$, then
$\sS^tx\succeq \sS^t(\rr^{u-h}-c)$ for all $t\geq 0$.
On the other hand, Lemma~\ref{lemma:convexity} implies that $\Delta \sS^t(\rr^{u-h}-c)\succeq 0$ for all $t\geq 0$. Since $\sS^t(\rr^{u-h}-c)$ satisfies the heat equation, we have $\frac{\d}{\d t}\sS^t(\rr^{u-h}-c)\succeq 0$ and thus $\sS^t(\rr^{u-h}-c)\succeq \rr^{u-h}-c $ for all $t\geq 0$. Hence, we obtain
\begin{align}\label{eq:x(t)_lower}
    \sS^tx \succeq \rr^{u-h}-c,\quad   t\geq 0.
\end{align}
Similarly, for all $ y\in \llbracket \rr^{v-h}-c,\rr^{v+h}+c\rrbracket $, we also have
\begin{align}\label{eq:y(t)_upper}
    \sS^ty \preceq \rr^{v+h}+c,\quad   t\geq 0.
\end{align}
Let us choose $\delta>0$ to satisfy $u-h-2\delta >v+h+2\delta$ and  consider $\eps\in(0, \delta\wedge 1)$. For $L>0$, we set
\begin{gather}\label{eq:x^one}
    x^\one = \rr^{u-h-\delta}-L,\qquad y^\one = \rr^{v+h+\delta}+L.
\end{gather}
In view of \eqref{eq:x(t)_lower} and \eqref{eq:y(t)_upper}, we can choose $L$ large so that
\begin{gather}
    (\sS^tx)_k - \eps k^\frac{3}{4}\geq x^\one_k,\qquad (\sS^ty)_k+\eps k^\frac{3}{4}\leq y^\one_k,\quad  k\in\N \label{eq:x>x^one}
\end{gather}
holds for all $t\geq 0$, for all $\eps<1$, all $x\in \llbracket \rr^{u-h}-c,\rr^{u+h}+c\rrbracket$, all $ y\in \llbracket \rr^{v-h}-c,\rr^{v+h}+c\rrbracket $.
Additionally, we set
\begin{align}
    x^\two = \rr^{u-h-2\delta}-L,\qquad y^\two = \rr^{v+h+2\delta}+L.\label{eq:x^two}
\end{align}
By Lemma~\ref{lemma:monot_det} applied to $\sS^{\cdot} (x^\two -y^\two)$, there is $T>0$ such that 
\begin{align}\label{eq:T_mt}
    \sS^Tx^\two-\sS^Ty^\two =\sS^T(x^\two -y^\two)\succ 0.
\end{align}

\smallskip

We move on to choosing some parameters.
We fix $n$ large enough to guarantee
\begin{align}\label{eq:choice_n}
    \sum_{k\geq n+1}(\log k)^2 k^{-\frac{3}{2}}\leq \eps^4.
\end{align}
We set $l$ to be
\begin{align}\label{eq:l}
    l=n \big(\|x^\one\|_\LL + \|y^\one\|_\LL +K\big)e^{KT}.
\end{align}
Here the large constant $K>0$ will be specified later.
By the condition~\eqref{eq:flat} with $\delta =\frac{\eps^4}{n}$, there is $a\in\R$ such that
\begin{align}\label{eq:zero_region}
    |f_k(s)|\leq \frac{\eps^4}{n}, \quad   (k,s)\in\{1,2,\dots,n\}\times [a-l,a+l].
\end{align}

\bigskip

Step 2. For~$a$ in \eqref{eq:zero_region}, we define $\aa$ by $\aa_k=a$, $k\in\N$. We set
\begin{align}\label{eq:t_0=eps^2}
    t_0 =\eps^2.
\end{align}
We want to show that with this definition of $t_0$, for  sufficiently small $\eps$, with positive probability,  $X(t_0)$ and $Y(t_0)$ are  $\eps$-close to $\sS^{t_0}x+\aa$:
\begin{gather}
    \|X(t_0) -\sS^{t_0}x -\aa\|_* <\eps,\label{eq:X-x-a_bd}\\
    \|Y(t_0) -\sS^{t_0}y -\aa\|_* <\eps.\label{eq:Y-y-a_bd}
\end{gather}

Here the norm is given in \eqref{eq:*norm}. 
We apply Lemma~\ref{lemma:tame_noise} with $t_1 =t_0$ and $t_2 =t_0+T$, where $T$ is given by Lemma~\ref{lemma:monot_det}, to see that, with positive probability,
\begin{align}
    \sup_{t\in[0,t_0]}&|W_k(t)|\leq Mk^\frac{1}{8},\quad   k\in\N,\label{eq:W_bounded_[t,t_0]}\\
    \sup_{t\in[0,T]}&|W_k(t_0+t)-a|\leq \eps^2 k^\frac{1}{8},\quad   k\in\N\setminus\{1\}, \label{eq:W-a_bd_[t_0,t_0+T]}\\
    \sup_{t\in[0,T]}&|W_1(t_0+t)-(t+1)a|\leq \eps^2 \label{eq:W-a_bd_[t_0,t_0+T],k=1}.
\end{align}
The first coordinate is special because  $(\Delta \aa)_k$ is nonzero only for $k=1$.
Hereafter, we fix any realization $W$ in $\WW_0$ satisfying \eqref{eq:W_bounded_[t,t_0]}--\eqref{eq:W-a_bd_[t_0,t_0+T],k=1}. In the following, $C$ stands for a constant uniform in all such $W$. In particular, \eqref{eq:W_bounded_[t,t_0]}--\eqref{eq:W-a_bd_[t_0,t_0+T],k=1} imply 
\begin{align}\label{eq:W_norm_L,*}
    \sup_{t\in[0,T]}\|W\|_\LL \leq \sup_{t\in[0,T]}\|W\|_*\leq C.
\end{align}

By Lemma~\ref{lemma:comp_in_*}, we obtain
\begin{align}\label{eq:X_bdd_[0,t_0]}
    \sup_{t\in[0,t_0]}\|X(t)-\sS^t x\|_* \leq C.
\end{align}
Subtracting $\aa$ from both sides of the equation that $X(t_0)-\sS^{t_0}x$ satisfies, we have
\begin{align*}
    \|X(t_0)-\sS^{t_0} x-\aa\|_* \leq \int_0^{t_0}C\|X(s)-\sS^sx\|_*\d s + \int_0^{t_0} \|f(X(s))\|_* \d s + \|W(t_0)-\aa\|_*.
\end{align*}
Due to $X(s)\in\sC(\R;\LL)$, by Lemma~\ref{Lemma:f_property}~\eqref{item:f_k_upper_bound}, we can verify $\sup_{s\in[0,t_0]}\|f(X(s))\|_*\leq C$. By this, \eqref{eq:t_0=eps^2} and \eqref{eq:X_bdd_[0,t_0]}, the two integrals on the right of the above display are bounded by $C t_0=C\eps^2$. Due to \eqref{eq:W-a_bd_[t_0,t_0+T]}--\eqref{eq:W-a_bd_[t_0,t_0+T],k=1} applied at $t=0$, it can seen that $\|W(t_0)-\aa\|_* \leq C\eps^2$. Using these estimates in the above display we obtain $\|X(t_0)-\sS^{t_0}x-\aa\|_*\leq C\eps^2$.  Hence, \eqref{eq:X-x-a_bd} holds for sufficiently small $\eps$.
Similarly, \eqref{eq:Y-y-a_bd} is valid.

\bigskip

Step 3. We study the behavior of $X(t_0+t)$ and $Y(t_0+t)$ for $t\in[0,T]$. The key results are \eqref{eq:X>X^one}, \eqref{eq:X^one-x^one} and the corresponding ones for $Y$. The estimate \eqref{eq:X-x-a_bd} implies
\begin{align*}
    X_k(t_0)\geq (\sS^{t_0}x)_k+a - \eps k^\frac{3}{4},\quad   k\in\N.
\end{align*}
In view of \eqref{eq:x>x^one}, this shows $X(t_0)\succeq x^\one +\aa$. 
Let 
\[X^\one(t_0+t)=\Phi^{t}_{F,\theta^{t_0}W}(x^\one+\aa),\quad t\ge0,\]
 i.e.,
$X^\one(t_0+\cdot)\in \cS(x^\one,\Delta,F,\theta^{t_0}W;\LL)$ satisfies
\begin{align}\label{eq:X^one}
\begin{split}
        &X^\one(t_0+t) = x^\one +\aa+\int_0^t\Delta X^\one(t_0+s)\d s \\&\qquad+ \int_0^t f(X^\one(t_0+s))\d s +W(t_0+t)-W(t_0).
\end{split}
\end{align}
Then, Lemma~\ref{Lemma:monotonicity} implies that 
\begin{align}\label{eq:X>X^one}
    X(t_0+t)\succeq X^\one (t_0+t),\quad   t\geq 0.
\end{align}
Let us show that
\begin{align}\label{eq:X^one-x^one}
    \sup_{t\in[0,T]}\|X^\one(t_0+t)-\sS^t x^\one-\aa\|_* \leq \eps.
\end{align}
We set $\hat\aa=-\Delta \aa$, i.e., recalling the definition of $\Delta$ in \eqref{eq:laplace},
 $\hat \aa_1 =a$ and $\hat\aa_k=0$ for all $k\geq 2$. We have 
\begin{align}\label{eq:lap_X^one}
    \Delta X^\one = \Delta (X^\one -\aa) - \hat\aa.
\end{align}
This,  \eqref{eq:X^one}, and the fact that $\sS^t x^\one$ solves the heat equation imply
\begin{align*}
        X^\one(t_0+t)-\sS^t x^\one-\aa = \int_0^t \Big(\Delta\big(X^\one(t_0+s)-\sS^sx^\one-\aa\big) \\ +  f(X^\one(t_0+s))\Big)\d s +  W(t_0+t)-t\hat\aa-W(t_0). 
\end{align*}
Therefore,
\begin{align}\label{eq:X^one-x^one_step}
\begin{split}
    \|X^\one(t_0+t)-\sS^t x^\one-\aa\|_* \leq \int_0^tC\|X^\one(t_0+s)-\sS^sx^\one-\aa\|_* \d s\\ + \int_0^t \|f(X^\one(t_0+s))\|_* \d s + \| W(t_0+t)-t\hat\aa-W(t_0)\|_*. 
\end{split}
\end{align}
Lemma~\ref{lemma:comp_in_*} applied to $X^\one(t_0+\cdot)$ and $\sS^\cdot x^\one$ yields $\sup_{t\in[0,T]}\|X^\one(t_0+t)-\sS^t x^\one\|_*\leq C$,
 which implies
\begin{align*}
    \sup_{t\in[0,T]}\|X^\one(t_0+t)-\sS^t x^\one-\aa\|_*\leq C.
\end{align*}
From \eqref{eq:W-a_bd_[t_0,t_0+T]}--\eqref{eq:W-a_bd_[t_0,t_0+T],k=1}, we see that
\begin{align}\label{eq:W-W}
    \sup_{t\in[0,T]}\|W(t_0+t)-t\hat \aa-W(t_0)\|_*,\quad \sup_{t\in[0,T]}\|W(t_0+t)-t\hat \aa-W(t_0)\|_\LL \leq C\eps^2.
\end{align}
To estimate the integral involving $f$, we want to use \eqref{eq:zero_region}. For this, we need 
\begin{align}
\label{eq:X-in-theright-region}
    X_k^\one(t_0+t) \in [a-l,a+l],\quad   k\in \{1,2,\dots,n\}, \   t\in[0,T].
\end{align}
so let us estimate $X^\one(t_0+t)-\aa$.
An estimate similar to \eqref{eq:X^n<} holds for $X^\one(t_0+\cdot)$ with $\|\cdot\|$ replaced by $\|\cdot\|_\LL$. Using \eqref{eq:W_norm_L,*} and Gronwall's inequality, we have
\begin{gather}
    \sup_{t\in [0,T]}\|X^\one(t_0+t)\|_\LL \leq C.\label{eq:sup_X^one}
\end{gather}
This and Lemma~\ref{Lemma:f_property}~\eqref{item:f_k_upper_bound} yield
\begin{gather}
     \sup_{t\in[0,T]}\|f(X^\one(t_0+t)\|_\LL \leq C.\label{eq:f_bdd_alpha_poisson_2}
\end{gather}
By \eqref{eq:lap_X^one} and \eqref{eq:X^one}, we have
\begin{align*}
    \|X^\one(t_0+t)-\aa\|_\LL \leq \|x^\one\|_\LL + \int_0^tC \|X^\one(t_0+s)-\aa\|_\LL \d s\\ +\int_0^t \|f(X^\one(t_0+s))\|_\LL \d s +\|W(t_0+t)-t\hat\aa-W(t_0)\|_\LL.
\end{align*}
Using \eqref{eq:f_bdd_alpha_poisson_2} and \eqref{eq:W-W} along with Gronwall's inequality, we obtain
\begin{align}\label{eq:X^one-a}
    \sup_{t\in[0,T]}\|X^\one(t_0+t)-\aa\|_\LL \leq \big(\|x^\one\|_\LL +K\big)e^{KT}
\end{align}
where the constant $K$  is determined only by  $\|\Delta\|_{\sL_\b(\LL)}$  and $f$, so it does not depend on the shift parameter $a$.  Thus we can use this $K$ in \eqref{eq:l}. 
Now, from \eqref{eq:l} and \eqref{eq:X^one-a}, we deduce \eqref{eq:X-in-theright-region}.
The latter, along with \eqref{eq:choice_n}, \eqref{eq:zero_region}, \eqref{eq:sup_X^one} and Lemma~\ref{Lemma:f_property}~\eqref{item:f_k_upper_bound}, implies that 
\begin{align*}
    \|f(X^\one(t_0+s))\|^2_*\leq \eps^4 + C\sum_{k\geq n+1}^\infty(\log k)^2 k^{-\frac{3}{2}} \leq C\eps^4,\quad   s\in[0,T].
\end{align*}
Inserting this and \eqref{eq:W-W} into \eqref{eq:X^one-x^one_step} and applying Gronwall's inequality, we arrive at 
\begin{align*}
    \sup_{t\in[0,T]}\|X^\one(t_0+t)-\sS^t x^\one-\aa\|_* \leq C\eps^2(1+T)e^{CT}.
\end{align*}
Hence,  \eqref{eq:X^one-x^one} holds for sufficiently small $\eps$.

\bigskip

Step 4. We want to deduce that monotonization happens, namely, $X(t_0+T)\succ Y(t_0+T)$. From \eqref{eq:X^one-x^one}, we have
\begin{align}\label{eq:X^one_k>x^one_k}
    X_k^\one(t_0+t) \geq (\sS^tx^\one)_k+a - \eps k^\frac{3}{4},\quad   k \in \N,\   t\in[0,T].
\end{align}
By \eqref{eq:x^one}, \eqref{eq:x^two}, we have $x^\one - x^\two = \rr^\delta$ and thus
\begin{align*}
    \sS^tx^\one-\sS^tx^\two =\sS^t \rr^\delta =\rr^\delta,\quad t\ge 0.
\end{align*}
Initially, we have already required $\eps<\delta$, which gives
\begin{align*}
    (\sS^t x^\one)_k-(\sS^tx^\two)_k\geq \eps k^\frac{3}{4},\quad k\in\N,\ t\ge0.
\end{align*}
This together with \eqref{eq:X>X^one} and \eqref{eq:X^one_k>x^one_k} implies 
\begin{align*}
    X(t_0+t)\succeq \sS^t x^\two +\aa,\quad   t\in[0,T].
\end{align*}
Symmetrically, the same arguments yield
\begin{align*}
    Y(t_0+t)\preceq \sS^t y^\two+\aa,\quad   t\in[0,T].
\end{align*}
Therefore, we have
\begin{align*}
    X(t_0+t)-Y(t_0+t)\succeq  \sS^t x^\two-\sS^t x^\one,\quad   t\in[0,T],
\end{align*}
and $X(t_0+T)-Y(t_0+T)\succ 0$ follows from this and \eqref{eq:T_mt}. 

\smallskip

Note that by our construction in Step~1, the desired inequality $\Phi^T_{F,W}x \succ \Phi^T_{F,W}y$
holds with this choice of $T$ on the event of positive probability described in \eqref{eq:W_bounded_[t,t_0]}--\eqref{eq:W-a_bd_[t_0,t_0+T],k=1} simultaneously 
 for all $x\in \llbracket \rr^{u-h}-c,\rr^{u+h}+c\rrbracket$ and all $y\in \llbracket \rr^{v-h}-c,\rr^{v+h}+c\rrbracket $,
 which completes the proof of the lemma.
\end{proof}

\subsection{Proof of Lemma~\ref{lemma:recur}}
We recall the sets $\FF(v)=\FF_\invtemp(v)$ introduced in Theorem~\ref{thm:thermodynamic-limit}. We set 
\begin{gather*}
    \mathcal{I}(u,v)=(u+\Q)\cup  (v+\Q),\\
    \FF(u,v)=\bigcap_{w\in \mathcal{I}(u,v)}\FF(w).
\end{gather*}
Then, we have $\sF(\FF(u,v))=1$.
By Theorem~\ref{thm:thermodynamic-limit}, if $F\in\FF(u,v)$, then for each $w \in \mathcal{I}(u,v)$, there is a unique IVPM $\mu_w=\mu_{w,F}$ concentrated on $S(w)$.
For $x\in \R^\N$, we define
\begin{align}\label{eq:def_C^>x}
    \begin{split}
        C^{\succeq x} & = \{y\in\R^\N: y\succeq x\},\\
    C^{\preceq x} & = \{y\in\R^\N: y\preceq x\}.
    \end{split}
\end{align}

\begin{Lemma} 
\label{lem:dominations-exist}
Suppose $F\in\FF(u,v)$, $w,w'\in \mathcal{I}(u,v)$ and $x\in S(w)$.   If $w'>w$, then $\mu_{w'}(C^{\succeq x})>0$. 
If $w'<w$, then $\mu_{w'}(C^{\preceq x})>0$. 
\end{Lemma}

\begin{proof}
By symmetry, we only need to prove  the first part. We denote 
\begin{align*}
B_{\ge n}&=\{y\in\R^{\{n,n+1,\ldots\}}:\ y_k\ge x_k,\ k\ge n\},\\
B_{< n}&=\{y\in\R^{\{1,\ldots n-1\}}:\ y_k\ge x_k,\ k< n\}.
\end{align*}

For every $y\in S(w')$, there is $n$ such that $y_{\ge n}\in B_{\ge n}(x)$. So let us find  $n$ such that
\begin{equation}
\label{eq:tail-ordered}
\mu_{w'}\{y:\ y_{\ge n}\in B_{\ge n}(x)\}>0.
\end{equation} 
Using regular conditional probabilities and decomposing  \[\mu_{w'}(\d y)=\mu_{w', <n}(\d y_{<n}|\, y_{\ge n}) \mu_{w', \ge n}(\d y_{\ge n}),\]
we can write
\[
\mu_{w'}(C^{\succeq x})=\int_{B_{\ge n}}  \mu_{w', \ge n}(\d y_{\ge n}) \mu_{w', <n}(B_{<n}; y_{\ge n}). 
\]
The lemma now follows once we notice that \eqref{eq:tail-ordered} means that $\mu_{w', \ge n}(B_{\ge n})>0$ and that  $\mu_{w', <n}(B_{<n}; y_{\ge n})>0$ for all $y_{\ge n}$ since $\mu_{w', <n}(\d y_{<n}|\, y_{\ge n})$ is a finite-dimensional Gibbs distribution equivalent to the Lebesgue measure.
\end{proof}

\begin{Lemma} 
\label{lem:dominations-exist-ergodic}
Suppose $F\in\FF(u,v)$, $w,w'\in \mathcal{I}(u,v)$ and $x\in S(w)$.   If $w'>w$, then there is an ergodic invariant measure $\nu_{w'}$ concentrated on $S(w')$ such that $\nu_{w'}(C^{\succeq x})>0$. 
If $w'<w$, then there is an ergodic invariant measure $\nu_{w'}$ concentrated on $S(w')$ such that $\nu_{w'}(C^{\preceq x})>0$. 
\end{Lemma}

\begin{proof} Since $S(w')$ is a Polish space, the ergodic decomposition $\mu_{w'}=\int \nu \, \Pi(d\nu)$  holds. Since the ergodic measures $\nu$ are absolutely continuous with respect to $\mu_{w'}$ and hence concentrated on $S(w')$, the claim of the lemma follows from Lemma~\ref{lem:dominations-exist}.
\end{proof}

\begin{proof}[Proof of Lemma~\ref{lemma:recur}]
First take a number $h'\in(0,h)\cap\Q$ and use Lemma~\ref{lem:dominations-exist-ergodic} to find ergodic invariant measures
$\nu_{u-h'}$, $\nu_{u+h'}$, $\nu_{v-h'}$, $\nu_{v+h'}$ on
$S(u-h')$, $S(u+h')$, $S(v-h')$, $S(v+h')$
satisfying

\begin{equation}
\label{eq:positive_measures_of_cones}
\nu_{u-h'}(C^{\preceq x}),\ \nu_{u+h'}(C^{\succeq x}),\ \nu_{v-h'}(C^{\preceq y}),\ \nu_{v+h'}(C^{\succeq y}) >0,
\end{equation}
Let us fix $\eps>0$ and find a constant $c>0$ sufficiently large
such that 
\begin{equation}
\label{eq:large-sets}
\nu_{u-h'}(C^{\succeq \rr^{u-h}-c}),\    \nu_{u+h'}(C^{\preceq \rr^{u+h}+c}),\ 
\nu_{v-h'}(C^{\succeq \rr^{v-h}-c}),\    \nu_{v+h'}(C^{\preceq \rr^{v+h}+c})> 1-\eps.
\end{equation}
For $A\in\Bc(\R^\N)$, $r\in\R$, $\delta>0$, we denote
\begin{equation}
\label{eq:def_B(r,A,eps)}
B(r,A,\delta)=\Bigg\{z\in S(r):\
\sW\Big\{W: \limsup_{T\to\infty}\frac{1}{T}\int_0^T \mathds{1}_{A}(\Phi^t_{F,W}z) dt >1- \delta\Big\}=1\Bigg\},
\end{equation}
By the ergodic theorem and~\eqref{eq:large-sets}, we have
\begin{align}
\label{eq:on-one-side-of-the-line-1}
\nu_{u-h'}\Big(B\big(u-h', C^{\succeq \rr^{u-h}-c},\eps\big)\Big)&=1,
\\
\nu_{u+h'}\Big(B\big(u+h', C^{\preceq \rr^{u+h}+c},\eps\big)\Big)&=1,
\\
\nu_{v-h'}\Big(B\big(v-h', C^{\succeq \rr^{v-h}-c},\eps\big)\Big)&=1,
\\
\nu_{v+h'}\Big(B\big(v+h', C^{\preceq \rr^{v+h}+c},\eps\big)\Big)&=1,
\label{eq:on-one-side-of-the-line-4}
\end{align}
Let us now take independent initial conditions 
$x^{u-h'}, x^{u+h'},  x^{v-h'}, x^{v+h'}$
distributed according to $\nu_{u-h'}$, $\nu_{u+h'},$
$\nu_{v-h'}$, 
$\nu_{v+h'}$.

Due to~\eqref{eq:positive_measures_of_cones} and ~\eqref{eq:on-one-side-of-the-line-1}--\eqref{eq:on-one-side-of-the-line-4}, we have 
\[
\nu_{u-h'}\times \nu_{u+h'} \times  \nu_{v-h'}\times  \nu_{v+h'}  (B)>0,
\]
where
\begin{multline}
\label{eq:def-B}
B = \big(C^{\preceq x}\times C^{\succeq x} \times C^{\preceq y}\times C^{\succeq y} \big)
\cap \Big( B\big(u-h', C^{\succeq \rr^{u-h}-c},\eps\big) \times
\\ \times B\big(u+h', C^{\preceq \rr^{u+h}+c},\eps\big) \times 
B\big(v-h', C^{\succeq \rr^{v-h}-c},\eps\big)
\times B\big(v+h', C^{\preceq \rr^{v+h}+c},\eps\big)\Big).
\end{multline}
In particular, $B\ne\emptyset$, so we can take some
 $(x^{u-h'}, x^{u+h'}, x^{v-h'}, x^{v+h'})\in B$.
 We have
\begin{align*}
x^{u-h'}\preceq x \preceq x^{u+h'},\\
x^{v-h'}\preceq y \preceq x^{v+h'},
\end{align*}
so that due to monotonicity of the flow,
\begin{align*}
\Phi^t_{F,W} x^{u-h'}\preceq  \Phi^t_{F,W} x \preceq \Phi^t_{F,W} x^{u+h'},\\
\Phi^t_{F,W}  x^{v-h'}\preceq \Phi^t_{F,W} y \preceq \Phi^t_{F,W} x^{v+h'}.
\end{align*}
This means that if $\Phi^t_{F,W} x^{u-h'}\in C^{\succeq \rr^{u-h}-c}$, 
$\Phi^t_{F,W} x^{u+h'}\in C^{\preceq \rr^{u+h}+c}$,
$\Phi^t_{F,W}  x^{v-h'}\in C^{\succeq \rr^{v-h}-c}$, 
and  $\Phi^t_{F,W} x^{v+h'}\in C^{\preceq \rr^{v+h}+c}$, 
then
$\Phi^t_{F,W} x\in \llbracket \rr^{u-h}-c,\rr^{u+h}+c\rrbracket$ and $\Phi^t_{F,W} y\in \llbracket \rr^{v-h}-c,\rr^{v+h}+c\rrbracket$. 

Therefore, using the definition \eqref{eq:def_B(r,A,eps)} in \eqref{eq:def-B}, we obtain
\begin{align*}
\sW\Big\{ \limsup_{T\to\infty}\frac{1}{T}\int_0^T \mathds{1}_{\llbracket \rr^{u-h}-c,\rr^{u+h}+c\rrbracket}(\Phi^t_{F,W}x) \mathds{1}_{\llbracket \rr^{v-h}-c,\rr^{v+h}+c\rrbracket}(\Phi^t_{F,W} y)  dt  > 1- 4\eps\Big\}=1.
\end{align*}
Choosing arbitrary $\eps<1/4$ and denoting by $\WW(u,v,h,F)$ the full measure set in the above display, we complete the proof.
\end{proof}

\section{One Force --- One Solution  Principle}\label{section:1F1S}

The goal of this section is to prove parts~\ref{it:uniq-erg-mix} and~\ref{it:1F1S} of Theorem~\ref{th:collect}. After auxiliary statements
on shear invariance, we first  prove a weak form of 1F1S in  Theorem~\ref{th:sample-measures-for-Gibbs-are-delta}. The unique ergodicity claim of part~\ref{it:uniq-erg-mix}  is part~\ref{it:inv-measure} of Theorem~\ref{th:sample-measures-for-Gibbs-are-delta}. The claim of existence and uniqueness of stationary nonanticipating global solution in part~\ref{it:1F1S} of Theorem~\ref{th:collect} is part~\ref{it:global-solution} of Theorem~\ref{th:sample-measures-for-Gibbs-are-delta}. The attractor property of the global solution  in part~\ref{it:1F1S} of Theorem~\ref{th:collect} is Theorem~\ref{th:1F1S-coord}. The mixing property in part~\ref{it:uniq-erg-mix} of Theorem~\ref{th:collect} is Corollary~\ref{cor:mixing}.

\subsection{Shear-invariance}
\label{sec:shear-inv}
Just like the Burgers equation dynamics studied in~\cite{Bakhtin-Li:MR3911894}, the polymer dynamics is skew-invariant with respect to
the shear semigroup. For brevity we speak of {\it shear-invariance}. To state this property, we first recall the definition~\eqref{eq:shear-on-F} of the shear transformations group $\Shear=(\Shear^v)_{v\in\R}$ acting on functions defined on $\NR$ such as potentials and their derivatives.
We also need to define the action of the same group on polymers themselves: for $v\in\R$ and $x\in\R^\N$,
\[
(\Shear^v x)_k =x_k+kv,\quad k\in\N.
\]

\begin{Lemma}
\label{lem:shear-invariance-of-dynamics}
For all potentials $F\in\FF_0$, for all noise realizations $W\in\WW_0$, all $(\alphap)\in\Pi$ all initial conditions $x\in \XX^\alphap$, all times $t>0$, we have 
\[
\Phi^t_{\Shear^v F,W} \Shear^v x = \Shear^v \Phi^t_{F,W}  x.
\]
\end{Lemma}

\bpf 
The shear transformations do not affect the noise realizations, so to check that
$X\in\cS(x,\Delta,F,W;\LL)$ implies $\Shear^v X\in\cS(x,\Delta,\Shear^v F,W;\LL)$
it suffices to note that the drift in right-hand side  of \eqref{eq:SDE}  transforms appropriately:
\[
\Delta_k \Shear^v x =(x_{k-1}+(k-1)v)+(x_{k+1}+(k+1)v)-2(x_n+kv) =\Delta_k x,\quad k\in\N,
\]
and $\Shear^v F\in\FF_0$ with
\[
(\Shear^v f)_k((\Shear^v x)_k)=f_k(x_k),\quad k\in\N.
\]
The lemma now follows from the uniqueness of solutions.
\epf

This is an important property which, along with monotonicity, and distributional invariance of potentials under shear transformations,
 helps us to study invariant distributions for polymer and in particular, establish 1F1S. 
The idea of the lack-of-space argument for the following statement  goes back at least to  \cite[Lemma~6]{HoNe3}. It has been used in uniqueness arguments in~
\cite{BCK:MR3110798,kickb:bakhtin2016,Bakhtin-Li:MR3911894}.

\begin{Lemma}\label{lem:jumps-of-random-monotone-f} Suppose $(X_v)_{v\in\R}$ is a stochastic process on a probability space $(\Omega,\Ac,\Pp)$ with a.s.-nondecreasing in $v$ right-continuous realizations, and with stationary increments. Then,  for every $v$, $\Pp\{X(v+)\ne X(v-)\}=0$.
\end{Lemma}
\bpf  Denoting $A_v=\{X(v+)\ne X(v-)\}$, noting that $p=p_v=\Pp(A_v)$ does not depend on $v$ due to our stationarity assumption,
and choosing any probability density $g$ with respect to the Lebesgue measure on $\R$, we obtain, changing the order of integration:
\[
p=\int_{\R}g(v)p_v \d v=\int_{\R}g(v)\Pp(A_v) \d v = \int_{\R}g(v)\Pp \ONE_{A_v} \d v = \Pp \int_{\R}g(v) \ONE_{A_v} \d v = 0
\]
since a monotone function has at most countably many discontinuity points.
\epf

 Our strategy to prove
uniqueness of various objects is based on 
 identifying intervals of their possible values with 
upward jumps of appropriate random monotone functions satisfying the conditions of Lemma~\ref{lem:jumps-of-random-monotone-f} due to shear invariance
thus showing that those intervals collapse into single points.

\subsection{Weak 1F1S}
\label{sec:weak1F1S}

Our argument is based on sample measures also known as Markov measures. This notion was first introduced in \cite{Ledrappier-LSY:MR968818}, see also \cite[Section~1.7 and Chapter~2]{Arnold:MR1723992} and~\cite{Crauel:MR1148346}.

Suppose $(P^t)$ is a Markov semigroup  on a Polish metric space $(\XX,d)$ with an invariant distribution $\rho$.
Suppose further that  $(P^t)$
generated by a continuous random dynamical system $(\Phi^t_W)$ on a white-noise probability space $(\WW,\Wc,\Ws)$ with a group of measure-preserving automorphisms $\tshift^t$ via  
\begin{equation*}
P^tg(x)=\Ws g(\Phi^t_{W}x),\quad x\in \XX,\ t\ge 0,\ g\in\sM_\b(\XX).
\end{equation*}
The white-noise property means that maps $\Phi^{s_1,t_1}_W, \Phi^{s_2,t_2}_W,\ldots,  \Phi^{s_n,t_n}_W $ are independent for all $n$ and all disjoint
$(s_1,t_1],\ldots,(s_n,t_n]$, where
\begin{equation*}
\Phi^{s,t}_W=\Phi_{\tshift^s W}^{t-s}.
\end{equation*}

Then
there is a $\tshift$-invariant set $\widehat\WW$ of probability~$1$ such that  the weak limit
\begin{equation}
\label{eq:time-dependent-sample-measures}
\rho_{W,t}=\lim_{s\to-\infty} \rho (\Phi^{s,t}_W)^{-1}
\end{equation}
exists for all $W\in\widehat\WW$ and all $t\in\R$. Here,  for a measure $\mu$ and a measurable function~$h$,
$\mu h^{-1}$ denotes the pushforward of $\mu$ under $h$: $\mu h^{-1}(A)=\mu(h^{-1}(A))$.

It is useful to single out $t=0$ and set $\rho_W=\rho_{W,0}$.

The {\it sample measures} $(\rho_{W,t})_{W\in\widehat W,\ t\in\R}$ describe the distribution worked out by the random dynamics given the history of the noise. They  
have the following properties:
\begin{enumerate}[(i)]
\item for all $W\in\widehat\WW$ and all $t\in\R$, $ \rho_{W,t}=\rho_{\tshift^t W} $;
\item for all $t\in\R$,  $\rho_{W,t}$ is $\Wc_t$-measurable and  independent of $\Wc_{[t,\infty)}$. 
\item  for all $t_1,t_2\in\R$ satisfying $t_1<t_2,$
\begin{equation}
\label{eq:sample-measures-skew-inv}
\rho_{W,t_1}(\Phi^{t_1,t_2}_W)^{-1}= \rho_{W,t_2};
 \end{equation}
\item For all $t\in\R$
\[
\rho(\cdot)=\Ws\rho_{W,t}(\cdot).
\]
\end{enumerate}

We can apply this concept to the IVPM $\mu_{v,F}$ constructed in  Theorem~\ref{thm:thermodynamic-limit}  because it is invariant under the polymer dynamics Markov semigroup  and concentrated on $S(v)$ which is a Polish space (see Proposition~\ref{prop:invariance} and Remark~\ref{rem:S(v)}). In agreement with  \eqref{eq:time-dependent-sample-measures},
we will denote the associated sample measures by $\mu_{v,F,W,t}$.

The idea is to approach a form of 1F1S through the ordering by noise property which helps to compare the sample measures concentrated on distinct~$S(v)$, in combination with the shear invariance.

We begin with definitions related to supports of measures. For a Borel measure $\nu$ on $\R$, we denote the leftmost point in its support by
\begin{align*}
    \lsupp(\nu) = \sup\{r\in\R: \nu\{r': r'\le r\}=0\}\in\R\cup\{-\infty\},
\end{align*}
and the rightmost point in its support by
\begin{align*}
    \rsupp(\nu) = \inf\{r\in\R: \nu\{r': r'\ge r\}=0\}\in\R\cup\{+\infty\}.
\end{align*}
We define the coordinate operator $\pi_k:\R^\N\to\R$  by $\pi_k x=x_k$. For a measure $\mu$ on~$\R^\N$, a point  $x\in (\R\cup\{-\infty\})^\N$  is called the lower boundary for the support of $\mu$ if
\begin{align*}
    x_k = \lsupp(\nu \pi^{-1}_k),\quad k\in\N,
\end{align*}
and $x\in  (\R\cup\{+\infty\})^\N$ is the upper boundary for the support of $\mu$ if
\begin{align*}
    x_k = \rsupp(\nu \pi^{-1}_k),\quad k\in\N.
\end{align*}

We recall the definition~\ref{eq:def_C^>x}.
\begin{Lemma}\label{lem:criterion-for-dominating-supports}
 If $\mu_1$ and $\mu_2$ are two probability measures on $\R^\N$ satisfying
\begin{equation}
\label{eq:cond_on_product_measure}
\mu_1\times\mu_2\{(y^1,y^2): y^1\preceq y^2\}=1,
\end{equation}
 then 
\[
\mu_1(C^{\preceq x})=\mu_2(C^{\succeq x})=1,
\]
where $x\in\R^\N$ can be taken as either the lower boundary for the support of $\mu_2$ or the upper boundary for the support of $\mu_1$.
\end{Lemma}
\bpf 
We take $x$ to be the lower boundary for the support of $\mu_2$ and $z$ to be the upper boundary for the support of $\mu_1$.
Then $-\infty<z_k\le x_k < +\infty$ for all $k\in\N$, i.e.,
\begin{equation}
z\preceq x,
\label{eq:z-preseq-x}
\end{equation}
 since the violation of this for some $k$ would  imply that  there is  $r$ satisfying $x_k<r<z_k$, and  
we would have $\mu_1\{y: y_k> r\}>0$ and $\mu_2\{y: y_k< r\}>0$, so
\[\mu_1\times\mu_2\{(y^1,y^2): y^1\not\preceq y^2\}\ge \mu_1\times\mu_2\{(y^1,y^2): y^1_k> r > y^2_k\}>0,\]
 contradicting  \eqref{eq:cond_on_product_measure}.

 By definitions of $x$ and $z$
there are $x',z'\in\R^\N$ such that $x\preceq x'$ and $z'\preceq z$. Combining this with~\eqref{eq:z-preseq-x}, we obtain
 that $x\in\R^\N$.

 Suppose $\mu_2(C^{\succeq x})<1$. This means that for some $k\in\N$, we have
$\mu_2\{y: y_k < x_k\}>0$, a contradiction against the definition of $x_k$. A similar argument applies to $\mu_1(C^{\preceq z})\le \mu_1(C^{\preceq x})$,  completing the proof.
\epf

\begin{Lemma}\label{lem:ordered-supports-of-sample-measures}
There is a full measure set $ \widetilde\FF\in\cF$ such that the following holds. For each $F\in\widetilde\FF $ and each $v_1,v_2\in\R$ satisfying $v_1<v_2$, suppose that $\mu_1,\mu_2$ are two probability measures on $\R^\N$ which are
\begin{itemize}
    \item concentrated on $S(v_1)$ and $S(v_2)$, respectively;
    \item invariant for the polymer dynamics Markov semigroup on $S(v_1)$ and $S(v_2)$, respectively.
\end{itemize}

Then there is a full measure set $\WW(\mu_1,\mu_2,F)\in\cW$ and an $\R^\N$-valued nonanticipating process $(x(t,F,W))_{t\in\R}$ such that, for all $t\in\R$ and all $W\in \WW(\mu_1,\mu_2,F)$, the sample measures for $\mu_1$ and $\mu_2$ satisfy
\begin{equation}
\label{eq:strong-order-on-sample-measures}
\mu_{1,W,t}(C^{\preceq x(t,F,W)})=\mu_{2,W,t}(C^{\succeq x(t,F,W)})=1.
\end{equation}
Moreover, $x(t,F,W)$ can be taken to be either the lower boundary for the support of $\mu_{2,W,t}$ or the upper boundary for the support of $\mu_{1,W,t}$.
\end{Lemma}

\bpf 
Due to Theorem~\ref{th:monotonization}, we 
can  take a set of full measure $\widetilde\FF $ such that for all  
$F\in \widetilde\FF$ and all $(x^1,x^2)\in S(v_1)\times S(v_2)$, ordering happens on a full measure set $\WW(F)\in\cW$.
By the invariance of $\mu_1$ and $\mu_2$, the sample measures in display~\eqref{eq:strong-order-on-sample-measures} are well-defined on a full measure set
$\widetilde\WW(\mu_1,\mu_2,F)$. We take $\WW(\mu_1,\mu_2,F)=\WW(F)\cap \widetilde\WW(\mu_1,\mu_2,F)$.

Let us introduce the following measure of departure from the domination property~\eqref{eq:strong-order-on-sample-measures}:
for two measures $\nu_1$ and $\nu_2$ on  $\R^\N$, we define
\[
q(\nu_1,\nu_2)=\nu_1\times\nu_2\{(x^1,x^2): x^1\not \preceq x^2 \}.
\]
In particular, if one can find sets $A_1$ and $A_2$ satisfying $A_1\preceq A_2$  (i.e., $x_1\preceq x_2$ for all $x_1\in A_1$, $x_2\in A_2$) and $\nu_1(A_1)=\nu_2(A_2)=1$, then $q(\nu_1,\nu_2)=0$.

Let us now introduce $Q(t)= \Ws[q(\mu_{1,W,t},\mu_{2,W,t})]$.
The right-hand side does not really depend on $t$ due to the stationarity, so $Q(t)=Q(0)$ for all $t$.
For all $t>0$, we have
\begin{align*}
Q(t)=& \sW[\mu_{1,W,t}\times\mu_{2,W,t}\{(x^1,x^2): x^1\not \preceq x^2 \}]
\\
=& \sW[\mu_{1,W,0}(\Phi^t_{F,W})^{-1}\times\mu_{2,W,0}(\Phi^t_{F,W})^{-1}\{(x^1,x^2): x^1\not \preceq x^2 \}]
\\
=& \sW[\mu_{1,W,0}\times\mu_{2,W,0}\{(x^1,x^2): \Phi^t_{F,W} x^1\not \preceq \Phi^t_{F,W} x^2 \}]
\\
=& 1 - \sW[\mu_{1,W,0}\times\mu_{2,W,0}\{(x^1,x^2): \Phi^t_{F,W} x^1 \preceq \Phi^t_{F,W} x^2 \}]
\\
= & 1 - \sW[\mu_{1,W,0}\times\mu_{2,W,0}\{(x^1,x^2): x^1\preceq x^2\}]
\\ &-
\sW[\mu_{1,W,0}\times\mu_{2,W,0}\{(x^1,x^2): x^1\not \preceq x^2;\ \Phi^t_{F,W} x^1 \preceq \Phi^t_{F,W} x^2 \}]
\\
= &Q(0) - \sW[\mu_{1,W,0}\times\mu_{2,W,0}\{(x^1,x^2): x^1\not \preceq x^2;\ \Phi^t_{F,W} x^1 \preceq \Phi^t_{F,W} x^2 \}],
\end{align*}
so for all $t>0$ we have 
\[\sW[\mu_{1,W,0}\times\mu_{2,W,0}\{(x^1,x^2): x^1\not \preceq x^2;\ \Phi^t_{F,W} x^1 \preceq \Phi^t_{F,W} x^2 \}]=0,\]
which, due to the choice of $F$ guaranteeing ordering, and independence of 
$\Wc_{0}$ and $\Wc_{[0,\infty)}$ implies that,
for $W\in\WW(\mu_1,\mu_2,F)$, $\mu_{1,W,0}\times\mu_{2,W,0}\{(x^1,x^2): x^1\not \preceq x^2\}=0$, i.e.,
\[
q(\mu_{1,W,0},\mu_{2,W,0}) = 0,\quad W\in\WW(\mu_1,\mu_2,F).
\]
This also implies that for all $m\in\Z$,
\[
q(\mu_{1,W,m},\mu_{2,W,m}) = 0, \quad W\in\WW(\mu_1,\mu_2,F).
\]
Due to the monotonicity of the cocycle and the skew-invariance property~\eqref{eq:sample-measures-skew-inv}, we can conclude that
\[
q(\mu_{1,W,t},\mu_{2,W,t}) = 0,\quad t\in\R,\  W\in\WW(\mu_1,\mu_2,F).
\]
Now we can apply Lemma~\ref{lem:criterion-for-dominating-supports}
to construct $x(t,F,W)$, which completes the proof of the lemma.
\epf

We will need the following lemma on shear invariance.

\begin{Lemma}\label{lemma:sample_meas_sh_inv}
For every $v_1,v_2\in\R$, the following holds:
\begin{enumerate}[1.]
\item There is a set $\FF'(v_1,v_2)\subset \FF(v_1)\cap\FF(v_2)$ such that $\Fs(\FF'(v_1,v_2))=1$ and 
for every $F\in \FF'(v_1,v_2)$,
\[
\mu_{v_2,\Shear^{v_2-v_1} F} = \mu_{v_1,F} (\Shear^{v_2-v_1})^{-1}.
\]
\item  There is a set $B'(v_1,v_2)\subset B(v_1,v_2)$ such that  $\Fs\times\Ws(B'(v_1,v_2))=1$
and for every $(F,W)\in B'(v_1,v_2)$ and all $t\in\R$,
\[
\mu_{v_2,\Shear^{v_2-v_1} F, W,t} = \mu_{v_1,F,W,t} (\Shear^{v_2-v_1})^{-1}.
\]
\end{enumerate}
\end{Lemma}
\bpf
The first part follows from  (i)~the fact that shear transformations preserve the Gibbs property of finite-volume polymer measures (this is Lemma~5.2
of \cite{Bakhtin-Li:MR3911894}) and (ii)~the uniqueness of IVPMs (our Theorem~\ref{thm:thermodynamic-limit} or Theorem 4.2 of~\cite{Bakhtin-Li:MR3911894}). The second part follows from the first one, the fact that infinite volume
polymer measures are invariant under the polymer dynamics (Proposition~\ref{prop:invariance}) and the shear-invariance property of the dynamics 
(Lemma~\ref{lem:shear-invariance-of-dynamics}).
\epf

The sets $B'(v_1,v_2)$ 
do depend on slopes
$v_1,v_2$. It is hard to make a statement valid for all $v_1,v_2$ at the same time, so we will need to use monotonicity and countable 
sets of~$v$.

Before stating a result on stationary global solutions, let us recall some terminology.
A process $(x(t))_{t\in\R}$ is called stationary if its distribution does not change under time shifts.
It is called a global solution for  $(\Phi_{F,W}^t)$ if there is a set $A$ of probability $1$ such that 
\begin{equation}
\label{eq:global-solution}
\Phi^{t,s}_{F,W} x(t,W)=x(s,W)
\end{equation}
 whenever $t<s$ and $W\in A$. A stationary nonanticipating global solution $(x(t))_{t\in\R}$ is said to be unique in  a space $L$ if for  any other stationary nonanticipating global solution $(y(t))_{t\in\R}$ with values in $L$,
there is a set $A'$ of probability $1$ such that $x(t,W)=y(t,W)$ for all $t\in\R$ and all $W\in A'$.

\begin{Th} \label{th:sample-measures-for-Gibbs-are-delta}
There is a random field $x_n(v,t)=x_n(v,t,F,W)$,  $t\in\R$, $v\in\R$, $n\in\N$, defined on $(\FF_0\times\WW_0,\Fc\times\Wc,\Fs\times \Ws)$
and, for each $v\in\R$,  a set $C(v)\in\Fc\times\Wc$ satisfying $\Fs\times\Ws(C(v))=1$
 such that the following holds:
\begin{enumerate}[1.]
\item \label{it:x-satisfies-shear-inv-lemma}
For all  $n\in\N$, $t\in\R$, the process $\{x_n(v,t)\}_{v\in\R}$ has right-continuous paths and  nonnegative stationary increments in $v$. 
\item \label{it:continuity-pt} For all $v\in\R$, $(F,W)\in C(v)$, $t\in\R,$  $n\in\N$,
$v$ is a continuity point of $x_n(v,t)$.
\item \label{it:sample-measures-are-Dirac}
Let us define $x(v,t)=x(v,t,F,W)=(x_n(v,t,F,W))_{n\in\N}$. Then, for all  $v\in\R,$ $(F,W)\in C(v)$,  $t\in \R$,  we have 
\[\mu_{v,F,W,t}=\delta_{x(v,t,F,W)}.
\]
\item\label{it:x-is-S-v-valued}
 For all $v\in\R$, $(F,W)\in C(v)$, $t\in\R,$  we have $x(v,t)\in S(v)$.
\item \label{it:global-solution}
For any fixed $v\in\R$ and $\Fs$-a.e.\ potential $F$, $(x(v,t))_{t\in\R}$ is a unique stationary nonanticipating global  solution of the polymer dynamics in $S(v)$.

\item \label{it:inv-measure} For any fixed $v\in\R$ and $\Fs$-a.e.\ potential $F$, $\mu_{v,F}$ is
a unique invariant measure for the polymer Markov semigroup $(P^t)$ on $S(v)$.
\end{enumerate} 
\end{Th}

Part~\ref{it:inv-measure} is one of the claims in part~\ref{it:uniq-erg-mix} of Theorem~\ref{th:collect}. 
Part~\ref{it:global-solution}  is one of the claims in part~\ref{it:1F1S}

\begin{Rem}
This theorem, together with the definition of the sample measures,  implies weak convergence of $\mu_{F,v} (\Phi^{s,t}_{F,W})^{-1}$ 
to $\delta_{x(v,t,F,W)}$. Since the initial condition in this statement is not arbitrary (it is distributed according to $\mu_{F,v}$) and the convergence
to $x(v,t,F,W)$ is only distributional, this is a weaker form of 1F1S. We will upgrade our theorem to the true 1F1S on each $S(v)$ in the next subsection.
\end{Rem}

\begin{Rem}\label{rem:continuity-conj}We conjecture that, with probability $1$, the right-continuous and nondecreasing in $v$ process $x(v,t,F,W)$ actually has no discontinuity 
points on the entire real line (this is a strengthening of part~\ref{it:continuity-pt} of the theorem).
\end{Rem}

\bpf
Let us fix $t$, $n$, and for any countable dense set $V\subset\R$ define 
\[
x_n(v,t,F,W,V)=\lim_{V\ni v'\downarrow v} \rsupp(\mu_{v',t,F,W}\pi^{-1}_n),\quad   v\in\R.
\]
Due to Lemma~\ref{lem:ordered-supports-of-sample-measures}, there is a set $A(V)\in \Fc\times\Wc$ of probability $1$ on which this function is well-defined, real,
nondecreasing and right-continuous. On the zero-measure complement of this set we can set, say, $x_n(v,t,F,W,V)=v$.
Due to Lemma~\ref{lem:ordered-supports-of-sample-measures},   for any two countable dense sets $V_1,V_2\subset\R$,  $x_n(v,t,F,W,V_1)=x_n(v,t,F,W,V_2)$ for $\Fs\times\Ws$-a.e.~$(F,W)$, so one can view $x_n(v,t,F,W)$ as uniquely defined
up to zero-measure modifications depending on a particular choice of $V$. Next, since $(\mu_{v,t,F,W})_{v\in\R}$ is shear-invariant
(Lemma~\ref{lemma:sample_meas_sh_inv}), we obtain that
$x_n(v,t,F,W)$ is a process with stationary increments in $v$, which proves part~\ref{it:x-satisfies-shear-inv-lemma} for a fixed $t$. It is also automatically
 valid for countably many~$t$ but we will need to extend this to uncountably many. 

Now we can  apply Lemma~\ref{lem:jumps-of-random-monotone-f} on jumps of stationary monotone processes to conclude that for each fixed $v\in\R$,
$x_n(v,t,F,W)$ is continuous at $v$ with probability $1$. Together with Lemma~\ref{lem:ordered-supports-of-sample-measures}, this implies that
 the support of $\mu_{v,F,W,t}\pi^{-1}_n$
is one point $x_n(v,t,F,W)$. The point $x(v,t,F,W)=(x_n(v,t,F,W))_{n\in\N}$ has to belong to $S(v)$ because $\mu_{v,F,W,t}$ is supported by $S(v)$.
Considering all $t\in\Z$ simultaneously, we combine countably many exceptional zero measure sets into a single one and see that parts \ref{it:x-satisfies-shear-inv-lemma}--\ref{it:x-is-S-v-valued} hold true simultaneously for all those values of~$t$. 

Since sample measures always satisfy~\eqref{eq:sample-measures-skew-inv},
 we obtain \eqref{eq:global-solution} for $t\in\Z$ and $s-t\in\N$. 
We can now use  this and \eqref{eq:sample-measures-skew-inv} to define $x(v,s,F,W)$ for noninteger $s$. This defines a global solution in agreement with parts \ref{it:x-satisfies-shear-inv-lemma}--\ref{it:x-is-S-v-valued}  of the theorem.
Part~\ref{it:sample-measures-are-Dirac} now implies that the process  $(x(v,t))_{t\in\R}$ is stationary and nonanticipating, completing the proof of existence in 
part~\ref{it:global-solution}
 of the theorem. 

Let us prove part~\ref{it:inv-measure}. 
Let $\rho$ be an invariant measure for for $(P^t)$ on $S(v)$. Then its
sample measures $\rho_{F,W,t}$ are well-defined on a set of full measure. 
Let us fix some~$t$, e.g., set $t=0$.
Due to Lemma~\ref{lem:ordered-supports-of-sample-measures} applied to $\rho$ and $\mu_{v',F}$ for $v'\in\Q$,
we also have that with probability 1, for every
$v_1,v_2\in\Q$ such that $v_1<v<v_2$, the support of
$\rho_{F,W,t}$ is squeezed between $x(v_1,t,F,W)$ and $x(v_2,t,F,W)$.
Since the probability that $v$ is a discontinuity point of the monotone function 
$v\mapsto x(v,t,F,W)$ is zero, we obtain that with probability $1$,
$\rho_{F,W,t}$ is 
uniquely defined as
the delta measure at 
\[ x(v,t,F,W)=\lim_{\Q \ni v_1\uparrow v} x(v_1,t,F,W)=\lim_{\Q \ni v_2\downarrow v} x(v_2,t,F,W).
\]
i.e., $\rho_{W,t}=\mu_{v,F,W,t}$. 
Therefore, $\rho$ is also uniquely defined and must coincide with~$\mu_{v,F}$.

To prove the remaining uniqueness  claim in part~\ref{it:global-solution} of the theorem, we assume that there is another
nonanticipating global stationary $S(v)$-valued solution $x'(v,t,F,W)$. Its distribution is invariant under $(P^t)$, so this distribution must coincide with
$\mu_{v,F}$. For the sample measures for $\mu_{v,F}$,  we have $\mu_{v,F,W,t}=\delta_{x(v,t,F,W)}$ on the one hand, and 
$\mu_{v,F,W,t}=\delta_{x'(v,t,F,W)}$ on the other hand,
which implies our claim and completes the proof.
\epf

\subsection{True 1F1S}
\label{sec:strong1F1S}

The attractor property in
part~\ref{it:1F1S} in Theorem~\ref{th:collect} is restated below in a more precise way.

\begin{Th} 
\label{th:1F1S-coord}
Let $v\in\R$. Then, for $\sF\times \sW$-a.e.\ $(F,W)$, we have

\[
\lim_{t\to-\infty} \|\Phi^{t,s}_{F,W} y - x(v,s,F,W)\|_{\LL}=0, \quad y\in S(v),\  s\in \R,
\]
where  $x(\cdot,\cdot,\cdot,\cdot)$ is the family of global solutions constructed in~Theorem~\ref{th:sample-measures-for-Gibbs-are-delta}.
\end{Th}

\begin{Rem} One can view $x(\cdot,s,F,W)$ as a global solution in the space of functions of $v$. 
Theorem~\ref{th:1F1S-coord} has an obvious extension for simultaneous convergence to $x(v,s,F,W)$ over countable sets of values $v$. 
Requiring that $y=y(v)$ is a nondecreasing function of $v$ (i.e., $y(v_1)\preceq y(v_2)$ for $v_1<v_2$),
one can make a stronger claim of convergence at every point of continuity of $x(\cdot,s,F,W)$.
If in addition  the continuity conjecture in Remark~\ref{rem:continuity-conj} holds, then the convergence will hold at all points $v\in\R$ without exception.
\end{Rem}
\bpf
Due to the skew-invariance identity \eqref{eq:global-solution} and the continuity of the flow, it suffices to consider only  $s=0$. 
Theorem~\ref{th:sample-measures-for-Gibbs-are-delta}  implies that on the set 
\begin{equation}
D(v)=\bigcap_{v'\in v+\Q} C(v')
\label{eq:D-all-Q-shifts-of-C}
\end{equation}
of probability~$1$, for all $v'\in v+\Q$, invariant measures $\mu_{v',F}$ and their sample measures $\mu_{v',F,W}$ are well-defined.

Let us fix $v'\in v+\Q$ satisfying $v'<v$ and define 
\[
L_y(v')= C^{\preceq y} \cap S(v'),\quad y\in S(v). 
\]
Let us now take any $y\in S(v)$ and denote  $a=\mu_{v',F}(L_y(v'))$. Lemma~\ref{lem:dominations-exist} implies  that $a>0$.

By the definition \eqref{eq:time-dependent-sample-measures} of sample measures, the definition \eqref{eq:D-all-Q-shifts-of-C} of $D(v)$ and  Theorem~\ref{th:sample-measures-for-Gibbs-are-delta}, we obtain that on $D(v)$, for each $\eps>0$,
there is  $t_{\eps,-}>0$ such that for all $t>t_{\eps,-}$, 
\[
\mu_{v',F} (\Phi^{-t,0}_{F,W})^{-1}  (O^c_\eps(x(v',0,F,W)))<a/2,
\]
where
\begin{equation}
\label{eq:O}
O_\eps(z)=\{z'\in\R^\N:\ z_k-k \eps< z'_k < z_k +k\eps,\ k\in\N  \}
\end{equation}
denotes the open $\eps$-neighborhood of a point $z$ in $\|\cdot\|_{\LL}$. Therefore, for each $t>t_{\eps,-}$, there is $z(t)\in  L_y(v')$ such that
$z'(t)=\Phi^{-t,0}_{F, W} (z(t)) \in O_\eps(x(v',0,F,W))$. Using $z(t)\preceq y$ and   the order-preserving property of the flow, we obtain that for all $t>t_{\eps,-}$, we have 
$z'(t)\in O_\eps(x(v',0,F,W))$ and $\Phi^{-t,0}_{F,W} y\succeq z'(t)$.  Due to~\eqref{eq:O}, this implies
\begin{equation}
\label{eq:lower-bound-on-pullback}
\Phi^{-t,0}_{F,W} y\succeq  \Shear^{-\eps}x(v',0,F,W),\quad t>t_{\eps,-}.
\end{equation}
A similar argument implies that there is $t_{\eps_+}>0$ such that
\begin{equation}
\label{eq:upper-bound-on-pullback}
\Phi^{-t,0}_{F,W} y\preceq \Shear^{+\eps}x(v',0,F,W),\quad t>t_{\eps,+}.
\end{equation}
 Since $\eps>0$ and $v'\in v+\Q$ are arbitrary, we obtain for all $k\in\N$:
\[
\liminf_{t\to\infty}(\Phi^{-t,0}_{F,W} y)_k\ge \lim_{v+\Q\ni v'\uparrow v} x_k(v',0,F,W).
\]
\[
\limsup_{t\to\infty}(\Phi^{-t,0}_{F,W} y)_k\le \lim_{v+\Q\ni v'\downarrow v} x_k(v',0,F,W).
\]
Using the a.s-continuity of $x(\cdot,0,F,W)$ at $v$ we conclude that with probability $1$, for each $k\in\N$
\begin{equation}
\label{eq:coordinatewise-pullback}
\lim_{t\to\infty}(\Phi^{-t,0}_{F,W} y)_k=  x_k(v,0,F,W).
\end{equation}

It remains to upgrade this to convergence in $\|\cdot\|_{\LL}$. Let us fix an arbitrary $\delta>0$ and set $\eps=\delta/3$ and $v'_{\pm}=v\pm\delta/3$.
Since $x(v'_\pm,0,F,W)\in S(v'_\pm)$, we can find $n\in\N$ such that for $k\ge n$,
\[
(x(v'_-,0,F,W))_k\ge k(v'_- -\eps),
\]
and
\[
(x(v'_+,0,F,W))_k\le k(v'_++\eps).
\]
So we have
\[
( \Shear^{-\eps}x(v'_-,0,F,W))_k\ge k(v'_--2\eps)\ge k(v-\delta)
\]
implying
\[
\inf_{k\ge n}\frac{
( \Shear^{-\eps}x(v'_-,0,F,W))_k- kv}{k} \ge -\delta,
\]
and, similarly,
\[
\sup_{k\ge n}\frac{
( \Shear^{-\eps}x(v'_+,0,F,W))_k- kv}{k} \le \delta.
\]
Combining these estimates with~\eqref{eq:lower-bound-on-pullback} applied to $v'=v'_-$, \eqref{eq:upper-bound-on-pullback} 
applied to $v'=v'_+$,
and 
\eqref{eq:coordinatewise-pullback} applied to $k<n$,
we obtain
\[
\limsup_{t\to\infty}\|\Phi^{-t,0}_{F,W} y - x(v,0,F,W)\|_{\LL}\le 2\delta.
\]
Since $\delta>0$ is arbitrary, this completes the proof.
\epf

\medskip

It remains to state a  corollary on mixing which, together with part~\ref{it:inv-measure} in Theorem~\ref{th:sample-measures-for-Gibbs-are-delta}, yields part~\ref{it:uniq-erg-mix} in Theorem~\ref{th:collect}.

A Markov  semigroup $(P^t)$ on a space $\XX$  is called mixing  with respect  to a $(P^t)$-invariant measure~$\mu$  if the associated transition probabilities
$P^t(x,\d y)$ converge weakly to $\mu(\d y)$ as $t\to\infty$, i.e.,
$\lim_{t\to\infty} P^t g(x) =\mu g$ for all continuous bounded functions~$g$ and all $x\in\XX$.

\begin{Cor}\label{cor:mixing} Let $v\in\R$. For $\sF$-a.e.\ $F$, the Markov semigroup $(P^t)$ is mixing with respect to $\mu_{v,F}$.
\end{Cor}
\bpf 
By Theorem~\ref{th:1F1S-coord} and the bounded convergence theorem,
for each $x\in S(v)$,
\[
P^tg(x)=\Ws[g(\Phi^t_{F,W}x )]=\Ws[g(\Phi^{-t,0}_{F,W} x)]\to  \mu_{v,F} g, 
\]
and the corollary follows.
\epf

\newcommand{\etalchar}[1]{$^{#1}$}

\bibliographystyle{alpha}
\end{document}